\documentclass[10pt]{article}
\usepackage{amssymb,amsthm,bbm}
\usepackage[centertags]{amsmath}
\usepackage{mathtools}
\usepackage[mathscr]{euscript}
\usepackage{tikz-cd}
\usepackage{hyperref}
\usepackage{xspace}
\usepackage{graphicx}
\usepackage{csquotes}
\usepackage{faktor}
\usepackage{tikz}
\usepackage[numbers, sort]{natbib}
\usetikzlibrary{shapes.symbols}

\newcommand{\pbmark}{\ar[dr, phantom, "\ulcorner" very near start, shift right=1ex]}
\newcommand{\pomark}{\ar[ul, phantom, "\ulcorner" very near start, shift right=1ex]}

\newtheorem{thm}{Theorem}[section]
\newtheorem*{theorem*}{Theorem}
\newtheorem{theorem}[thm]{Theorem}

\newtheorem{lemma}[thm]{Lemma}

\newtheorem{proposition}[thm]{Proposition}

\newtheorem{corollary}[thm]{Corollary}
\theoremstyle{definition}
\newtheorem{defn}[thm]{Definition}
\newtheorem{definition}[thm]{Definition}

\newtheorem{example}[thm]{Example}

\newtheorem{remark}[thm]{Remark}

\newcommand\cat[1]{\ensuremath{\mathsf{#1}}\xspace}
\newcommand{\Set}{\cat{Set}}
\newcommand{\Cat}{\cat{Cat}}
\newcommand{\wCat}{\cat{\omega Cat}}
\newcommand{\Comp}{\cat{Comp}}
\newcommand{\Glob}{\cat{Glob}}
\newcommand{\Bat}{\cat{Bat}}

\mathchardef\mhyphen="2D

\makeatletter
\renewcommand{\operator@font}{\sf}
\makeatother

\DeclareMathOperator{\Type}{Sphere}
\DeclareMathOperator{\Cell}{Cell}

\DeclareMathOperator{\Pos}{Pos}

\DeclareMathOperator{\ty}{bdry}
\DeclareMathOperator{\fv}{supp}

\DeclareMathOperator{\fim}{fim}
\DeclareMathOperator{\inc}{inc}
\DeclareMathOperator{\cov}{cov}
\DeclareMathOperator{\br}{br}
\DeclareMathOperator{\var}{var}
\DeclareMathOperator{\coh}{coh}
\DeclareMathOperator{\obpos}{init}
\DeclareMathOperator{\herepos}{init}
\DeclareMathOperator{\inr}{inr}
\DeclareMathOperator{\inl}{inl}
\DeclareMathOperator{\pr}{pr}
\DeclareMathOperator{\src}{src}
\DeclareMathOperator{\tgt}{tgt}
\DeclareMathOperator{\srcpos}{src-pos}
\DeclareMathOperator{\tgtpos}{tgt-pos}
\DeclareMathOperator{\postype}{pos-Sphere}
\DeclareMathOperator{\id}{id}
\DeclareMathOperator{\cdpth}{cell-depth}
\DeclareMathOperator{\mdpth}{mor-depth}
\DeclareMathOperator{\sk}{sk}
\DeclareMathOperator{\cosk}{cosk}
\DeclareMathOperator{\Sk}{Sk}
\DeclareMathOperator{\Cosk}{coSk}
\DeclareMathOperator{\Free}{Free}
\DeclareMathOperator{\comp}{comp}
\DeclareMathOperator{\dom}{dom}
\DeclareMathOperator{\cod}{cod}

\DeclareMathOperator*{\colim}{colim}
\DeclareMathOperator{\op}{{op}}
\DeclareMathOperator{\fc}{{fc}}
\DeclareMathOperator{\el}{el}
\DeclareMathOperator{\self}{self}

\DeclareMathOperator{\gen}{gen}
\DeclareMathOperator{\fr}{fr}

\newcommand{\N}{\ensuremath{\mathbb{N}}}

\makeatletter
\def\calign@preamble{%
   &\hfil\strut@
    \setboxz@h{\@lign$\m@th\displaystyle{##}$}%
    \ifmeasuring@\savefieldlength@\fi
    \set@field
    \hfil
    \tabskip\alignsep@
}
\let\cmeasure@\measure@
\patchcmd\cmeasure@{\divide\@tempcntb\tw@}{}{}{}
\patchcmd\cmeasure@{\divide\@tempcntb\tw@}{}{}{}
\patchcmd\cmeasure@{\ifodd\maxfields@
  \global\advance\maxfields@\@ne
  \fi}{}{}{}
\newenvironment{calign}
{%
  \let\align@preamble\calign@preamble
  \let\measure@\cmeasure@
  \align
}
{%
  \endalign
}
\makeatother

\newcounter{nodemaker}
\renewcommand{\-}[0]{\nobreakdash-\hspace{0pt}}

\newcounter{commcounter}
\begin{document}

\title{\bf Computads for weak \texorpdfstring{$\omega$}{ω}-categories
        \\as an inductive type}
\author{Christopher J. Dean\footnote{Dalhousie University, Canada, \texttt{christopher.dean@dal.ca}}~, Eric Finster\footnote{University of Birmingham, UK, \texttt{E.L.Finster@bham.ac.uk}}~, Ioannis Markakis\footnote{University of Cambridge, UK, \texttt{im496@cam.ac.uk}}~,\\David Reutter\footnote{University of Hamburg, Germany, \texttt{david.reutter@uni-hamburg.de}}~ and Jamie Vicary\footnote{University of Cambridge, UK, \texttt{jamie.vicary@cl.cam.ac.uk}}}
\date{}

\maketitle

\begin{abstract}
        We give a new description of computads for weak globular $\omega$\-categories by giving an explicit inductive definition of the free words. This yields a new understanding of computads, and allows a new definition of $\omega$\-category that avoids the technology of globular operads. Our framework permits direct proofs of important results via structural induction, and we use this to give new proofs that every $\omega$\-category is equivalent to a free one, and that the category of computads with generator-preserving maps is a presheaf topos, giving a direct description of the index category. We prove that our resulting definition of $\omega$\-category agrees with that of Batanin and Leinster and that the induced notion of cofibrant replacement for $\omega$-categories coincides with that of Garner.
\end{abstract}

\section{Introduction}

\paragraph{Motivation.}Given any type of algebraic structure, a basic question is to understand the data from which one can generate it freely. Just as we generate a free monoid from a set, or a free category from a directed graph, we may generate a free $\omega$\-category from a \textit{computad}, in a way which is described by an adjunction between appropriate categories:
\begin{align*}
        \begin{aligned}
                \begin{tikzpicture}
                        \node [anchor=east] (1) at (0,0) {\Set};
                        \node [anchor=west] (2) at (1,0) {\textsf{Mon}};
                        \draw [->] ([yshift=5pt] 1.east) to ([yshift=5pt] 2.west);
                        \draw [<-] ([yshift=-5pt] 1.east) to ([yshift=-5pt] 2.west);
                \end{tikzpicture}
        \end{aligned}
         &  &
        \begin{aligned}
                \begin{tikzpicture}
                        \node [anchor=east] (1) at (0,0) {\textsf{DGraph}};
                        \node [anchor=west] (2) at (1,0) {\textsf{Cat\vphantom{p}}};
                        \draw [->] ([yshift=5pt] 1.east) to ([yshift=5pt] 2.west);
                        \draw [<-] ([yshift=-5pt] 1.east) to ([yshift=-5pt] 2.west);
                \end{tikzpicture}
        \end{aligned}
         &  &
        \begin{aligned}
                \begin{tikzpicture}
                        \node [anchor=east] (1) at (0,0) {$\Comp$};
                        \node [anchor=west] (2) at (1,0) {$\omega \Cat\vphantom{p}$};
                        \draw [->] ([yshift=5pt] 1.east) to node [above] {} ([yshift=5pt] 2.west);
                        \draw [<-] ([yshift=-5pt] 1.east) to node [below] {} ([yshift=-5pt] 2.west);
                \end{tikzpicture}
        \end{aligned}
\end{align*}
Today these computads  play a substantial role in higher category theory and its applications, including in rewriting~\cite{Mimram2014}, word problems and homology~\cite{Burroni1993, makkai_word_2005, Lafont2008}, homotopy theory~\cite{garner_homomorphisms_2010} and topological quantum field theory~\cite{Bartlett2015}.

Initiated by Street~\cite{Street1976} with the study of 2-computads, the theory of $n$\-computads is now well-developed, with key contributions by Batanin~\cite{batanin_computads_1998}, Garner~\cite{garner_homomorphisms_2010} and others. Modern treatments typically proceed as follows~\cite[Section 5]{garner_homomorphisms_2010}: firstly, a category  $\omega\Cat$ of weak $\omega$-categories and strict $\omega$\-functors is defined as the category of algebras for a finitary monad on globular sets; secondly, a notion of $n$\-computad, and free $\omega$\-category over an $n$\-computad, is defined inductively on~$n$, via a colimit construction in $\wCat$. This is guaranteed to be successful thanks to general theorems about cocompleteness properties of categories of algebras of finitary monads~\cite{adamek_1980, adamek_1994}. 


This powerful method quickly establishes that computads are well-defined mathematical objects. However, the resulting definition is not fully explicit: while general theory tell us the necessary colimits must exist, their construction involves an infinite sequence of pushouts~\cite[Section~I.5]{adamek_1980}, which exhibits the set of free words in each dimension as an intricate quotient of a larger set. This technology exhibits the final set via a universal property, but does not give a straightforward description of the resulting quotient set.

In our work we show that this quotient set does in fact admit a direct description in each dimension, as a family of inductive sets. This yields an elementary and fully explicit definition of computads for weak globular $\omega$\-categories.

Since our definition of computad is concrete and elementary, it allows us to demonstrate known results about computads in new and simple ways. We show that our definition reproduces the universal cofibrant replacement of Garner~\cite{garner_homomorphisms_2010}, giving an explicit description of this construction for the case of $\omega$\-categories, and showing in particular that every $\omega$\-category is equivalent to one which is free on a computad. Furthermore, we show that the category of computads with structure-preserving maps is a presheaf topos, and give what we believe is the first explicit description of the index category. We also show that the definition of $\omega$\-category arising from our work precisely agrees with the notion of weak $\omega$\-category described by Batanin and Leinster~\cite{batanin_monoidal_1998,leinster_higher_2003}.

\paragraph{Inductive structure.} Our definition of $n$\-computads is {\textit{inductive}} in each dimension $n$, in the following sense. As the definition proceeds, all the necessary sets are described via {\textit{constructors}}, which allow us to uniquely exhibit any particular element, without requiring that we pass to a quotient. Here we make use of \textit{structural induction}, a powerful technique which enjoys wide use in theoretical computer science, and which generalizes the ordinary mathematical notion of induction over the natural numbers, allowing in cases such as ours more efficient definitions and proofs.
Introductory texts on this topic include Pierce~\cite[Section 3.3]{pierce_textbook}, Winskel~\cite[Chapter 3]{winskel_textbook} and Mitchell~\cite[Chapter 1]{mitchell_textbook}. An equivalent development can be given in terms of polynomial endofunctors~\cite{gambino_2012, moerdijk_2000}, the initial algebras of which correspond precisely to W-types in type theory.

Since the standard approach to computads requires constructing quotient sets, it follows that our definition is different from the standard one obtained via the theory of locally presentable categories, although it describes isomorphic objects. To illustrate this point, consider the situation with free monoids. For a set $X$, we can describe the free monoid on $X$ as having an underlying set given by binary trees with leaves in $X \coprod \{\star\}$, modulo the quotient given by the the associativity and unit laws. This is of course well-defined, but not purely inductive, due to the quotient. Nonetheless, the resulting quotient also admits a direct inductive description, namely, the set of words on $X$, and this inductive structure is convenient for many purposes.

The same question can be asked for many sorts of algebraic structure: does the underlying set of a free instance admit an inductive description? For 2\-categories this is much harder to achieve, perhaps impossible, due to the complex nature of the quotient induced by the interchange law. For $n$-categories for $n>2$ it  is harder still, whether strict or weak, due to the complex set of equational constraints that must be admitted, and there seems no good reason to expect a positive answer.

Our demonstration of this freeness property for $\omega$\-categories may therefore be surprising. Furthermore, since every $\omega$\-category is equivalent to a free one (a phenomenon we discuss below), something even stronger is true: the  entire theory of $\omega$\-categories and functors can be handled in terms of free algebraic structures.

That such a presentation of the theory of computads might exist was suggested by the recent development of a purely syntactic, type-theoretic presentation of the theory of $\omega$\-categories~\cite{finster2017type}.  This theory was, in turn, itself motivated by deep connections between homotopy theory, higher category theory and dependent type theory~\cite{hottbook}.  We feel the present work is a nice example of how ideas from type theory can lead to insight even in classical mathematics.

\paragraph{Validity and formalisation.} Of course, it must be clear that our definition has an unambiguous interpretation in terms of sets and functions, which we call its \textit{semantics}, and is free from pathologies such as non-wellfoundedness. Standard techniques exist to establish wellfoundedness of inductive definitions, but they cannot be applied directly to our definition, which has a complex structure; for example, it uses large induction-recursion in the parameters. We have formalised our definition in the proof assistant Agda~\cite{markakis_computads_2023}, which accepts it as correct. However, this is not in itself a proof of validity, as is known that Agda will accept some definitions whose semantics is not fully understood~\cite{stackexchange-meven}. We tackle this problem by constructing a suitable polynomial endofunctor, which we show has an initial algebra that yields a semantics for our definition. 

\paragraph{Computad-first.} Our approach yields a category of computads $\Comp$ as the primary object of study, which we construct directly, without passing via a pre-existing definition of $\omega$\-category. We then exhibit an adjunction between $\Comp$ and the category $\Glob$ of globular sets, yielding a monad on $\Glob$, and we define $\wCat$ as the category of algebras for this monad. This gives a new definition of $\wCat$ that avoids the standard technology of globular operads and contractions. The following diagram illustrates our approach:
\[\begin{tikzcd}
        \Comp && \wCat \\
        & \Glob
        \arrow["{K^w}", from=1-1, to=1-3]
        \arrow[""{name=0, anchor=center, inner sep=0}, "\Free", shift left=2, from=2-2, to=1-1]
        \arrow[""{name=1, anchor=center, inner sep=0}, "\Cell", shift left=2, from=1-1, to=2-2]
        \arrow[""{name=2, anchor=center, inner sep=0}, "{F^w}"', shift right=2, from=2-2, to=1-3]
        \arrow[""{name=3, anchor=center, inner sep=0}, "{U^w}"', shift right=2, from=1-3, to=2-2]
        \arrow["\dashv"{anchor=center, rotate=50}, draw=none, from=0, to=1]
        \arrow["\dashv"{anchor=center, rotate=131}, draw=none, from=2, to=3]
\end{tikzcd}\]
This contrasts with other approaches to the theory of computads as discussed above, which begin with a definition of \wCat, and use it as a building block in the definition of computad. Our work therefore presents a shift in perspective, which generates many of the benefits of our approach. Indeed, we propose that computads, as syntactic objects, are fundamentally simpler structures than $\omega$\-categories, and that our approach brings out this simplicity.

\paragraph{Computadic replacement.} In a locally finitely presentable category, given a set of morphisms $I$, usually referred to as  \emph{generating cofibrations}, we can apply Quillen's small object argument to obtain a weak factorisation system $(\mathcal L, \mathcal R)$, where $\mathcal R := I^\pitchfork$ and $\mathcal L := {}^\pitchfork \mathcal R$.\footnote{For a set of morphisms $S$, we write $S^\pitchfork$ and ${}^\pitchfork \! S$ for the sets of morphisms with the right  and left lifting property with respect to $S$, respectively.} We say an object $X$ is \emph{cofibrant} when the initial morphism $0 \to X$ is in $\mathcal L$, and furthermore that $X$ is a \emph{cofibrant replacement} for $Y$ when there exists an $\mathcal R$-morphism $X \to Y$. It was shown by Garner~\cite[Section 2]{garner_homomorphisms_2010} that  one can then obtain a \textit{universal cofibrant replacement comonad} $(Q,\Delta,r)$. This monad generates the cofibrant replacement structure, in the following sense: for any object $Y$,  the object $QY$ is cofibrant, and the counit $r_Y :\ QY \to Y$ is the necessary $\mathcal R$-map exhibiting $QY$ as a cofibrant replacement for $Y$. He shows that this comonad is universally determined by a certain lifting property.

For the globular $\omega$\-categories that we study here, we take our generating cofibrations $I$ to be  the inclusions $\iota_n:\mathbb S^{n-1} \to \mathbb D ^n$ of spheres as the boundaries of disks. Garner's lifting condition then takes the following form:
\[
        \label{garner_lifting}
        \begin{tikzcd}
                {\mathbb{S}^{n-1}} & QX \\
                {\mathbb{D}^n} & X
                \arrow[from=2-1, to=2-2]
                \arrow[from=1-1, to=1-2]
                \arrow["{r_X}", from=1-2, to=2-2]
                \arrow["{\iota_n}"', from=1-1, to=2-1]
                \arrow[dotted, from=2-1, to=1-2]
        \end{tikzcd}
\]
We give a direct, elementary proof of this lifting property by structural induction on our definition of computad. It follows that our notion of computad generates the universal cofibrant replacement of Garner.

To justify our choice of generating cofibrations, let us recall that a basic recursive definition of categorical equivalence may begin as as follows~\cite[Section~3]{baez_cohomology}: a 0\-functor is a 0\-equivalence when it is a bijection of sets, and an $(n+1)$\-functor is an $(n+1)$\-equivalence when it is essentially surjective on objects, and an $n$\-equivalence on hom\-$n$\-categories. This definition is not complete, since the ``essentially surjective'' property requires an understanding of which 1\-morphisms are  ``invertible''. However, \textit{identity} 1\-morphisms should certainly count as invertible, and so this leads us to the strictly stronger notion of \emph{local equivalence}: a 0\-functor is a local 0\-equivalence when it is a bijection, and an $(n+1)$\-functor is a local $(n+1)$\-equivalence when it is surjective on objects, and a local $n$\-equivalence on hom\-$n$\-categories. A striking property of this definition is that we check surjectivity at every level, but injectivity is only relevant at level $n$. For the limiting case of $n=\omega$, a local $\omega$\-equivalence then simply involves verifying a surjection at every level, and so we may also call it a \emph{local $\omega$\-surjection}.\footnote{Note that an injectivity condition is never needed. A similar phenomenon appears in Homotopy Type Theory, where the property of a map $f$ being an embedding can be expressed by requiring that the induced map $\operatorname{ap} f$ on paths admit a section.} It is now a simple matter to check that the maps which have the right lifting property with respect to the inclusions $\iota_n:\mathbb S^{n-1} \to \mathbb D ^n$ are exactly the local $\omega$-surjections.  Since we prove the lifting property, it follows that every $\omega$\-category $X$ is indeed equivalent to $QX$, the free $\omega$\-category on the underlying $\omega$\-computad of $X$.

\paragraph{Presheaf structure.} The study of presheaf structure on computads has a rich history. A 1-computad is a directed graph, and so the category of 1-computads is certainly a presheaf category. The category of 2-computads is also a presheaf category, as shown by Carboni and Johnstone~\cite{carboni_connected_1995}, and given a modern presentation by Leinster~\cite[Section 7.6]{leinster_higher_2003}; this holds in both the strict and weak forms. However, strict 3-computads do not form a presheaf category, as first shown by Makkai and Zawadowski~\cite{makkai_category_2008}, and further elucidated by Cheng~\cite{cheng_direct_2012} and Leinster~\cite[Section~7.6]{leinster_higher_2003}. On the other hand, weak $n$\-computads do form a presheaf category, as shown by Batanin~\cite{batanin_computads_2002}, in the general setting of algebras for globular operads satisfying a normalization condition.

We use our definition to give an explicit proof of the presheaf property for the category of weak $n$\-computads, including a direct construction of the index category. We believe this is the first time the index category has been explicitly presented. Roughly speaking, a representable has a unique top-level  ``dominating generator with generic type''. Each cell in the terminal computad induces a canonical diagram of representables in lower levels. The dominating generators appearing in this diagram induce a canonical generic type in its colimit.

In a different direction, Henry~\cite{henry_non-unital_2018} identifies subcategories of strict $n$\-computads that are presheaf categories. He then gives for such subcategories a bijection between generators of the terminal computad and \emph{plexes}, the objects of the index category. These plexes are strict computads satisfying a number of properties~\cite[Proposition 2.2.3]{henry_non-unital_2018}, and are analogous to the representable computads that we study.

\paragraph{Comparison with other definitions.} The work of Grothendieck, Maltsiniotis, Batanin and Leinster has led to a notion of weak $\omega$\-category, as an algebra for the initial contractible globular operad. By work of Berger \cite{berger_cellular_2002}, this operad can be seen as a \emph{globular theory}; that is a category whose objects are Batanin trees (globular pasting diagrams) where morphisms admit a certain factorisation. We give an explicit description of these factorisations for morphisms out of a computad generated by a Batanin tree. In this way we obtain a homogeneous globular theory. Furthermore, we show that this globular theory satisfies the universal property of the globular theory corresponding to the initial contractible globular operad. It follows that the weak $\omega$-categories corresponding to our computads are exactly Batanin-Leinster weak $\omega$-categories.

\paragraph{Acknowledgements.}
The work of CD on this project was partially supported by an LMS Early Career Fellowship funded by the London Mathematical Society with support from the Heilbronn Institute for Mathematical Research and UKRI. IM is supported by the Onassis Foundation – Scholarship ID: F~ZQ~039-1/2020-2021. DR is funded by the Deutsche Forschungsgemeinschaft (DFG, German Research Foundation) – 493608176 and is grateful for the hospitality and financial support of the Max-Planck Institute for Mathematics in Bonn where some of this work was carried out. JV gratefully acknowledges funding from the Royal Society.

\section{Batanin trees}\label{sec:Batanin}

Our definition of $\omega$-categorical computad will be based on a distinguished collection of such objects parameterized
by rooted planar trees.  We will adopt the name \emph{Batanin trees}
here to emphasize their interpretation as parameterizing globular pasting
diagrams. As we will be working extensively with inductively
generated sets in what follows, we give an inductive characterization
of such trees as well as their \emph{positions}.

\begin{defn}
        The set $\Bat$ of Batanin trees is defined inductively by
        \begin{itemize}
                \item For every list $[B_1,\dots,B_n]$ of Batanin trees,
                      a tree $\br \, [B_1,\dots,B_n]$.
        \end{itemize}
\end{defn}

\noindent Equivalently, the set of Batanin trees is an initial algebra for the polynomial endofunctor $F : \Set \to \Set$ taking a set to the set
\[
        F(X) = \coprod_{n\in \N} X^n
\]
of lists with elements from it. It can be computed as the increasing union of the sequence of sets defined by
\begin{align*}
        \Bat_{0}   & = \{\br\,[ \, ]\}                                         \\
        \Bat_{k+1} & = \{ \br\,[B_1,\dots,B_n]\, : n\in \N,\, B_i\in \Bat_k \}
\end{align*}
From the description of the list endofunctor $F$ as a coproduct of powers, it follows that the set of Batanin trees is the initial set with an $n$\-ary operation for every $n\in \N$. Moreover, Batanin trees are the cells of free strict $\omega$\-category on the terminal globular set, as explained by Leinster~\cite{leinster_higher_2003}.

Using the constructor $\br$ repeatedly, we may construct a variety of trees. Any such tree can be represented graphically as a rooted, planar tree, where $\br[B_1,\dots, B_n]$ has a new root with branches given by the trees $B_i$. Some examples can be given as follows:
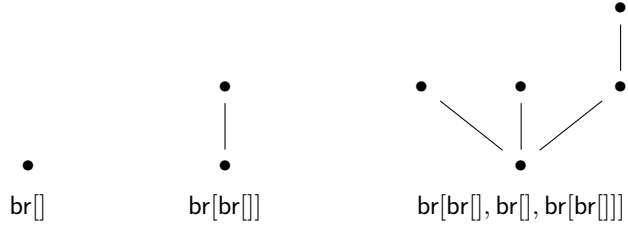
\begin{calign}
\nonumber
\begin{tikzcd}
  \phantom{\bullet} \\
  \phantom{\bullet} \\
  \bullet
\end{tikzcd}
&
\begin{tikzcd}
  \phantom{\bullet} \\
  \bullet \\
  \bullet
  \arrow[from =3-1,to =2-1, no head]
\end{tikzcd}
&
\begin{tikzcd}[ampersand replacement=\&]
  \&\&\bullet \\
  \bullet \& \bullet \& \bullet \\
  \& \bullet
  \arrow[from=3-2, to=2-1, no head]
  \arrow[from=3-2, to=2-2, no head]
  \arrow[from=3-2, to=2-3, no head]
  \arrow[from=2-3, to=1-3, no head]
\end{tikzcd}
\\ \nonumber
\br[] & \br[\br[]] & \br[\br[],\br[],\br[\br[]]]
\end{calign}

\noindent We define the \emph{dimension} of a Batanin tree $B$ to be the least $k\in \N$ such that $B\in \Bat_k$. The dimension function $\dim : \Bat \to \N$ is defined recursively by
\begin{align*}
        \dim \, (\br \, [ \, ])            & = 0                                                         \\
        \dim \, (\br \, [B,B_1,\dots,B_n]) & = \max \, (\dim \, B + 1, \dim \, (\br \, [B_1,\dots,B_n]))
\end{align*}
As can be easily seen by induction, the dimension of a tree coincides with the length of its longest path to the root, also known as the \textit{height} of the tree.

Recall that a globular set $X$ consists of a set $X_n$ of \emph{$n$-cell}s for every natural number $n$, together with source and target maps
\[\src, \tgt : X_n \to X_{n-1}\]
satisfying the \emph{globularity conditions}
\begin{align*}
        \src \, \circ \, \src & = \src \, \circ \, \tgt, &
        \tgt \, \circ \, \src & = \tgt \, \circ \, \tgt.
\end{align*}
They form a category $\Glob$ where morphisms $f : X\to Y$ are sequences of functions $f_n : X_n\to Y_n$ commuting with the source and target maps. Equivalently, $\Glob$ is the category of set-valued presheaves on the category $\mathbb{G}$ with objects natural numbers, written $[n]$, and morphisms generated by the primitive source and target inclusions
\[s, t : [n] \to [n+1] \]
for all $n$, under the duals of the globularity conditions.

Batanin trees parameterise a family of globular sets, known as \emph{pasting diagrams}, which play an important role in the theory of $\omega$\-categories; see Leinster for a detailed description of pasting diagrams and their relation to Batanin trees~\cite[Section~8.1]{leinster_higher_2003}. We will call the
pasting diagram parameterised by a tree, its \emph{globular set of positions}
$\Pos(B)$. It is formally defined as follows:

\begin{defn}\label{def:positions}
        For a Batanin tree $B$, we define the set $\Pos_n(B)$ of
        \emph{positions of $B$ of dimension $n$} inductively as follows:
        \begin{itemize}
                \item For every Batanin tree $B$, a position $\herepos(B)\in \Pos_0(B)$.
                \item For every $n \in \mathbb{N}$ and every
                      triple $(B,L,p)$ where
                      \begin{itemize}
                              \item $B$ is a Batanin tree,
                              \item $L = [B_1,\dots,B_n]$ is a list of Batanin trees,
                              \item $p \in \Pos_n(\br \, L)$ is an $n$-position of $\br \, L$,
                      \end{itemize}
                      a position $\inr \, (B,L,p) \in \Pos_{n}(\br \, [B,B_1,\dots,B_n])$.
                \item For every $n \in \mathbb{N}$ and every
                      triple $(B,L,q)$ where
                      \begin{itemize}
                              \item $B$ is a Batanin tree,
                              \item $L = [B_1,\dots,B_n]$ is a list of Batanin trees,
                              \item $q \in \Pos_n(B)$ is an $n$-position of $B$,
                      \end{itemize}
                      a position $\inl \, (B,L,q) \in \Pos_{n+1}(\br \, [B,B_1,\dots,B_n])$.
        \end{itemize}
\end{defn}

\noindent Every positive dimensional position may be assigned a source and a target, representing the globular structure of the pasting diagram represented by a Batanin tree. For a given Batanin tree $B$, the source and target maps
\begin{align*}
        \srcpos & : \Pos_{n+1}(B) \to \Pos_n(B) \\
        \tgtpos & : \Pos_{n+1}(B) \to \Pos_n(B)
\end{align*}
are defined inductively by the following equations:
\begin{align*}
        \srcpos \, (\inr \, (B,L,q)) & = \inr \, (\srcpos \, q) \\
        \srcpos \, (\inl \, (B,L,p)) & =
        \begin{cases}
                \herepos               & \text{if }n=0,     \\
                \inl \, (\srcpos \, p) & \text{if }n \ge 1,
        \end{cases}             \\
        \tgtpos \, (\inr \, (B,L,q)) & = \inr \, (\tgtpos \, q) \\
        \tgtpos \, (\inl \, (B,L,p)) & =
        \begin{cases}
                \inr \, (\herepos)  & \text{if }n=0      \\
                \inl \, (\tgtpos p) & \text{if } n \ge 1 \\
        \end{cases}
\end{align*}
It is not difficult to show by induction that the globularity conditions hold,
so the positions of $B$ form a globular set $\Pos\,(B)$.

\begin{example}\label{example:globes-trees}
        The \emph{globes} are the Batanin trees defined inductively by
        \begin{align*}
                D_0 & = \br\, [ \, ], & D_{n+1} & = \br\, [D_n].
        \end{align*}
        It is easy to see by induction that the globular set of positions of $D_n$ is isomorphic to the representable globular set $\mathbb{G}(-,[n])$.

\end{example}

\begin{example}\label{example:one-dimensional-trees}
        Batanin trees of dimension at most one are sequences of the form
        $B_n = \br\,[D_0,\overset{n}{\dots},D_0]$ for some natural number
        $n\in \N$. The globular set of positions of $B_n$ is a sequence of $n$
        composable arrows.
        
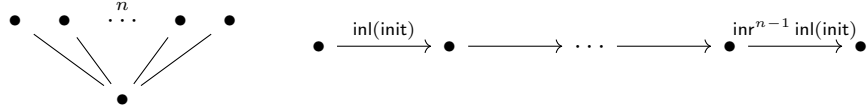
\[
          \begin{tikzcd}[column sep = 5]
            \bullet & \bullet & \overset{n}{\cdots} & \bullet & \bullet \\
            && \bullet
            \arrow[no head, from = 2-3, to = 1-1]
            \arrow[no head, from = 2-3, to = 1-2]
            \arrow[no head, from = 2-3, to = 1-4]
            \arrow[no head, from = 2-3, to = 1-5]
          \end{tikzcd}\qquad
          \begin{tikzcd}[column sep = large]
            \bullet & \bullet & \cdots & \bullet & \bullet
            \arrow["{\inl(\herepos)}", from=1-1, to=1-2]
            \arrow[from=1-2, to=1-3]
            \arrow[from=1-3, to=1-4]
            \arrow["{\inr^{n-1}\inl(\herepos)}", from=1-4, to=1-5]
        \end{tikzcd}\]
\end{example}

\begin{example}\label{ex:tree-with-positions}
        Consider the tree $B = \br\,[ \br\,[D_0, D_0], \, D_0]$. It has $9$
        positions, which we can draw on the tree as follows:
        
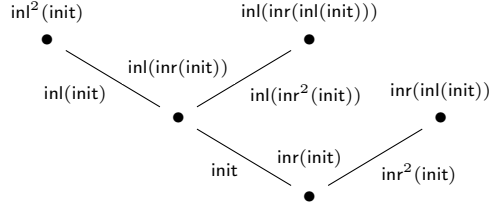
\vspace{-15pt}
        \[\begin{tikzcd}[column sep = large]
          {\phantom{}} && {\phantom{}} \\
          \bullet & {\phantom{}} & \bullet & {\phantom{}} \\
          & \bullet & {\phantom{}} & \bullet \\
          && \bullet
          \arrow["{\inr^2(\herepos)}"', no head, from=4-3, to=3-4]
          \arrow["{\herepos}", no head, from=4-3, to=3-2]
          \arrow["{\inr(\herepos)}"{description}, draw=none, from=4-3, to=3-3]
          \arrow["{\inr(\inl(\herepos))}"{description, near start}, draw=none, from=3-4, to=2-4]
          \arrow["{\inl(\inr^2(\herepos))}"', no head, from=3-2, to=2-3]
          \arrow["{\inl(\herepos)}",  no head, from=3-2, to=2-1]
          \arrow["{\inl(\inr(\herepos))}"{description}, draw=none, from=3-2, to=2-2, yshift=4pt, inner sep=0pt]
          \arrow["{\inl^2(\herepos)}"{description, near start}, draw=none, from=2-1, to=1-1]
          \arrow["{\inl(\inr(\inl(\herepos)))}"{description, near start}, draw=none, from=2-3, to=1-3]
        \end{tikzcd}\]
        Here the positions of the form $\inl\ p$ are on the first branch of the tree, while the positions of the form $\inr\,q$ are on the remaining branches. The position $\herepos(B)$ is the one connecting the root to the first branch. The source and target of each position are given by the positions adjacent to the edge below it, so that in particular the globular set $\Pos\,(B)$ is the one pictured below.
        \begin{calign}
                \nonumber
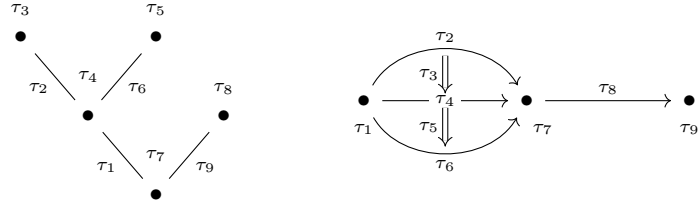
                \begin{aligned}
                \begin{tikzpicture}
                \path [use as bounding box] (-1.5,-1.5) rectangle (1.5,1);
                \node at (0,0) {$\begin{tikzcd}[column sep = small, ampersand replacement=\&]
                  {\phantom{}} \&\& {\phantom{}} \\
                  \bullet \& {\phantom{}} \& \bullet \& {\phantom{}} \\
                  \& \bullet \& {\phantom{}} \& \bullet \\
                  \&\& \bullet
                  \arrow["{\tau_9}"', no head, from=4-3, to=3-4]
                  \arrow["{\tau_1}", no head, from=4-3, to=3-2]
                  \arrow["{\tau_7}"{description}, draw=none, from=4-3, to=3-3]
                  \arrow["{\tau_8}"{description, near start}, draw=none, from=3-4, to=2-4]
                  \arrow["{\tau_6}"', no head, from=3-2, to=2-3]
                  \arrow["{\tau_2}",  no head, from=3-2, to=2-1]
                  \arrow["{\tau_4}"{description}, draw=none, from=3-2, to=2-2]
                  \arrow["{\tau_3}"{description, near start}, draw=none, from=2-1, to=1-1]
                  \arrow["{\tau_5}"{description, near start}, draw=none, from=2-3, to=1-3]
                \end{tikzcd}$};
                \end{tikzpicture}
                \end{aligned}
                &
                \begin{aligned}
                \begin{tikzpicture}
                \path [use as bounding box] (-2.5,-1) rectangle (2.5,1);
                \node at (0,0) {$\begin{tikzcd}[ampersand replacement=\&]
                  {} \\
                  \bullet \&\& \bullet \&\& \bullet \\
                  {\phantom{}} \&\& {\phantom{}} \&\& {\phantom{}}
                  \arrow["{\tau_1}"{description, near start}, draw=none, from=2-1, to=3-1]
                  \arrow["{\tau_7}"{near start}, draw=none, from=2-3, to=3-3]
                  \arrow["{\tau_9}"{description, near start}, draw=none, from=2-5, to=3-5]
                  \arrow[""{name=0, anchor=center, inner sep=0}, "{\tau_2}", bend left = 60, from=2-1, to=2-3]
                  \arrow[""{name=1, anchor=center, inner sep=0}, "{\tau_6}"', bend right = 60, from=2-1, to=2-3]
                  \arrow[""{name=2, anchor=center, inner sep=0}, "{\tau_4}"{description}, from=2-1, to=2-3]
                  \arrow["{\tau_8}", from=2-3, to=2-5]
                  \arrow["{\tau_5}"', shorten <=2pt, shorten >=2pt, Rightarrow, from=2, to=1]
                  \arrow["{\tau_3}"', shorten <=2pt, shorten >=2pt, Rightarrow, from=0, to=2]
                \end{tikzcd}$};
                \end{tikzpicture}
                \end{aligned}
                \end{calign}
        The process of picturing cells of a pasting diagram as sectors of the corresponding tree can be found in detail in the work of Berger~\cite[Example~1.4]{berger_cellular_2002}.
\end{example}

\begin{remark}\label{rem:trees-are-sums}
        Batanin trees are in bijection with the zigzag sequences described by Weber \cite[Section~4]{weber_pra_2004}, hence also to the trees defined by Batanin \cite{batanin_monoidal_1998}. The sequence corresponding to $\br\,[B_1,\dots,B_n]$ is obtained from the sequences corresponding to the trees $B_i$ by raising their entries by $1$ and concatenating them with a $0$ in between. One can easily show by induction that the globular set of positions of a tree $B$ realizes the globular sum of the corresponding sequence.
\end{remark}

The \emph{boundary} $\partial_k B$ of a Batanin tree $B$ at some $k\in \N$ is the tree obtained by removing all nodes of height more than $k$. The positions of the boundary may be included in those of $B$ in two ways, identifying them with the positions in the source or target of the pasting diagram $\Pos\,(B)$. We now proceed to define the boundaries and the \emph{source} and \emph{target inclusions} more formally.


\begin{definition}
        For $k\in \N$, the \emph{$k$-boundary} of a Batanin tree $B$, is the Batanin tree $\partial_k B$ defined inductively by
        \begin{align*}
                \partial_0 \br\,[B_1, \dots, B_n]     & =\br\,[ \, ]                                 \\
                \partial_{k+1} \br\,[B_1, \dots, B_n] & = \br\,[\partial_k B_1,\dots,\partial_k B_n]
        \end{align*}
\end{definition}

\noindent We define the source and target inclusions
\[s_k^B, t_k^B : \Pos\,(\partial_k B) \to \Pos\,(B)\]
recursively by
\[ s_k^{\br\,[ \, ]} = t_k^{\br\,[ \, ]} = \id,\]
and for every tree $B$ and list of trees $L = [B_1,\dots,B_n]$,
\begin{align*}
        s_0^{\br\,[B,B_1,\dots,B_n]}(\herepos)     & = \herepos,                    & t_0^{\br\,[B,B_1,\dots,B_n]}(\herepos)     & = \inr\,(t_0^{\br\,L}(\herepos)), \\
        s_{k+1}^{\br\,[B,B_1,\dots,B_n]}(\herepos) & = \herepos,                    & t_{k+1}^{\br\,[B,B_1,\dots,B_n]}(\herepos) & = \herepos,                       \\
        s_{k+1}^{\br\,[B,B_1,\dots,B_n]}(\inl\,p)  & = \inl\,(s_{k}^{B}(p)),        & t_{k+1}^{\br\,[B,B_1,\dots,B_n]}(\inl\,p)  & = \inl\,(t_{k}^{B}(p)),           \\
        s_{k+1}^{\br\,[B,B_1,\dots,B_n]}(\inr\,p)  & = \inr\,(s_{k+1}^{\br\,L}(p)), & t_{k+1}^{\br\,[B,B_1,\dots,B_n]}(\inr\,p)  & = \inr\,(t_{k+1}^{\br\,L}(p)).
\end{align*}
We illustrate these inclusions as follows, using again the same tree from Example~\ref{ex:tree-with-positions}. The image of each variable under the inclusions is indicated by the choice of variable name.
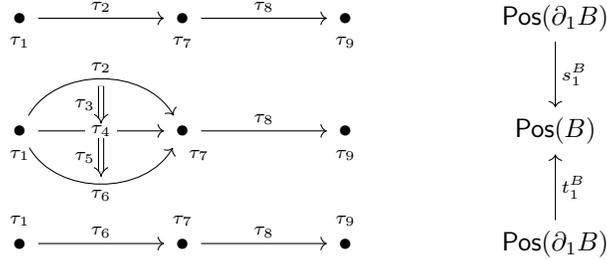
\[\begin{tikzcd}[row sep=small]
        \bullet && \bullet && \bullet && \Pos(\partial_1 B)\\
        {\phantom{}} && {\phantom{}} && {\phantom{}} \\
        \bullet && \bullet && \bullet && \Pos(B)\\
        {\phantom{}} && {\phantom{}} && {\phantom{}} \\
        \bullet && \bullet && \bullet && \Pos(\partial_1 B)
        \arrow["{s_1^B}", from = 1-7, to = 3-7]
        \arrow["{t_1^B}"', from = 5-7, to = 3-7]
        \arrow["{\tau_1}"{description, near start}, draw=none, from=3-1, to=4-1]
        \arrow["{\tau_7}"{near start}, draw=none, from=3-3, to=4-3]
        \arrow["{\tau_9}"{description, near start}, draw=none, from=3-5, to=4-5]
        \arrow[""{name=0, anchor=center, inner sep=0}, "{\tau_2}", bend left = 60, from=3-1, to=3-3]
        \arrow[""{name=1, anchor=center, inner sep=0}, "{\tau_6}"', bend right = 60, from=3-1, to=3-3]
        \arrow[""{name=2, anchor=center, inner sep=0}, "{\tau_4}"{description}, from=3-1, to=3-3]
        \arrow["{\tau_8}", from=3-3, to=3-5]
        \arrow["{\tau_5}"', shorten <=2pt, shorten >=2pt, Rightarrow, from=2, to=1]
        \arrow["{\tau_3}"', shorten <=2pt, shorten >=2pt, Rightarrow, from=0, to=2]
        \arrow["{\tau_1}"{description, near start}, draw=none, from=1-1, to=2-1]
        \arrow["{\tau_7}"{near start, description}, draw=none, from=1-3, to=2-3]
        \arrow["{\tau_9}"{description, near start}, draw=none, from=1-5, to=2-5]
        \arrow["{\tau_8}", from=1-3, to=1-5]
        \arrow["{\tau_2}", from=1-1, to=1-3]
        \arrow["{\tau_1}"{description, near start}, draw=none, from=5-1, to=4-1]
        \arrow["{\tau_7}"{near start, description}, draw=none, from=5-3, to=4-3]
        \arrow["{\tau_9}"{description, near start}, draw=none, from=5-5, to=4-5]
        \arrow["{\tau_8}", from=5-3, to=5-5]
        \arrow["{\tau_6}", from=5-1, to=5-3]
      \end{tikzcd}\]

\noindent It is easy to show by induction that $s_k^B$ and $t_k^B$ are morphisms of globular sets and that they coincide with the cosource and cotarget inclusions described by Weber~\cite[Section~4]{weber_pra_2004}. In particular, the following properties hold, which may also be verified directly by induction.


\begin{proposition}\label{prop:src-tgt-equations}
        For every Batanin tree $B$ and $k, l \in \N$,
        \begin{enumerate}
                \item $\dim\, (\partial_k B) = \min\,(k, \dim\, B)$
                \item If $\dim\, B \le k$, then $\partial_k B = B$ and $s_k^B = t_k^B = \id$.

                \item If $l < k$, then $\partial_l \partial_k B = \partial_lB$ and
                      \begin{align*}
                              s_k^B \circ s_l^{\partial_k B} & =
                              t_k^B \circ s_l^{\partial_k B} = s_l^B \\
                              s_k^B \circ t_l^{\partial_k B} & =
                              t_k^B \circ t_l^{\partial_k B} = t_l^B
                      \end{align*}
                \item The morphisms $s_k^B$ and $t_k^B$ are equal and bijective on positions of dimension less than $k$. Moreover, they are injective on positions of dimension~$k$.
        \end{enumerate}
\end{proposition}

\noindent The \emph{strict $\omega$-category monad} $(\fc^s,\eta^s,\mu^s)$ on the category of globular sets is defined in terms of Batanin trees. The endofunctor $\fc^s$ is the familially represented endofunctor sending a globular set $X$ and $n\in \N$ to
\[ \fc^s_nX = \coprod_{\dim\,B \le n} \Glob(\Pos\,(B), X).\]
The source and target maps
\[ \src,\tgt : \fc^s_{n+1} X \to \fc^s_{n} X \]
are given by precomposition with the source and target inclusions $s_n^B$ and ${t_n^B : \Pos(\partial_k B) \to \Pos(B)}$ respectively. The unit $\eta^s : \id \Rightarrow \fc^s$ sends an $n$-cell to the morphism out of $\Pos\,(D_n) \cong \mathbb{G}(-,[n])$ to which it corresponds under the Yoneda embedding. The monad multiplication uses that certain colimits of pasting diagrams are again pasting diagrams~\cite[Proposition~4.7]{weber_pra_2004}.

Before finishing this section, we introduce two families of positions of a Batanin tree, the source and target boundary positions. These are the positions which are not the target, or respectively the source, or another position. To illustrate this, recall the tree $B$ from Example~\ref{ex:tree-with-positions}. In this example, the only source 0-position is $\tau_1$, the source 1-positions are $\tau_2$ and $\tau_8$, and the source 2\-positions are $\tau_3$ and $\tau_5$. Similarly, the only target 0-position is $\tau_9$, the target 1-positions are $\tau_6$ and $\tau_8$, and the target 2-positions are again $\tau_3$ and $\tau_5$.

\begin{definition}
        Let $B$ be a Batanin tree and $p \in \Pos_n(B)$ a position of dimension $n$.  We say that $p$ is a \emph{source boundary position} if
        there does not exist any $q \in \Pos_{n+1}(B)$ such that $\tgtpos \, q = p$.  Conversely, we say that $p$ is a \emph{target boundary position} of $B$ if there does not exist any $q \in \Pos_{n+1}(B)$ such that $\srcpos \, q = p$.  We write $\partial^s_n(B)$ and $\partial^t_n(B)$ for the set of $n$-dimensional source and target positions respectively.
\end{definition}

It is easy to give an inductive characterization of the source and target boundary positions, similar to the one for all positions. The unique position of $\br\,[\,]$ is both source and target boundary, while for every Batanin tree $B$ and list $L = \,[B_1,\dots, B_n]$ of trees
\begin{itemize}
        \item $\herepos(B, L)$ is source boundary.
        \item if $p\in \Pos_n(B)$ is source boundary or target boundary, so is $\inl\,(B,L,p)$.
        \item if $q\in \Pos_n(B)$ is target boundary, so is $\inr\,(B,L,q)$.
        \item if $q\in \Pos_n(B)$ is source boundary and $L\not=[ \, ]$, so is $\inr\,(B,L,q)$.
\end{itemize}
The converses of the last three clauses also hold, so this characterizes all source and target boundary positions of $\br\,[B,B_1,\dots,B_n]$. In particular, we see that every Batanin tree has a unique source and target boundary position of dimension $0$.

We can understand the $k$-source of a Batanin tree $B$\ as the image\footnote{This is computed as the set-theoretic image in each dimension.} of the morphism of globular sets $s_k^B : \Pos(\partial_k B) \to \Pos (B)$. As a globular subset of $\Pos(B)$, this is generated by the source $k$\-positions of $B$. By generation, we mean the smallest globular subset of $B$\ that contains the given positions. We establish this property of source boundaries, and similarly for target boundaries, via the following proposition.


\begin{proposition}\label{prop:boundary-positions-generation}
        Let $B$ a Batanin tree and $k\in \N$. The image of \[s_k^B : \Pos(\partial_k B) \to \Pos (B)\] is the globular subset of $\Pos\,(B)$ generated by the
        source boundary positions of dimension at most $k$. The image of \[t_k^B : \Pos(\partial_k B) \to \Pos (B)\] is the globular subset of $\Pos\,(B)$ generated by the target boundary positions of dimension at most $k$.
\end{proposition}
\begin{proof}
        Let $S_{kl}^B$ and $T_{kl}^B$ the sets of positions of dimension $l$ in the globular subset of $\Pos\,(B)$ generated by the source and target boundary positions of dimension at most $k$ respectively. Since $\mathbb{G}$ is a direct category, both sets are empty when $l>k$ and they satisfy for $l \le k$ that
        \[ S_{kl}^B = (\partial_l^s B) \cup \left(\bigcup\nolimits_{p\in S_{k,l+1}^B} \{\srcpos\,p,\tgtpos\,p\}  \right), \]
        \[ T_{kl}^B = (\partial_l^s B) \cup \left(\bigcup\nolimits_{p\in S_{k,l+1}^B} \{\srcpos\,p,\tgtpos\,p\}  \right). \]
        Since $s_k^B$ and $t_k^B$ are bijective on positions of dimension less than $k$, we need to show that
        \begin{itemize}
                \item a position of dimension $k$ is in $S_{kk}^B$ exactly when it is source boundary,
                \item a position of dimension $k$ is in $T_{kk}^B$ exactly when it is target boundary,
                \item $S_{kl}^B = T_{kl}^B = \Pos_l(B)$ for all $l<k$.
        \end{itemize}

        All of those statements follow trivially by the definitions when $B = D_0$ or $k = 0$, so let $B_0$ a Batanin tree and $L = [B_1,\dots, B_n]$ a list of Batanin trees and suppose that the statements hold for $B_0$ and $\br\,L$, and that they hold for $B = \br\,[B_0,B_1, \dots, B_n]$ when $k=0$.

        For any $k\in \N$, by definition of $s_{k+1}^B$ on positions of dimension $k+1$, $S_{k+1,k+1}^B$ consists of positions of the form $\inl\,(B_0,L,s_k^{B_0}\,p)$ or $\inr\,(B_0,L,s_{k+1}^{\br\,L}\,q)$. By the inductive hypothesis and the inductive characterization above, those are precisely the source boundary positions of $B$ of dimension $k+1$. The same argument shows that the second statement is also true for $B$.

        We will prove the last statement by induction on $k - l$. Let $k\in \N$. We will first show that $S_{k+1,k}^B$ contains all positions of $B$ of dimension $k$. It clearly contains all source boundary positions, so let $p\in \Pos_k(B)$ and suppose that it is not source boundary. Then one of the following holds.
        \begin{itemize}
                \item Suppose $k=0$ and $p = \inl\,(B_0,L,p_0)$ for $p_0\in \Pos_{k-1}(B_0)$ not source boundary. By the inductive hypothesis, $p_0\in S_{k,k-1}^{B_0}$, so it is the source or target of some position $p_1 \in S_{kk}^{B_0} = \partial_k^s B_0$. Then $p$ is the source or target of $\inl\,p_1 \in \partial_{k+1}^s B$, so it is in $S_{k+1,k}^B$.
                \item Suppose that $p = \inr\,(B_0,L,q)$ for some $q\in \Pos_{k}(\br\,L)$ that is not source boundary. Using the inductive hypothesis for $\br\,L$, we can conclude again as above that $p\in S_{k+1,k}^B$.
                \item Suppose finally that $k = 0$, $L = [ \, ]$ and $p = \inr\,(B_0,L, \obpos)$. Then $p$ is the target of $\inl\,(B_0,L,\obpos)$, which is source boundary, so $p\in S_{k+1,k}^B$
        \end{itemize}
        We can similarly show that $T_{k+1,k}^B$ contains all positions of $B$ of dimension $k$. It contains target boundary positions, so let $p\in \Pos_k(B)$ not target boundary. Splitting cases as above and using the same argument, it remains to show $p = \obpos \in T_{1,0}^B$. For that, let $p_0\in \Pos_0(B_0)$ the unique target boundary $0$-position of $B_0$. Then $\inl\,p_0$ is target boundary and has source $p$, so $p\in T_{1,0}^B$.

        Finally, we are left to show for $l<k$ that $S_{kl}^B = \Pos_l(B) = T_{kl}^B$, assuming that this is true for $l+1$. Since every position of dimension $l$ is either source boundary or the target of some position in $S_{k,l+1}^B$, the first equality holds. The second follows similarly.
\end{proof}

\begin{remark}
Where it does not cause confusion and simplifies our presentation, we will often write $B$\ instead of $\Pos(B)$.
\end{remark}

\section{Computads and their cells}\label{sec-computads}

In this section we will give our main definition  of the category $\Comp$ of computads, and investigate its consequences. We begin in Section~\ref{sec:maindef} with the definition of finite-dimensional computads, making heavy use of the concepts of inductive set and recursively-defined function; for those readers who might appreciate an introduction to these important concepts, we suggest the works of Pierce~\cite[Section 3.3]{pierce_textbook}, Winskel~\cite[Chapter 3]{winskel_textbook} and Mitchell~\cite[Chapter 1]{mitchell_textbook}.  We then explore this definition conceptually in Section~\ref{sec:lowdimension} with some simple examples in low dimension.
Our definition is complex, and we address wellfoundedness in Section~\ref{sec:well-foundedness}, giving a concrete interpretation via the initial algebra of a polynomial endofunctor. In Section~\ref{sec:infinitecomputads} we define the category \Comp of infinite-dimensional computads as a limit of the finite-dimensional ones.

\subsection{Main definition}
\label{sec:maindef}

Here we present an indexed inductive-recursive definition of computad. Our definition has been formalised in Agda~\cite{markakis_computads_2023}, and interested readers may view that formalisation alongside the text below, guided by references within the formalisation which indicate the relevant part of the definition.

Our definition begins by supplying some notation: we set \mbox{$\Comp_{-1} = \star$}, the terminal category, and define $\Type_{-1} : \star \to \Set$ as a functor picking out a singleton set. We then define by ordinary induction on $n\in \N$ the following structures:
\begin{itemize}
        \item A category $\Comp_n$ of $n$\-computads and $n$\-homomorphisms.
        \item A forgetful functor $u_n : \Comp_{n} \to \Comp_{n-1}$, giving for every $n$\-computad its underlying $(n-1)$\-computad.
        \item For every globular set $X$, a computad $\Free_n X$.
        \item A functor $\Cell_n : \Comp_n \to \Set$ giving for every $n$\-computad its set of cells, and for every $n$\-homomorphism its associated cell function.
        \item A functor $\Type_n : \Comp_n \to \Set$ giving for every $n$\-computad its set of $n$\-spheres, and for every $n$\-homomorphism its associated sphere function.
        \item A natural transformation $\ty_n : \Cell_{n} \to \Type_{n-1} \circ\, u_n$, giving for every cell of an $n$\-computad its boundary sphere in the underlying $(n-1)$\-computad.
        \item For every globular set $X$, a function $\postype_{n,X}$ giving an $n$\-sphere of $\Free_n X$ from a pair $(a,b)\in X_n\times X_n$ with common source and target.
        \item For every $n$-computad $C$ and $n$-cell $c\in \Cell_n(C)$, a set $\fv\,(c)$ the \emph{support} of $c$.
        \item For every Batanin tree $B$, a distinguished subset of $\Type_n\Free_n(B)$ called the \emph{full spheres}.
\end{itemize}

\noindent Here we give some intuition as to how these structures will be used later in the paper. We will see that $n$\-computads consist of sets of generators up to dimension $n$ with specified source and target, and that they form generating data for weak $\omega$\-categories, while $n$\-homomorphisms are precisely strict functors between the generated $\omega$\-categories. The assignment $\Free$ includes globular sets to $n$\-computad by forgetting its cells above dimension $n$. The functor $\Cell_n$ sends an $n$\-computad to the set of $n$\-cells of the $\omega$\-category it generates; those are either generators of the computad, or they are obtained by them by using the operations of $\omega$\-categories, which we call coherences. The functor $\Type_n$ sends an $n$-computad to the set of parallel pairs of $n$-cells. The boundary transformation sends a cell to its source and target, exhibited by a sphere of lower dimension. The function $\postype$ is an auxiliary function, viewing a pair of parallel cells of a globular set as a pair of parallel generator cells of the free $\omega$\-category it generates. The concepts of support and fullness are important for generating the operations of the $\omega$\-category.

\paragraph{Base case.} Our definition starts by giving the base case for the induction. This corresponds to the dimension $n=0$. Here we proceed as follows:

\begin{itemize}
        \item $\Comp_0 = \Set$
        \item $u_0 : \Set \to \star$ is the unique such functor
        \item For a globular set $X$, we let $\Free_0X = X_0$
        \item $\Cell_0 : \Set \to \Set$ is the identity functor
        \item $\Type_0 : \Set \to \Set$ acts as $\Type_0(C) = C \times C$, the Cartesian product
        \item $\ty_0 : \id_\Set \to \Type_{-1} \circ\, u_0$ is the unique such natural transformation
        \item For a globular set $X$, we let $\postype_{0,X}$ be the identity of $X_0\times X_0$
        \item For any $0$-computad $C$ and $c\in \Cell_0(C)= C$, we set the support of $c$ to be $\fv(c) = \{c\}$
        \item For a Batanin tree $B$, a $0$-sphere $(a,b)$ of $\Free_n B$ is a pair of $0$-positions. We will say that it is full when $a$ is source boundary and $b$ is target boundary. By the discussion in the previous section, every tree has exactly one full $0$-sphere.
\end{itemize}

\noindent
To understand this assignment, we point out that there should be no ways to ``compose objects'' in an $\omega$\-category, the $0$\-cells of the $\omega$\-category generated by a $0$\-computad should be exactly generators. That forces $\Cell_0$ to be the identity functor. Similarly, since all $0$\-cells are parallel, $\Type_0$ sends a set to the set of pairs of elements of that set. The definition of full 0-spheres derives from Maltsiniotis's notion of admissible pairs~\cite{maltsiniotis_grothendieck_2010}.

For the inductive step, we now fix $n > 0$ and suppose that we have
defined all of the required data for $n-1$.  In particular, in what follows we may freely refer to $\Comp_{n-1}$, $u_{n-1}$, $\Free_{n-1}$, $\Cell_{n-1}$, $\Type_{n-1}$, $\ty_{n-1}$, $\postype_{n-1}$ and $\operatorname{Full}_{n-1}$. The definition then proceeds as follows.

\paragraph{Objects.}  The objects of the category $\Comp_n$ are triples
$C = (C_{n-1}, V^C_n, \phi^C_n)$, called \textit{$n$\-computads}, where
\begin{itemize}
        \item $C_{n-1} \in \Comp_{n-1}$
        \item $V^C_n$ is a set, whose elements we call \emph{generators}
        \item $\phi^C_n : V^C_n \to \Type_{n-1}(C_{n-1})$ is a function, assigning to each
              generator a \emph{sphere} of $C_{n-1}$; that is, a choice of source and target cells
\end{itemize}
The forgetful functor $u_n : \Comp_n \to \Comp_{n-1}$ acts on $n$-computads by projecting out the first component.

\paragraph{Globular sets.} For a globular set $X$, we define the $n$-computad $\Free_n X$ associated to $X$ by
\[ \Free_n X = (\Free_{n-1}X, X_n, \phi^{\Free X}_{n}) \] where
\[\phi^{\Free X}_{n}(x) = \postype_{n-1,X} \, (\src \, x, \tgt \, x).\]
For $n\ge 2$, the globularity conditions ensure that the positions $\src \, x$ and $\tgt \, x$ have common source and target, so that $\postype_{n-1,X} \, (\src \, x, \tgt \, x)$ is defined. Note that $\postype_{n-1}$ has been defined by the inductive hypothesis.

\vspace{15pt}
In what follows we fix an  $n$-computad $C$, and define the following by mutual induction: the set $\Cell_n(C)$ of cells of $C$, the boundary operation $\ty_n$ for those cells, and the set of $n$\-homomorphisms $\Comp_n(D,C)$ from an arbitrary $n$-computad $D$ to $C$. The mutual nature of this induction means that we must refer to some of these sets and functions before they are defined; in particular we refer to $\Cell_n$ and $\ty_n$ in the definition of $n$\-homomorphisms given immediately below. This is a typical feature of mutual induction, and a thorough analysis is given in Section~\ref{sec:well-foundedness}.

\paragraph{Morphisms.} An $n$-homomorphism $\sigma : (D_{n-1},V^D_n,\phi^D_n)\to (C_{n-1},V^C_n,\phi^C_n)$ is a pair $(\sigma_{n-1}, \sigma_V)$ consisting of
\begin{itemize}
        \item a $(n-1)$-homomorphism $\sigma_{n-1} : D_{n-1} \to C_{n-1}$
        \item a function $\sigma_V : V^D_n \to \Cell_{n}(C_{n-1}, V_n^C, \phi^C_n)$
\end{itemize}
such that the following square commutes
\[\begin{tikzcd}[column sep = 70pt]
        V_n^D & \Cell_n(C) \\
        \Type_{n-1}(D_{n-1}) & \Type_{n-1}(C_{n-1})
        \arrow["{\phi_n^D}", from=1-1, to=2-1]
        \arrow["{\ty_{n,C}}", from=1-2, to=2-2]
        \arrow["{\Type_{n-1}(\sigma_{n-1})}"', from=2-1, to=2-2]
        \arrow["{\sigma_V}", from=1-1, to=1-2]
\end{tikzcd}\]
The forgetful functor $u_n$ is the evident one, sending $\sigma$ to $\sigma_{n-1}$.

We recall that $n$\-homomorphisms can be interpreted as strict functors between the $\omega$\-categories generated by the $n$\-computads. The universal property of free $\omega$\-categories on computads, as stated by Street~\cite{street_orientals} in the strict setting, shows that such functors are determined by their action on generators, which may be chosen freely subject to source and target conditions. Here the action on generators is given by $\sigma_{n-1}$ and $\sigma_V$, while compatibility of those actions with the source and target functions is given by the commutativity of the square above.

\paragraph{Cells and their boundary.}\label{cellsboundary} For an $n$-computad $C = (C_{n-1},V^C_n,\phi^C_n)$, the set
$\Cell_n(C)$ is defined inductively as follows:
\begin{itemize}
        \item For every element $v \in V_n^C$ an element $\var  v \in \Cell_n(C)$.
        \item For every triple $(B, A, \tau)$ where
              \begin{itemize}
                      \item $B$ is a Batanin tree with $\dim(B) \leq n$,
                      \item $A$ is a full sphere in $\Type_{n-1}(\Free_{n-1}B)$,
                      \item $\tau = (\tau_{n-1}, \tau_V)$ is an $n$-homomorphism
                            $\Free_n B \to C$
              \end{itemize}
              an element $\coh \, (B, A, \tau) \in \Cell_n(C)$.
\end{itemize}
The $n$\-cells of an $n$\-computad will become the $n$\-cells of the $\omega$\-category that it generates. The first constructor $\var$ ensures that each generator gives rise to an $n$\-cell. Cells of the form $\coh\,(B , A , \id)$ play the role of the liftings to the pair $A$ in the sense of Maltsiniotis~\cite{maltsiniotis_grothendieck_2010}, so we may interpret $\coh\,(B , A , -)$ as an ``operation'' in the $\omega$\-category generated by $C$. When \mbox{$n=\dim B$,} this can be thought of as a primitive ``composition operation'', while for $n > \dim B$ it can be thought as a primitive ``coherence law'' which relates the operations determined by the source and target of $A$. We will see some examples of this in Section~\ref{sec:lowdimension}.

For every $c\in \Cell_n(C)$, the boundary $\ty_{n,C}(c)$ of $c$ is then defined recursively as follows:
\begin{itemize}
        \item If $c = \var v$ for some $v\in V^C_n$, then
              \[\ty_{n,C}\,(\var v) = \phi_n^C(v). \]
        \item If $c = \coh\, (B, A, \tau)$ is a coherence cell, then
              \[\ty_{n,C}\,(\coh \, (B, A, \tau)) = \Type_{n-1}(\tau_{n-1})(A).\]
\end{itemize}

\paragraph{Composition and functoriality of cells.}
Let $C = (C_{n-1}, V^C_n, \phi^C_n)$ and $D = (D_{n-1}, V^D_n, \phi^D_n)$ be $n$-computads and let $\sigma = (\sigma_{n-1},\sigma_V) : C\to D$ an $n$-homomorphism. We define the function $\Cell_n(\sigma)$ and post-composition by $\sigma$ recursively as follows:
\begin{itemize}
        \item If $c = \var\, v$ for some $v\in V^C_n$, we set
              \[ \Cell_n(\sigma)(\var \, v) = \sigma_V(v) \]
        \item If $c = \coh \, (B, A, \tau)$ is a coherence cell of $C$, we set
              \[ \Cell_n(\sigma)(\coh \, (B, A, \tau)) = \coh \, (B, A, \sigma \circ \tau) \]
        \item Given an $n$-homomorphism $\tau = (\tau_{n-1}, \tau_V) : (E_{n-1}, V^E_{n}, \phi^E_n) \to C$, we define the composition of $\sigma$ and $\tau$ by
              \[ \sigma \circ \tau = (\sigma_{n-1} \circ \tau_{n-1}, v \mapsto \Cell_n(\sigma)(\tau_V(v)))\]
\end{itemize}

We may show that this definition makes $\Comp_n$ into a category and $\Cell_n$ into a functor by structural induction. For instance, given $\sigma : C\to D$ and $\tau : D\to E$ a pair of composable $n$-homomorphisms, we can easily show by induction that for every $c\in \Cell_n(C)$,
\[ \Cell_n(\tau \circ \sigma) (c) =  \Cell_n(\tau) (\Cell_n(\sigma)(c))\]
and that for every $n$-homomorphism $\rho$ with target $C$,
\[      (\tau \circ \sigma)\circ \rho = \tau\circ(\sigma\circ \rho).\]
With a similar argument, we can also show that for every $n$-computad $C$, the $n$-homomorphism given by
\[ \id_C = (\id_{C_{n-1}}, v\mapsto \var \, v)\]
is the identity of $C$ and that it acts trivially on $\Cell_n(C)$.

\paragraph{Spheres.} For an $n$-computad $C = (C_{n-1}, V^C_n, \phi^C_n)$, the set
$\Type_n(C)$ is defined to be the set of triples $(A,a,b)$
where
\begin{itemize}
        \item $A \in \Type_{n-1}(C_{n-1})$ is an $(n-1)$-sphere of $C_{n-1}$
        \item $a, b \in \Cell_n(C)$ are $n$-cells which satisfy
              \[ \ty_{n,C}(a) = \ty_{n,C}(b) = A \]
\end{itemize}
An $n$-homomorphism $\sigma = (\sigma_{n-1},\sigma_V) : C\to D$ acts on $n$-spheres by
\[ \Type_n(\sigma)(A,a,b) = (\Type_{n-1}(\sigma_{n-1})(A), \Cell_n(\sigma)(a), \Cell_n(\sigma)(b))\]
and this assignment is clearly functorial. We define also for $i = 1, 2$ a natural transformation $\pr_i : \Type_n \Rightarrow \Cell_n$ by the equations
\begin{align*}
        \pr_{1,C}(A, a, b) & = a & \pr_{2,C}(A, a, b) & = b
\end{align*}

\paragraph{Globular sets (cont.)} For a globular set $X$, the $n$-sphere associated to a pair $(p,q)\in X_n \times X_n$ with common source and target is given by
\[      \postype_{n,X} \, (p,\ q) = (\postype_{n-1,X}\, (\srcpos\, p , \tgtpos\, p),  \var\, p, \var\, q).       \]

\paragraph{Support.}
We can determine the support of a cell with a function
\[\fv : \Cell_n(C_{n-1}, V^C_n, \phi^C_n) \to \mathcal{P}(V^C_n)\]
defined by
induction on the structure of $c \in \Cell_n(C_{n-1}, V^C_n, \phi^C_n)$, as follows.
\begin{itemize}
        \item If $c = \var \, v$ for some generator $v \in V^C_n$, we set
              \[ \fv\,(\var \, v) = \{ v \} \]
        \item If $c = \coh \, (B, A, \tau)$ is a coherence cell, we set
              \[ \fv\,(\coh \, (B, A, \tau)) = \bigcup_{p \in \Pos_n(B)} \fv\,(\tau_V(p)) \]
\end{itemize}
Considering an $n$\-cell as a formal composite, the support of a cell is the set of $n$\-dimensional generators that appears in it.

\paragraph{Fullness.}  For a Batanin tree $B$, we can ask if a given $n$-sphere of $\Free_n B$ is full. We say that a sphere $(A, a, b) \in \Type_n(\Free_n B)$ is \emph{full} when:
\begin{itemize}
        \item $\fv(a) = \partial_n^s(B)$
        \item $\fv(b) = \partial_n^t(B)$
        \item $A \in \Type_{n-1} (\Free_{n-1}B)$ is full
\end{itemize}

\noindent Proposition \ref{prop:boundary-positions-generation} implies that for a full sphere $(A,a,b)$, the globular set generated by the support of $a$ and its iterated sources and targets, is the image of the $n$\-boundary of $B$ under the source inclusion $s_n^B$. Dually, the one generated by the support of $b$ and its iterated sources and targets is the image of the $n$\-boundary of $B$ under the target inclusion $t_n^B$. The converse will be shown in Proposition~\ref{full-types-as-covers}. Thus, intuitively, fullness is the condition that $a$ and $b$ ``cover'' the entire $n$\-dimensional source and target, respectively, of $B$. This is an analogue of the notion of admissible pairs of morphisms in the sense of Maltsiniotis~\cite{maltsiniotis_grothendieck_2010}, which are pairs of morphisms from a disk to a Batanin tree that cover its source and target respectively. Note that this notion of fullness is directed, in the sense that it is not symmetric with respect to $a$ and $b$ when $\dim B < n$. This is necessary for the directed behaviour of the weak $\omega$\-categories that we are modelling.

\subsection{Low-dimensional examples}
\label{sec:lowdimension}

Here we unpack our main definition in low dimensions, demonstrating its properties and behaviour in some simple cases. We give an explicit description of $1$-computads and their cells, and we  describe composition of $1$-cells, as well as exhibiting  $2$-cells that witness associativity. We describe how to extend $n$-computads to computads

\paragraph{Understanding 1-computads.}
A $1$-computad $C = (V_0,V_1,\phi : V_1\to V_0\times V_0)$ can be seen as a directed graph. We will denote a cell $f$ such that $\phi(f) = (x, y)$ by an arrow
\begin{equation*}
        \begin{tikzcd}
                x \ar[r, "f"] & y.
        \end{tikzcd}
\end{equation*}
Recall that trees of dimension at most one are of the form $B_n = [D_0,\dots,D_0]$ for $n \in \N$. Morphisms ${\star_v :\Free_1 B_0 \to C}$ are in bijection with vertices $v \in V_0$, while for $n>0$, morphisms $\sigma_{f_1,\dots,f_n}: \Free_1 B_n\to C$ correspond to sequences of composable $1$-cells
\begin{equation*}
        \begin{tikzcd}
                x_0 \ar[r, "f_1"] &
                x_1 \ar[r, "f_2"] &
                \cdots \ar[r, "f_n"] &
                x_n.
        \end{tikzcd}
\end{equation*}
Recall also that every Batanin tree has a unique full $0$-sphere. Equipped with these observations, we can unpack the inductive definition of cells (``Cells and their boundary'', page~\pageref{cellsboundary}) as follows. The constructor $\var$
embeds edges into cells. The constructor $\coh$ specialised to the tree
$B_0 = D_0$ and its unique full $0$\-sphere $(\herepos,\herepos)$ constructs a
new cell $\coh(B_0,(\herepos,\herepos),\star_v)$ out of a vertex $v\in V_0$, that we will denote by $\id_v$. Similarly, the constructor $\coh$ specialised to the tree $B_n$ for $n > 0$ with its unique full $0$\-sphere $(\herepos, \inr^n(\herepos))$ constructs a new cell $\coh(B_n,(\herepos,\inr^n(\herepos)),\sigma_{f_1,\dots,f_n})$ out of a sequence of composable cells $f_1,\dots,f_n$  as above, that we will denote $\comp(f_1,\dots,f_n)$. Since those cases exhaust all trees of dimension at most $1$ and all their full $0$\-types, the set of $1$-cells of $C$ is defined inductively by the following rules:
\begin{itemize}
        \item For each edge $e\in V_1$ with $\phi(e) = (x, y)$, there exists a cell $\var e$ with boundary sphere $\phi(e)$. That is,
              \begin{equation*}
                      \begin{tikzcd}
                              x \ar[r, "\var e"] & y.
                      \end{tikzcd}
              \end{equation*}

        \item For every vertex $v \in V_0$, there exists an \emph{identity cell}
              \begin{equation*}
                      \begin{tikzcd}
                              v \ar[r, "\id_v"] & v.
                      \end{tikzcd}
              \end{equation*}
        \item For any composable $1$-cells $f_1,\dots,f_n$, there exists a \emph{composite cell}
              \begin{equation*}
                      \begin{tikzcd}[column sep=6em]
                              x_0 \ar[r, "{\comp(f_1,\dots,f_n)}"] & x_n.
                      \end{tikzcd}
              \end{equation*}
\end{itemize}
Thus, each $1$-cell with boundary $(x, y)$ in $C$ corresponds to a parenthesised path from $x$ to $y$ built out of edges of $e$ and identity cells. Hence, a morphism of $1$-computads $\sigma :C\to D$ amounts to a function $\sigma_0 : V_0^C\to V_0^D$, together with a function assigning to every edge $e : x\to y$ of $C$, a parenthesised path from $\sigma_0(x)$ to $\sigma_0(y)$. This is how we would expect a strict functor to act on generating 1-cells.

\paragraph{Some coherence cells.}
Having defined composition of $1$-cells, we proceed to give examples of coherence $2$-cells. For this, consider the $2$-computad $\Free_2 B_3$, which we may picture as follows
\begin{equation*}
        \begin{tikzcd}
                d_0 \ar[r, "f"] &
                d_1 \ar[r, "g"] &
                d_2 \ar[r, "h"] &
                d_3,
        \end{tikzcd}
\end{equation*}
and the $1$-sphere
\[
        A = ((d_0, d_3), \comp(f, \comp(g, h)), \comp(\comp(f, g), h)).
\]
Since $d_0$ is the unique source boundary $0$-position of $B_3$ and $d_3$ is the unique target boundary position, the first component of $A$ is full. Moreover, the supports of $\comp(f, \comp(g, h))$ and of $\comp(\comp(f, g), h)$ consist of all $1$-positions of $B_3$. Since $B_3$ does not contain any higher dimensional positions, all $1$-positions of $B_3$ are both source and target boundary, making $A$ a full sphere. Therefore, we may form the following $2$-cell:
\[\alpha = \coh(B_3,A,\id).\]
Given any $2$-computad $E = (C, V_2, \psi)$, a $2$-homomorphism
$\Free_2 B_3\to E$ is the same as a $1$-homomorphism $\Free_1 B_3\to C$, since $B_3$ has no $2$-dimensional generators. Hence, for every triple of composable $1$-cells $f_1,f_2,f_3$ of $C$, we may form the following $2$-cell, which we  interpret as an associator:\[\Cell_2(\sigma_{f_1,f_2,f_3})(\alpha) : \comp(f_1,\comp(f_2,f_3)) \Rightarrow \comp(\comp(f_1,f_2),f_3)\]
Different choices of Batanin tree and full spheres induce different coherence $2$-cells. For example, let $A'$ be the full $1$-sphere defined by
\[
        A' = ((d_0, d_3), \comp(f, g, \id_{d_2}, h), \comp(f, g, h)).
\]
Then $\upsilon = \coh(B_3, A', \sigma_{f_1,f_2,f_3})$ can be seen as an unbiased sort of \emph{unitor}. Similarly we can produce coherence cells corresponding to left and right unitors, interchangers, and other higher coherences that are expected to form part of a theory of weak globular higher categories.

\paragraph{Simplicial sets.}
The non-degenerate simplices of a $2$-simplicial set $X$ assemble into a $2$-computad $C(X)$. The
underlying $1$-computad $C_1X$ is defined via the boundary function
\[\phi_1^{C(X)} : X_1^{\text{nd}} \to \Type_0(X_0) = X_0\times X_0\]
sending a non-degenerate $1$-simplex to its faces
\[\phi_1^{C(X)}(x) = (d^1x, d^0x).\]
Every $1$-simplex of $X$ gives rise to a cell of $C_1X$ via the function
\[\psi^X : X_1 \to \Cell_1(C_1X) \]
which sends a non-degenerate simplex to the generator it generates, and which sends a degenerate simplex $x = s^0d^1x = s^0d^0x$ to the coherence cell $\id_{d^1 x}$ defined above.
The computad $C_2X$ is then defined via the boundary function
\[\phi_2 : X_2^{\text{nd}} \to \Type_1(C_1X) \]
sending a $2$-simplex $x$ to
\[\phi_2(x) = \big((d^1d^1x,d^0d^1x), \comp(\psi(d^2x),\psi(d^0x)), \psi(d^1x) \big).\]
The fact that $\phi_2(x)$ is well-defined follows easily from the simplicial identities.

\[\begin{tikzcd}
                & \bullet \\
                \bullet && \bullet
                \arrow[""{name=0, anchor=center, inner sep=0}, "{d^1x}"', from=2-1, to=2-3]
                \arrow["{d^2x}", from=2-1, to=1-2]
                \arrow["{d^0x}", from=1-2, to=2-3]
                \arrow["x", shorten >=4pt, shorten <=4pt, Rightarrow, from=1-2, to=0]
        \end{tikzcd}\]

\noindent Further work is needed to extend  this construction to $n$-simplicial sets for $n>2$, since the composition operations in our computads are only weakly associative and unital.

\subsection{Remarks on the definition}
\label{sec:well-foundedness}

In Section~\ref{sec:maindef} we gave an indexed inductive-recursive definition of a category $\Comp_n$ of $n$\-computads, and certain functors and natural transformations attached to it. The definition has a complex structure; for example, it uses large induction-recursion in the parameters, since an $n$-computad has a set of variables, and sets form a proper class. As a result of this complexity, the standard techniques of induction-recursion are not sufficient to prove that our definition is well-founded. As already discussed, our definition is formalised in Agda; however, Agda supports a broad class of inductive constructions, including some cases where the semantics are not yet fully described in the literature~\cite{stackexchange-meven}.
To overcome this, we give explicit semantics for our definition in terms of the initial algebra of an endofunctor. We use this to define a depth invariant of cells and homomorphisms, which will be used in some inductive arguments later in the paper.

To begin, let us analyse the structure of our definition. At the top level, our definition proceeds by induction on the dimension $n$. The base case is quite straightforward, so let $n > 0$ and suppose that the various structures have all been defined for $n-1$. Then the definition of those structures for $n$ proceeds as follows.

First of all, the objects of the category $\Comp_n$ are given. Then the action of the forgetful functor $u_n$ on objects, and the  computads $\Free_n(X)$ obtained from a globular set $X$, can be defined in either order.

Once those structures have been defined, we fix a computad $C$ and define simultaneously the set of $n$\-cells of $C$ with the boundary natural transformation $\ty_n C$, and the sets of morphisms of computads with target $C$ with the action of the forgetful functor on them, using induction-recursion. More precisely, our inductive-recursive definition describes functions
\begin{align*}
    \ty_{n,C} : \Cell_nC &\to \Type_{n-1}C_{n-1} \\
    u_{n,D,C} : \Comp_n(D,C) &\to \Comp_{n-1}(D_{n-1}, C_{n-1})
\end{align*}
where the second family ranges over all $n$\-computads. These functions are carriers of the initial algebra of an endofunctor on the following large category:
\[
    \faktor{\Set}{\Type_{n-1}C_{n-1}} \times \prod_{D\in \Comp_n} \faktor{\Set}{\Comp_{n-1}(D_{n-1}, C_{n-1})}.
\]
Using the equivalence $\Set_J \cong \Set^J$ for every set $J$, we can easily see that this category is equivalent to the category of families of sets indexed by the class
\[
    \Type_{n-1}C_{n-1} \amalg \coprod_{D \in \Comp_n} \Comp_{n-1}(D_{n-1}, C_{n-1}).
\]
The endofunctor $F$ then takes a family of sets $\{X_A\}$ indexed by $(n-1)$\-spheres of $C_{n-1}$, and a family of sets $\{X_{D,\rho}\}$
indexed by $n$\-computads $D$ and morphisms $\rho_{n-1} : D_{n-1} \to C_{n-1}$, and returns the family of sets
\begin{align*}
    F( \{ X_A \} , \{ X_{D,\rho} \}) = &\left( \Big\{ (\phi_n^C)^{-1}(A) + \coprod_{\mathclap{B,A_0,\tau}} X_{\Free_n B, \tau} \Big\}, \right.
\\[-10pt]
&\hspace{1cm} \left. \Big\{ \prod_{\mathclap{v\in V_n^D}} X_{\Type_{n-1}(\rho_{n-1})\phi_{n}^D(v)} \Big\} \right)
\end{align*}
where the disjoint union ranges over all Batanin trees $B$ of dimension at most $n$, all full $(n-1)$\-types $A_0$ of $B$, and all morphisms of computads $\tau : \Free_{n-1}B \to C_{n-1}$ satisfying that $\Type_{n-1}(\tau)(A_0) = A$.

This endofunctor $F$ therefore encapsulates the definition  of the cells and homomorphisms for an $n$-computad $C$. On the first line, the first component amounts to the constructor $\var$ including generators into cells, while the second component amounts to the coherence constructor $\coh$. The second line amounts precisely to a function $\rho_V : V_n^D\to \Cell_n(C)$ making the square in the definition of morphisms commute.

\begin{proposition}
  The endofunctor $F$ has an initial algebra.
\end{proposition}
\begin{proof}
  To show that this endofunctor has an initial algebra, we use Adamek's construction of free algebras \cite{adamek_1974}. In particular, we consider the transfinite sequence of families of sets defined inductively by
  \begin{align*}
      X^0 &= \emptyset & X^{d+1} &= F(X^d) & X^\alpha &= \colim_{d<\alpha} X^d
  \end{align*}
  for $\alpha$ limit ordinal. Using that colimits commute with colimits, we see that for every $A\in \Type_{n-1}C_{n-1}$, we have
  \[
      X^{\omega+1}_A \cong X^{\omega}_A.
  \]
  which implies that for every computad $D$ and morphism $\rho : D_{n-1}\to C_{n-1}$, we have that
  \[X^{\omega+2}_{D,\rho} \cong X^{\omega+1}_{D,\rho}.\]
  Similarly, using that finite products commute with filtered colimits  we see that for every tree $B$ and every morphism $\tau : \Free_{n-1}B\to C_{n-1}$, we have
  \[
      X^{\omega+1}_{\Free_{n-1}B , \tau} \cong X^{\omega}_{\Free_{n-1}B , \tau}
  \]
  and hence also that
  \[
      X^{\omega+2}_A \cong X^{\omega+1}_A.
  \]
  Therefore, $X^{\omega+2} \cong X^{\omega+1}$, so $X^{\omega+1}$ with the isomorphism $X^{\omega+2}\to X^{\omega+1}$ is the initial algebra of $F$.
\end{proof}

\noindent In the notation of the proof, the set of cells of $C$ and the sets of morphisms with target $C$ are given by the disjoint unions
  \begin{align*}
      \Cell_n(C) &= \coprod_{{A\in \Type_{n-1}C_{n-1}}}X^{\omega+1}_A \\
      \Comp_n(D,C) &= \coprod_{{\tau : D_{n-1} \to C_{n-1}}}X^{\omega+1}_{D,\tau}
  \end{align*}
and the boundary function and the forgetful functor are given by the obvious projections. This description of cells and morphisms gives rise to an important invariant of them, that we call the \emph{inductive depth}. We define the inductive depth of a cell $c$ and a morphism $\sigma$ respectively as the least ordinal $\alpha\le \omega+1$, such that $c$ and $\sigma$ are in the image of the map $X^{\alpha}\to X^{\omega+1}$. Unwrapping the definition of the inductive depth, we see that it can be computed by the following recursive equations
\begin{align*}
  \cdpth(\var v) &= 1 \\
  \cdpth(\coh(B,A,\tau)) &= \mdpth(\tau)+1 \\
  \mdpth(\tau : D\to C) &= \sup \{ \cdpth\tau_V(v) \;:\; v\in V_n^D \}+1.
\end{align*}
All instances of structural induction used in the rest of the article can be reduced into transfinite induction on the inductive depth of cells and morphisms, by observing that the inductive hypothesis is always applied to cells and morphisms of strictly lower inductive depth.

Continuing the analysis  of our definition, once cells and morphisms have been defined, we can define composition of morphisms and the action of morphisms on cells. This can be performed by structural recursion as done in the previous section, which is fully rigorous. For maximum clarity, we reformulate this in terms of recursion on the invariants just defined.  In either case, we fix a morphism of $n$\-computads $\sigma:C\to D$, and we define mutually recursively the pre-composition operation
\[
  (\sigma \circ -) : \Comp_n(E,C) \to \Comp_n(E,D)
\]
for all $n$\-computads $E$, the action of homomorphisms on cells
\[
  \Cell_n(\sigma) : \Cell_n(C)\to \Cell_n(D),
\]
along with the proof that the following naturality square commutes:
\[\begin{tikzcd}[column sep=2.6cm]
        {\Cell_n(C)} & {\Cell_n(D)} \\
        {\Type_{n-1}(C_{n-1})} & {\Type_{n-1}(D_{n-1})}
        \arrow["{\ty_{n,D}}", from=1-2, to=2-2]
        \arrow["{\ty_{n,C}}", from=1-1, to=2-1]
        \arrow["{\Cell_n(\sigma)}", from=1-1, to=1-2]
        \arrow["{\Type_{n-1}(\sigma_{n-1})}"', from=2-1, to=2-2]
\end{tikzcd}\]
To express this mutual recursion by an induction on depth, we
define first for $d\le\omega + 1$ and for every computad $E$,
\begin{align*}
  \Cell_n^{(d)}(C) &= \coprod_{{A\in \Type_{n-1}C_{n-1}}}X^{d}_A \\
  \Comp_n^{(d)}(E,C) &= \coprod_{{\tau : D_{n-1} \to C_{n-1}}}X^{d}_{D,\tau},
\end{align*}
where $X_A^d$ and $X_{D,\tau}^d$ are defined in the proof above. We then proceed to define functions
\begin{align*}
  (\sigma\circ -)^{(d)} &: \Comp_n^{(d)}(E,C) \to \Comp_n(E,D) \\
  \Cell_n^{(d)}(\sigma) &: \Cell_n^{(d)}(C) \to \Cell_n(D)
\end{align*}
by transfinite recursion on $d$, showing together that the naturality square above commutes when restricted to $\Cell_n^{(d)}(C)$.

For $d = 0$, those functions are defined as the unique functions out of the empty sets. For $d = \omega$, they are defined by the universal property of the colimit. Finally, for $d = \alpha + 1$ a successor ordinal, we use the formulae of the previous section. Namely for $v\in V_n^C$, we set
\[
  \Cell_n^{(\alpha+1)}(\sigma)(v) = \sigma_V(v),
\]
while for every Batanin tree $B$ of dimension at most $n$, full $(n-1)$\-sphere $A$ of $B$ and morphism $\tau \in \Comp_n^{(\alpha)}(\Free_nB, C)$, we set
\[
  \Cell_n^{(\alpha+1)}(\sigma)(B,A,\tau) = (B, A, (\sigma\circ\tau)^{(\alpha)}).
\]
With this definition, we may check easily that the naturality square commutes when restricted to $\Cell_n^{(\alpha+1)}(C)$. Finally, for every $n$\-computad $E$ and morphism $\tau\in \Comp_n^{(\alpha)}(E,C)$, we set
\[
  (\sigma\circ \tau)^{(\alpha)} = (\sigma_{n-1}\circ\tau_{n-1}, \Cell_n^{(\alpha)}(\sigma)\circ \tau_V).
\]
This is a well-defined morphism of computads by commutativity of the naturality square when restricted to $\Cell_n^{(\alpha)}(C)$. The composition operation and the action of $\sigma$ on cells are given by taking $d = \omega + 1$.

In a similar manner, we may prove by mutual induction that the composition is associative and that the functor of cells preserves composition; in this case, we do this by fixing two composable arrows. Similarly, we can show that the composition is unital and that the functor of cells preserves identities. This completes the definition of the category $\Comp_n$, the functors $u_n$ and $\Cell_n$, and the natural transformation $\ty_n$. Once those are defined, the definition and the proof of functoriality of $\Type_n$ is straightforward. After that is defined, we may finally define the function $\postype_n$ and the full spheres of a Batanin tree.

\subsection{Infinite-dimensional computads}
\label{sec:infinitecomputads}

Having defined the category of $n$\-computads for every $n\in \N$, we may finally define the category of computads (or $\omega$\-computads) as follows.

\begin{definition}
        The category $\Comp$ of computads and homomorphisms is the limit of the forgetful functors:
        \[\begin{tikzcd}
                        \cdots \ar[r, "u_{n+1}"] & \Comp_n \ar[r, "u_n"]&\cdots \ar[r, "u_2"] & \Comp_1 \ar[r, "u_1"] & \Comp_0
                \end{tikzcd}\]
\end{definition}

A computad $C = (C_n)$ therefore comprises a choice for each $n$ of an $n$\-computad $C_n$, with the property that  $C_{n+1} = (C_n, V^C_{n+1}, \phi^C_{n+1})$ for some sets $V_{n+1}^C$ and functions $\phi_{n+1}^C$. A homomorphism $\sigma : C\to D$ comprises similarly a choice for each $n$ of an $n$-homomorphism $\sigma_n : C_n \to D_n$ such that $\sigma_{n+1} = (\sigma_n, \sigma_{n+1,V})$ for some functions $\sigma_{n+1,V} : V^C_{n+1} \to \Cell_{n+1}(D_{n+1})$. By definition of category $\Comp$ as a limit, there exist forgetful functors
\[U_n : \Comp \to \Comp_n,  \]
discarding the structure of a computad above dimension $n$.

\begin{example}\label{ex:disks-and-spheres}
  For every globular set $X$, the $n$-computads $\Free_n X$ for all $n\in \N$ assemble into a computad $\Free X$. In particular, the $n$-globe $D_n$ of Example \ref{example:globes-trees} gives rise to a computad \[\mathbb{D}^n = \Free D_n\] which we will also call the \emph{$n$-globe}. The \emph{spheres} $\mathbb{S}^n$ for $n\in \N\cup\{-1\}$ are another family of computads free on globular sets defined as follows: the $n$\-sphere has two $k$\-dimensional generators $V_k^{\mathbb{S}^n} = \{e_k^-,e_k^+\}$ for every $k \le n$, and no generators of higher dimension; its attaching functions are given by
  \begin{align*}
    \phi^{\mathbb{S}^\infty}_{1}(e_1^\pm)       & = (e_0^-,e_0^+),                                                                  \\
    \phi^{\mathbb{S}^\infty}_{n+2}(e_{n+2}^\pm) & = (\phi^{\mathbb{S}^\infty}_{n+1}(e_{n+1}^-), \var\,e_{n+2}^-, \var\,e_{n+2}^+ ).
  \end{align*}
  Observe that in the same way, we can define the $\infty$\-sphere $\mathbb{S}^\infty$ with two generators in every dimension.
\end{example}

We will now describe for every $n\in\N$ a way to produce a computad out of an $n$-computad. We define first the \emph{skeleton functor}
\[ \sk_n : \Comp_n \to \Comp_{n+1} \]
on an $n$-computad $C$ and on an $n$-homomorphism $\sigma : C\to D$ respectively by
\begin{align*}
        \sk_n(C) & = (C, \emptyset, \{\}) & \sk_n(\sigma) & = (\sigma, \{\})
\end{align*}
where $\{\}$ denotes the empty function.
It follows immediately from the definition that
\[ u_{n+1} \circ \sk_n = \id_{\Comp_n} \]
and that for every pair $(C, D)$ where $C$ is an $n$-computad and $D$ is an $(n+1)$-computad the functor $u_{n+1}$ induces a natural bijection
\begin{align*}
        u_{n+1} & : \Comp_{n+1}(\sk_n C, D) \xrightarrow{\sim} \Comp_n(C, u_{n+1} D),
\end{align*}
whose inverses sends an $n$-homomorphism $\sigma$ to $(\sigma , \{\})$. This shows that $\sk_n$ is left adjoint to $u_{n+1}$ with unit the identity natural transformation.
In particular, $\sk_n$ is fully faithful and injective on objects with image the replete full subcategory of $(n+1)$-computads with no generators of dimension $(n+1)$.

Repeatedly applying the skeleton functors, we get a functor
\[ \Sk_n : \Comp_n  \to \Comp \]
determined by
\[ U_m \circ \Sk_n =
        \begin{cases}
                u_{m+1} \cdots u_n,     & \text{if } m\le n \\
                \sk_{m-1} \cdots \sk_n, & \text{if } m>n.
        \end{cases}\]
As in the finite case, it follows easily that $\Sk_n$ is left adjoint to $U_n$ with unit the identity, and that it is fully faithful with image the replete full subcategory of computads with no generators of dimension more than $n$.

\section{Defining \texorpdfstring{$\omega$}{ω}-categories}
\label{sec-weak-omega-categories}

\subsection{The free computad adjunction}

As we saw in the previous section, every globular set $X$ gives rise to a computad $\Free X$ whose generators are the cells of $X$ and the boundary functions are given by the source and target functions of $X$. This process can easily be made functorial. Recursively on $n\in \N$, we will extend the assignment $\Free_n$ of the previous section to a functor
\[ \Free_n : \Glob \to \Comp_n\]
by setting for a morphism of globular sets $f = (f_n) : X\to Y$,
\begin{align*}
        \Free_0f     & = f_0                                   \\
        \Free_{n+1}f & = (\Free_n f, x\mapsto \var (f_n \, x))
\end{align*}
We let the \emph{free computad functor}
\[\Free : \Glob \to \Comp\]
be the functor induced by this sequence.

On the other hand, recall that for each $n > 0$, an element of $\Type_n(C)$ is a triple $(A,a,b)$
where $A \in \Type_{n-1}(C_{n-1})$ is an $(n-1)$-sphere of $C_{n-1}$ and $a, b \in \Cell_n(C)$ are $n$-cells which satisfy $\ty_{n,C}(a) = \ty_{n,C}(b) = A$. It follows immediately that for every computad $C$, the sets of cells $\Cell_n(U_n C)$ of $C$ form a globular set $\Cell\,(C)$ with source and target maps given by the functions
\[\pr_1 \ty_n: \Cell_n(U_n C) \to \Cell_{n-1}(U_{n-1}C)\]
and
\[\pr_2 \ty_n: \Cell_n(U_n C) \to \Cell_{n-1}(U_{n-1}C)\]
respectively. This construction may be extended to a functor
\[\Cell : \Comp \to \Glob\]
using that $\Cell_n U_n$ is a functor for every $n\in \N$, and the source and target maps are natural.
Similarly, we may define a functor of spheres
\[ \Type : \Comp\to \Glob \]
where the source and target of an $n$-sphere are both given by $\ty_n\pr_1 = \ty_n \pr_2$.

\begin{proposition}
        The functor $\Cell$ is right adjoint to $\Free$.
\end{proposition}
\begin{proof}
        We will prove the proposition by constructing the unit and counit of the adjunction
        \begin{align*}
                \eta        & : \id_{\Glob} \Rightarrow \Cell \circ \Free, \\
                \varepsilon & : \Free\circ \Cell \Rightarrow \id_{\Comp}
        \end{align*}
        The component of the unit $\eta$ on a globular set $X$ is the morphism
        \[\eta_X : X\to \Cell\,(\Free X)\]
        defined on $x\in X_n$ by
        \[\eta_{X,n}(x) = \begin{cases}
                        x,         & \text{if }n=0      \\
                        \var\,(x), & \text{if }n \ge 1. \\
                \end{cases}\]

        The component of the counit $\varepsilon$ on a computad $C = (C_n)$ is the homomorphism given by the sequence
        \[\varepsilon_{C,n} : \Free_n(\Cell\,(C)) \to C_n\]
        defined by induction on $n\in \N$ as follows: The $0$-homomorphism $\varepsilon_{C,0}$ is the identity of the set $C_0$. The $n$-homomorphism $\varepsilon_{C,n}$ for $n\ge 1$ is the pair $(\varepsilon_{C,n-1}, \varepsilon_{C,n,V})$ where
        \[\varepsilon_{C,n,V} : V^{\Free\,(\Cell\,C)}_n = \Cell_n(C) \to \Cell_n(C)\]
        is the identity function. Naturality of both constructions and the triangle laws can be easily verified.

\end{proof}

\begin{corollary}\label{cor:representability-of-cells}
        The functors $\Cell_nU_n,$ and $\Type_nU_n : \Comp\to \Set$ are represented by the \emph{$n$-globe} $\mathbb{D}^n$ and the \emph{$n$-sphere} $\mathbb{S}^n$ respectively. Under those representations, the source and target of an $(n+1)$-cell is given by composition with the source and target inclusions $\Free(s_n^{D_{n+1}})$ and $\Free(t_n^{D_{n+1}})$, respectively.
\end{corollary}
\begin{proof}
        The computad $\mathbb{D}^n$ is free on $D_n$, which is isomorphic to the representable globular set on $[n]$. Hence for every computad $C$, there exists a natural isomorphism
        \[\Comp\,(\mathbb{D}^n, C) \cong \Glob\,(\mathbb{G}(-,[n]), \Cell\,C) \cong \Cell_n(C).\]
        This isomorphism sends a morphism $\sigma : \mathbb{D}^n\to C$ to the image of the unique generator of $\mathbb{D}^n$ of dimension $n$.

        Let now $S_n$ the globular set that the $n$-sphere is free on and let
        $j_n^\pm : D_n\to S_n$ be the morphism of globular sets which sends the position of $D_n$ of dimension $n$ to $e_n^\pm$, and let $i_n : S_{n-1}\to D_n$ be the morphism of globular sets which sends $e_k^-$ and $e_k^+$ to the iterated source and target of dimension $k$ of that position. Then, the following square is a pushout in the category of globular sets:
        \[\begin{tikzcd}
                        {S_{n-1}} & {D_n} \\
                        {D_n} & {S_n} \pomark
                        \arrow["{j_n^-}"', from=2-1, to=2-2]
                        \arrow["{j_n^+}", from=1-2, to=2-2]
                        \arrow["{i_n}", from=1-1, to=1-2]
                        \arrow["{i_n}"', from=1-1, to=2-1].
                \end{tikzcd}\]
        Moreover,
        \begin{align*}
                s_n^{D_n+1} & =  i_{n+1}j_n^- &
                t_n^{D_n+1} & =  i_{n+1}j_n^+
        \end{align*}
        We will construct for every $n\ge -1$ a representation of $\Type_nU_n$ by $\mathbb{S}^n$ such that the natural transformations $\pr_1, \pr_2$ and $\ty_n$ are induced by $j_n^-$, $j_n^+$ and $i_n$ respectively.

        The functor $\Free$ is cocontinuous, so $\mathbb{S}^{-1} = \Free(S_{-1})$ is the initial computad. Since $\Type_{-1}U_{-1}$ is constant on a singleton set, it is represented by the $(-1)$\-sphere. Moreover, $\ty_0U_0$ is induced by $\Free(i_0)$ being the only homomorphism out of the initial computad. Suppose therefore for $n\in\N$ that a representation of $\Type_{n-1}U_{n-1}$ by $\mathbb{S}^{n-1}$ has been given such that $\ty_nU_n$ is induced by $\Free(i_n)$. Then there exists a commutative diagram of the following form, where both the interior and exterior squares are pullbacks:
        \[\begin{tikzcd}
                        {\Type_n(C)} &&& {\Cell_n(C)} \\
                        & {\Comp(\mathbb{S}^n,C)} \pbmark & {\Comp(\mathbb{D}^n,C)} \\
                        & {\Comp(\mathbb{D}^n,C)} & {\Comp(\mathbb{S}^{n-1},C)} \\
                        {\Cell_n(C)} &&& {\Type_{n-1}(C)}
                        \arrow["{\ty_n}", from=1-4, to=4-4]
                        \arrow["{\pr_1}"', from=1-1, to=4-1]
                        \arrow["{\ty_n}"', from=4-1, to=4-4]
                        \arrow["{\pr_2}", from=1-1, to=1-4]
                        \arrow["{(j_n^-)^*}"', from=2-2, to=3-2]
                        \arrow["{(j_n^+)^*}", from=2-2, to=2-3]
                        \arrow["{(\iota_n)^*}"', from=3-2, to=3-3]
                        \arrow["{(\iota_n)^*}", from=2-3, to=3-3]
                        \arrow["\sim", from=1-4, to=2-3]
                        \arrow["\sim", from=4-4, to=3-3]
                        \arrow["\sim", from=4-1, to=3-2]
                        \arrow["\sim", dashed, from=1-1, to=2-2]
                \end{tikzcd}\]
        The universal properties of those squares gives a representation of $\Type_nU_n$ by~$\mathbb{S}^n$.

        To show that $\ty_{n+1}U_{n+1}$ is induced by $\Free(i_{n+1})$ and complete the proof, we need to show that if $A_n \in \Type_n(\mathbb{S}^n)$ is the universal $n$-sphere, then the sphere of the top dimensional position of $D_{n+1}$ is $\Type_{n}\Free\,(i_{n+1})(A_n)$. The explicit description of $j_n^\pm$ and commutativity of the diagram above shows that $A_n$ comprises of the cells $e_n^-$ and $e_n^+$. Those are sent to the source and target of the top dimensional position of $D_{n+1}$ under $\Cell_n\Free_n(i_{n+1})$, so $\ty_{n+1}U_{n+1}$ is indeed induced by $\Free(i_{n+1})$.
\end{proof}


\subsection{The free \texorpdfstring{$\omega$}{ω}-category monad}

The free computad adjunction
\[\begin{tikzcd}
                {\Free: \Glob} & {\Comp : \Cell}
                \arrow[""{name=0, anchor=center, inner sep=0}, shift right=1.5, from=1-1, to=1-2]
                \arrow[""{name=1, anchor=center, inner sep=0}, shift right=1.5, from=1-2, to=1-1]
                \arrow["\dashv"{anchor=center, rotate=90}, draw=none, from=0, to=1]
        \end{tikzcd}\]
induces a monad $(\fc^{w},\eta^w,\mu^w)$ on the category of globular sets, whose endofunctor sends a globular set $X$ to the cells $\Cell\,(\Free X)$ of the free computad on~it.

\begin{definition}
        We will call $\fc^w$-algebras \emph{$\omega$-categories} and their morphisms \emph{homomorphisms}. We will denote their category by $\omega\Cat$. We will call $\fc^w$ the free (weak) $\omega$\-category monad.
\end{definition}

\noindent
We note that these morphisms correspond to \textit{strict} functors, since being an algebra homomorphism amounts to preserving the coherence operations strictly. For a way to obtain weak homomorphisms, see the work of Garner~\cite{garner_homomorphisms_2010}. We will also write
\begin{align*}
        F^w  : \Glob & \to \omega\Cat, & U^w  : \omega\Cat & \to \Glob
\end{align*}
for the \emph{free $\omega$-category} and \emph{underlying globular set} functors and
\[K^w : \Comp\to \omega\Cat\]
for the comparison functor making the following triangle commute in both directions:
\[\begin{tikzcd}
                \Comp && \omega\Cat \\
                & \Glob
                \arrow["\Free", shift left=2, from=2-2, to=1-1]
                \arrow["\Cell", from=1-1, to=2-2]
                \arrow["K^w", from=1-1, to=1-3]
                \arrow["F^w", from=2-2, to=1-3]
                \arrow["U^w", shift right=-2, from=1-3, to=2-2]
        \end{tikzcd}\]

We proceed to show that $\omega\Cat$ is a locally finitely presentable category, and that homomorphisms of computads and free $\omega$-categories coincide.

\begin{proposition}
        The monad $\fc^w$ is finitary, i.e. it preserves directed colimits.
\end{proposition}
\begin{proof}
        Using the notation of Section~\ref{sec:well-foundedness}, we have for every $n\in \N$, every $B$ Batanin tree of dimension at most $n$ and $0\le d \le \omega+1$, a functor
        \[ C_{d,n}^B \cong \Comp_n^{(d)}(\Free_n B, \Free_n -) : \Glob \to \Set \]
        sending a globular set $X$ to the subset of homomorphisms $\sigma : \Free B\to \Free X$ satisfying that $U_n\sigma$ has depth at most $d$. It suffices to show that those functors are finitary, since $\fc^w_n \cong C^{D_n}_{n,\omega+1}$. To do so, we proceed by induction on $n$ and then induction on $d$.

        Suppose first that $n = 0$ and hence $B = D_0$. By convention all $0$\-cells have depth $1$ being generators, so $C^{D_0}_{n,d}$ is constant on the empty set for $d < 2$. Otherwise, it is isomorphic to the functor sending a globular set $X$ to the set $X_0$. In either case, it is cocontinuous.

        Fix now $n > 0$ and let $B$ a tree of dimension at most $n$. The functor $C^B_{n,0}$ is constant on the empty set, so it is cocontinuous. The functor $C^B_{n,\omega}$ is the colimit of $C^B_{n,d}$ for finite $d$, so it finitary by commutativity of colimits with colimits. Therefore, it remains to show that $C^B_{n,d+1}$ is finitary for $0 \le d \le \omega$. Letting first $B = D_n$, we have that $C^{D_n}_{n,1}$ is constant on the empty set, while for $d>0$, we have that
        \[
          C^{D_n}_{n,d+1}(X) \cong \Cell_n^{(d)}(\Free_n X) \cong X_n +\coprod_{\substack{\dim B \le n \\ A\in \operatorname{Full}_{n-1}(B)}} C_{n,d}^B(X)
        \]
        which is finitary, since colimits commute with colimits, the functors $C_{n,d}^B$ are finitary and the functor $X\mapsto X_n$ is cocontinuous. Letting then $B = D_k$ for $k < n$,
        we have that
        \[
          C^{D_k}_{n,d+1} \cong C^{D_k}_{n-1,\omega+1}
        \]
        as $D_k$ has no $n$\-positions, so $C^{D_k}_{n,d+1}$ is finitary as well. Finally, for $B$  arbitrary, we can easily see that the isomorphism
        \[\Comp\,(\Free B,\Free-) \cong \lim_{\substack{k\le n, \\p\in \Pos_k\,(B)}}\Comp\,(\mathbb{D}^k,\Free-)\]
        given by the density lemma for presheaves and cocontinuity of $\Free$ restricts to an isomorphism
        \[C^B_{n,d+1} \cong \lim_{k,p}C^{D_k}_{n,d+1}.\]
        Since connected limits and filtered colimits of sets commute, this implies that $C^B_{n,d+1}$ is also finitary.
\end{proof}

\begin{corollary}
        The category $\omega\Cat$ is locally finitely presentable.
\end{corollary}
\begin{proof}
        It is the category of algebras of a finitary monad on a locally finitely presentable category \cite[Remark~2.76]{adamek_1994}.
\end{proof}

\begin{proposition}\label{morphisms-are-functors}
        The comparison functor $K^w : \Comp\to \omega\Cat$ is fully faithful.
\end{proposition}
\begin{proof}
        Let $C$ and $D$ be computads and $f : K^wC \to K^wD$ a homomorphism. We will show by induction on $n\le \omega$ that there exists a unique $n$-homomorphism $\sigma_n : C_n\to D_n$ satisfying for every finite $k\le n$ and every $k$-cell $c\in\Cell_k(C_n)$ that
        \[ \Cell_k(\sigma_n)(c) = f(c). \]
        Then $\sigma_\omega$ must be the unique homomorphism such that $K^w\sigma_\omega = f$.

        For $n=0$, the equation above is satisfied only by the function $\sigma_0 = f_0$. Assuming then that unique $\sigma_n$ exists for some $n\in \omega$, define $\sigma_{n+1}$ to consist of $\sigma_n$ and the composite
        \[ V_{n+1}^C\xrightarrow{\var} \Cell_{n+1}(C)\xrightarrow{f_{n+1}} \Cell_{n+1}(D).\]
        This is the only morphism satisfying the equation above for cells of dimension at most $n$ and for $(n+1)$-dimensional generators. Given a coherence cell ${c = \coh\,(B,A,\tau)}$ of dimension $n+1$, let $\tau^\dagger:B\to \Cell\,(C)$ the morphism of globular sets given by $p\mapsto \tau_V(p)$. Then
        \[ \tau = \varepsilon_{n+1,C} \circ \Free_{n+1}(\tau^\dagger) \]
        and we may assume by structural induction on cells that
        \[ \Cell(\Sk_{n+1}\sigma_{n+1}) \circ \tau^\dagger = f \tau^\dagger. \]
        Since $f$ is a homomorphism of $\omega$-categories, we may also observe that
        \begin{align*}
                f(c) & = f\circ \Cell(\varepsilon_C) (\coh\,(B,A,\Free_{n+1}\tau^\dagger))         \\
                     & = \Cell(\varepsilon_D \circ \Free\, f)(\coh\,(B,A,\Free_{n+1}\tau^\dagger)) \\
                     & = \coh\,(B,A, \varepsilon_{n+1,D} \circ \Free_{n+1}(f \tau^\dagger)).
        \end{align*}
        Combining the last two equations and using naturality of $\varepsilon_C$, we see that
        \[f(c) = \Cell_{n+1}(\sigma_{n+1})(c),\]
        so $\sigma_{n+1}$ satisfies the claimed equation. Finally, we let $\sigma_\omega = (\sigma_n) : C\to D$.
\end{proof}

\begin{remark}
        Since the functor $K^w$ is injective-on-objects and fully faithful, by abuse of notation, we will treat $K^w$ as a subcategory inclusion, identifying a computad with the $\omega$-category it generates.
\end{remark}

\begin{remark}
        By initiality of the Kleisli adjunction, there exists also a second
        fully faithful comparison functor \(L^w : \operatorname{KL}^w \to \Comp\) from the Kleisli category of \(\fc^w\) to \(\Comp\), giving an
        (identity-on-objects, fully faithful) factorisation of the functor
        \(\Free\). In Section~\ref{sec:var-to-var}, we will define a (fully faithful, identity-on-objects) factorisation of the functor \(\Free\)
        via the subcategory \(\Comp^{\var}\) defined there.
\end{remark}

\paragraph{Composition operations in $\omega$-categories.}
Let $B$ be a Batanin tree. In a strict $\omega$-category, $B$-shaped diagrams of cells can be composed (uniquely) into a single cell of dimension $\dim\,B$, or more generally into a cell of any dimension $n\ge \dim\,B$ by also using identity operations. For weak $\omega$-categories, we will now also define for every tree $B$ and $n \ge \dim\,B$ an $n$-cell $\comp^B_n$ of $B$  with a similar interpretation. A more general discussion of operations in $\omega$-categories in terms of globular operads and theories follows in Section~\ref{sec:globular-theories}. 

Intuitively, the cell corresponding to the unbiased composition over $B$ is the coherence cell filling the sphere determined by the unbiased composite of the source and target of this diagram. Formally, we define inductively for every $n\in \N$ and Batanin tree $B$ a full $n$-sphere $A_{B,n}$, and when $n\ge \dim\,B$ an $n$-cell $\comp_n^B$. In the base case, the only Batanin tree of dimension $0$ is $D_0$ and we let
\[  \comp^{D_0}_0 = \obpos \]
its unique cell. We let then for every Batanin tree $B$,
\[  A_{B,0} = (s_0^B(\comp_0^{\partial_0 B}), t_0^B(\comp_0^{\partial_0 B})) \]
be its unique full $0$-sphere.


Let now $n\in \N$ and suppose that the full spheres $A_{B,n}$ have been defined for every Batanin tree.
Then we may let
\[\comp^{D_{n+1}}_{n+1} = \var\,(\inl^{n+1}(\obpos)) \]
the unique top dimensional position of $D_{n+1}$ and for every other Batanin tree $B$ of dimension at most $n+1$,
\[\comp^{B}_{n+1} = \coh\,(B, A_{B,n}, \id).\]
In both cases, we have constructed a cell with boundary $A_{B,n}$, since the source and target of the top dimensional position of $D_{n+1}$ are its unique source and target $n$-positions. Finally, we define for arbitrary Batanin tree, the sphere $A_{B,n+1}$ by
\begin{align*}
        \pr_1A_{B,n+1} = \Cell_{n+1}\Free_{n+1}(s^B_{n+1})(\comp^{\partial_{n+1}B}_{n+1}), \\
        \pr_2A_{B,n+1} = \Cell_{n+1}\Free_{n+1}(s^B_{n+1})(\comp^{\partial_{n+1}B}_{n+1}).
\end{align*}
This sphere is full by Proposition \ref{prop:boundary-positions-generation}.

As a particular example, binary composites of $(n+k+1)$-cells along a common $k$-dimensional boundary correspond to the Batanin tree $B_{k,n}$ inductively defined by
$B_{0,n} = \br\,[D_n, D_n]$ and $B_{k+1,n} = \br\,[B_{k,n}]$. Note that this notion of composition operation implicitly also contains identity operations. Namely, the source and target of the cell $\comp^{D_n}_{n+1}$ are both the top dimensional position of $D_n$, allowing the operation corresponding to this cell to be interpreted as an identity operation.

\section{The generator-preserving subcategory}\label{sec:var-to-var}

An important class of morphisms between computads is given by those maps which ``preserve generating data''.

\begin{definition}
        A morphism of $n$-computads $\sigma : C \to D$ is said to be \emph{generator-preserving} when it satisfies the following inductively defined property:
        \begin{itemize}
                \item When $n = 0$, every morphism is generator-preserving.

                \item When $n > 0$ is finite, a morphism $\sigma$ is generator-preserving when $\sigma_{n-1}$ is generator-preserving and for each generator $v \in V_C$, we have that $\sigma_V(v) = \var(w)$ for some generator $w \in V_D$.

                \item When $n = \omega$, a morphism is generator-preserving when each of its projections, $\sigma_k$, is generator-preserving.
        \end{itemize}
\end{definition}

It follows that a generator-preserving morphism, $\sigma$, determines and is uniquely determined by a unique assignment
\[
        \sigma'_k : V^C_k \longrightarrow V^D_k
\]
from generators of $C$ to generators of $D$ for every finite $k\le n$. It is easily verified that the class of generator-preserving morphisms contains identities and is closed under composition.

\begin{definition}
        For each $n$, we define $\Comp_n^{\var}$ to be the  subcategory of $\Comp_n$ whose morphisms are generator-preserving. We write  $\zeta_n : \Comp_n^{\var} \hookrightarrow \Comp_n$ for the inclusion.
\end{definition}

\begin{remark}\label{generator-preserving-two-out-of-three}
        Suppose that we have the following commutative triangle of morphisms of computads:
        \[\begin{tikzcd}
                        & D \\
                        C && E
                        \arrow["\rho"', from=2-1, to=2-3]
                        \arrow["\sigma", from=2-1, to=1-2]
                        \arrow["\tau", from=1-2, to=2-3]
                \end{tikzcd}\]
        where $\rho$ is a generator-preserving $n$-homomorphism.
        If a generator of $C$ were sent to a coherence cell by $\sigma$, the same would be true for $\rho$, so $\sigma$ must also be a generator-preserving morphism.
        This cancellation property implies in particular that all isomorphisms of $n$-computads must be generator-preserving.
\end{remark}

\begin{remark}
        The forgetful functors $u_n$ and $U_n$ can be restricted to give functors
        \begin{align*}
                u_n^{\var} & : \Comp^{\var}_n \to \Comp^{\var}_{n-1} \\
                U_n^{\var} & : \Comp^{\var} \to \Comp_n^{\var}
        \end{align*}
        between the subcategories of generator-preserving homomorphisms, so that the functors $U_n^{\var}$ form a limit cone above the diagram
        \[\begin{tikzcd}
                        \cdots \ar[r, "u^{\var}_{n+1}"] & \Comp^{\var}_n \ar[r, "u^{\var}_n"]&\cdots \ar[r, "u^{\var}_2"] & \Comp^{\var}_1 \ar[r, "u^{\var}_1"] & \Comp^{\var}_0.
                \end{tikzcd}\]
\end{remark}

\begin{example}
        Whenever $\sigma : C\to D$ is a generator-preserving $n$-homomorphism, both $\sk_n\sigma$ and $\Sk_n\sigma$ are generator-preserving homomorphisms, so we may restrict the two functors to get functors
        \begin{align*}
                \sk^{\var}_n & : \Comp_n^{\var} \to \Comp_{n+1}^{\var} \\
                \Sk^{\var}_n & : \Comp_n^{\var} \to \Comp^{\var}
        \end{align*}
        left adjoint to $u_{n+1}^{\var}$ and $U_n^{\var}$ respectively. The functors $u_{n+1}^{\var}$ and $U_{n}^{\var}$ further admit right adjoints
        \begin{align*}
                \cosk^{\var}_n & : \Comp_n^{\var} \to \Comp_{n+1}^{\var} \\
                \Cosk^{\var}_n & : \Comp_n^{\var} \to \Comp^{\var}.
        \end{align*}
        The former extends an $n$-computad $C$ by its set of $n$-spheres and an $n$-homomorphism $\sigma : C\to D$ by its action on $n$-spheres
        \begin{align*}
                \cosk^{\var}_n(C) & = (C, \Type_n(C), \id), & \cosk^{\var}_n(\sigma) = (\sigma,\var \circ\Type_n(\sigma)).
        \end{align*}
        The latter is defined similarly to $\Sk_n$ by
        \[ U_m^{\var}\circ\Cosk_n^{\var} =
                \begin{cases}
                        u_{m+1}^{\var}\cdots u_n^{\var},         & \text{if }m\le n \\
                        \cosk_{m-1}^{\var}\cdots \cosk_n^{\var}, & \text{if }m>n    \\
                \end{cases}
        \]
\end{example}

\begin{example}
        The image of the free computad functor $\Free : \Glob \to \Comp$ is contained in $\Comp_n^{\var}$. Thus, factoring through $\zeta_n$, we obtain a functor
        \[
                \Free^{\var} : \Glob \to \Comp^{\var}.
        \]
        It is easily seen that $\Free^{\var}$ is fully faithful. Let $X$ and $Y$ be two globular sets and $\sigma = (\sigma_n) : \Free X \to \Free Y$ a generator-preserving homomorphism. Then there exists for every $n\in \N$ a unique function $\sigma_{n,V}' : X_n \to Y_n$ such that $\sigma_0 = \sigma_{0,V}'$ and for every $n\ge 1$ and $x\in X_n$,
        \[\sigma_{n,V}(x) = \var(\sigma_{n,V}' x).\]
        For every such $x\in X_n$, we easily compute that
        \begin{align*}
                \pr_1(\Type_{n-1}(\sigma_{n-1})(\phi_n^{\Free X} x)) & = \var(\sigma_{n-1,V}'(\src\, x))   \\
                \pr_1(\ty_{n}(\sigma_{n,V}(x)))                      & = \var\, (\src\, (\sigma_{n,V}'x))  \\
                \pr_2(\Type_{n-1}(\sigma_{n-1})(\phi_n^{\Free X} x)) & = \var(\sigma_{n-1,V}'(\tgt\, x))   \\
                \pr_2(\ty_{n}(\sigma_{n,V}(x)))                      & = \var\, (\tgt\, (\sigma_{n,V}'x)).
        \end{align*}
        On the other hand, we also know that
        \[\ty_n(\sigma_{n,V}(x)) = \Type_{n-1}(\sigma_{n-1})(\phi_n^{\Free X} x),\]
        since $\sigma_n$ is an $n$-homomorphism, so
        \begin{align*}
                \src \, (\sigma_{n,V}' x) & = \sigma_{n-1,V}' (\src\, x) &
                \tgt \, (\sigma_{n,V}' x) & = \sigma_{n-1,V}' (\tgt\, x)
        \end{align*}
        In other words, the functions $\sigma_{n,V}'$ constitute a morphism of $\sigma' : X\to Y$, which one may easily check is the unique morphism satisfying that $\sigma = \Free(\sigma')$.
\end{example}

\begin{proposition}\label{prop-comp-cocomplete}
        For each $n$, the category $\Comp_n^{\var}$ is cocomplete. Furthermore, the inclusion $\zeta_n : \Comp_n^{\var} \to \Comp_n$ preserves all colimits.
\end{proposition}
\begin{proof}
        We prove the statement by induction on $n$.
        The statement holds trivially for $n=-1$. Suppose that $n\in\N$. Given a diagram $J : \mathcal{D} \to \Comp^n_{\var}$, we set:
        \begin{align*}
                (\colim\, J)_{n-1} & = \colim\nolimits_{A \in D}\,{(J A)_{n-1}}, &
                V^{\colim\, J}_n   & = \colim\nolimits_{A \in D} V^{J A}_n
        \end{align*}
        and we define the boundary function $\phi^{\colim\, J}_n : V^{\colim\, J}_n \to \Type_{n-1}(\colim\, J)$ to be the following canonical morphism:
        \[
                V^{\colim\,J}_n = \colim\, V_n^{J}
                \longrightarrow
                \colim\, \Type_{n-1}(u_{n-1} J)
                \longrightarrow
                \Type_{n-1}(\colim J)_{n-1}
        \]
        The universal property of this $n$-computad in both $\Comp_n$ and $\Comp_n^{\var}$ is now easily verified. When $n = \omega$, we define colimits pointwise.
\end{proof}

\begin{lemma}
        For each $n$, the category $\Comp_n^{\var}$ has a terminal object $\top_n$.
\end{lemma}
\begin{proof}
        We proceed by induction on $n$. The result is trivial when $n =- 1$. Hence suppose that $n\in \N$ and that $\Comp_{n-1}^{\var}$ has a terminal object $\top_{n-1}$. Then since $\cosk^{\var}_{n-1} : \Comp_{n-1}^{\var} \to \Comp_n^{\var}$ is a right adjoint, it preserves limits and so we may define $\top_n = \cosk^{\var}_{n-1}(\top_{n-1})$. Explicitly $(\top_n)_{n-1} = \top_{n-1}$, and $\top_n$ has a unique $n$-cell for each $(n-1)$-sphere in $\top_{n-1}$. Finally, we define $\top_{\omega}$ to be the computad $\Cosk^{\var}_0(\top_0) =(\top_0, \ldots, \top_n, \ldots)$.
\end{proof}

\subsection{\texorpdfstring{$\Comp^{\var}$}{Comp\^var} is a presheaf topos}

It turns out that $\Comp^{\var}$ is extremely well-behaved; in fact we will prove that it is a presheaf category. This is consistent with work of Batanin~\cite{batanin_computads_2002}, although our approach is significantly different, as discussed in the Introduction. This presheaf property is striking since it is well known that the category of strict computads and generating data preserving maps is not a presheaf topos~\cite{makkai_category_2008}. Thus, this result can be seen as evidence that weak higher categories are, in a sense, better behaved than strict higher categories.

The following construction will play a pivotal role in this section. Here we make reference to the \textit{category of elements} or \emph{Grothendieck construction} of a presheaf $X : \mathcal{C}^{\op} \to \Set$. It is the category $\el X$ with class of objects the disjoint union of the sets $X(c)$ for $c\in \mathcal{C}$, and morphisms $f : (c,x) \to (c',x')$ the morphisms $f : c\to c'$ such that $X(f)(x') = x$.

\begin{theorem}\label{lem-Cell-is-familially-representable}
        Let $\Cell^{\var}$ and $\Type^{\var}$ be the restriction of the cell and sphere functors to the generator-preserving subcategory, as follows:
        \begin{align*}
                \Cell^{\var} & = \Cell \circ\, \zeta : \Comp^{\var} \to \Glob
                \\
                \Type^{\var} & = \Type \circ\, \zeta : \Comp_n^{\var} \to \Glob
        \end{align*}
        Then these functors are familially representable; there exist globular sets $\mathfrak{C}$, and $\mathfrak{T}$, together with functors from their category of elements
        \begin{align*}
                J^\mathfrak{C} : \el\mathfrak{C} \to \Comp^{\var}
                 &
                \qquad
                 &
                J^\mathfrak{T} : \el\mathfrak{T} \to \Comp^{\var}
        \end{align*}
        such that we have isomorphisms, natural in $C \in \Comp^{\var}$ and $n \in \mathbb{G}$,
        \begin{align*}
                \Cell^{\var}_n (C) & \cong \coprod_{c_0 \in \mathfrak{C}_n} \Comp^{\var}(J^\mathfrak{C}(c_0),C)   \\
                \Type^{\var}_n (C) & \cong \coprod_{A_0 \in \mathfrak{T}_n} \Comp^{\var}(J^\mathfrak{T}(A_0), C).
        \end{align*}
\end{theorem}
\begin{proof}
        We will show more generally that the functor giving the $k$-cells of an $n$-computad and the functor giving the  $k$-spheres for $k\le n$ are familially representable. Let $\mathfrak{C}$ and $\mathfrak{T}$ the globular sets of cells and spheres of the terminal computad $\top_{\omega}$ respectively, and define for $c_0\in \mathfrak{C}_k$ and $A_0\in \mathfrak{T}_k$ the set \emph{of cells of shape $c_0$} of an $n$-computad $C$ and that \emph{of spheres of shape $A_0$} by
        \begin{align*}
                \Cell\,(C; c_0) & = \{ c\in \Cell_k(C) \,:\, \Cell(!)(c) = c_0.\} \\
                \Type\,(C; A_0) & = \{ A\in \Type_k(C) \,:\, \Type(!)(A) = A_0.\}
        \end{align*}
        where $!$ denotes the unique generator-preserving homomorphism to $\top_{n}$. It is easy to see that generator-preserving homomorphisms preserve the shape of cells and spheres, so both definitions give rise to functors $\Comp_n^{\var} \to \Set$, for which we have natural isomorphisms,
        \begin{align*}
                \Cell^{\var}_k(\Sk_nC) & = \coprod_{c_0\in \mathfrak{C}_k} \Cell\,(C; c_0), \\
                \Type^{\var}_k(\Sk_nC) & = \coprod_{A_0\in \mathfrak{T}_k} \Type\,(C; A_0).
        \end{align*}
        The natural transformations $\pr_i$ for $i=1,2$, and the natural transformations $\ty$, $\src$ and $\tgt$ for $k>0$ may be restricted to natural transformations that we denote by the same names
        \begin{align*}
                \pr_i & : \Type\,(-;A_0) \Rightarrow \Cell\,(-;\pr_i A_0)  \\
                \ty   & : \Cell\,(-;c_0) \Rightarrow \Type\,(-;\ty_k c_0), \\
                \src  & : \Cell\,(-;c_0) \Rightarrow \Cell\,(-;\src c_0),  \\
                \tgt  & : \Cell\,(-;c_0) \Rightarrow \Cell\,(-;\tgt c_0).
        \end{align*}

        To prove the theorem, it suffices to show that the functors of cells of shape $c_0$ and spheres of shape $A_0$ are representable. Hence, we will construct by induction on $n\le \omega$, for every finite $k\le n$
        \begin{enumerate}
                \item for all $c_0\in \mathfrak{C}_k$, an $n$-computad $J_n^{\mathfrak{C}}(c_0)$ and a canonical cell $\self_n(c_0)$ of shape $c_0$ of it such that evaluation at the canonical cell induces an isomorphism
                      \[ \Comp_n^{\var}(J_n^{\mathfrak{C}}(c_0),-) \Rightarrow \Cell\,(-;c_0)   \]
                \item for all $A_0\in \mathfrak{T}_k$, an $n$-computad $J_n^{\mathfrak{T}}(A_0)$ and a canonical sphere $\self_n(A_0)$ of shape $A_0$ of it such that evaluation at the canonical sphere induces an isomorphism
                      \[ \Comp_n^{\var}(J_n^{\mathfrak{T}}(A_0),-) \Rightarrow \Type(-;A_0)  \]
                \item a pair of generator-preserving $n$-homomorphisms for all $A_0 \in \mathfrak{T}_k$,
                      \[
                              J_{n,A_0,i} : J_{n}^{\mathfrak{C}}(\pr_i A_0) \to J_n^{\mathfrak{T}}(A_0)
                      \]
                      such that
                      \[
                              \Cell\,(J_{n,A,i})(\self_n \pr_i A_0) = \pr_i (\self_n A_0)
                      \]
                \item for $k>0$ and all $c_0 \in \mathfrak{T}_k$ generator-preserving $n$-homomorphisms
                      \begin{align*}
                              J_{n,\partial c_0} : J_{n}^{\mathfrak{T}}(\ty_k c_0) \to J_n^{\mathfrak{C}}(c_0) \\
                              J_{n, c_0, \src} : J_{n}^{\mathfrak{C}}(\src c_0) \to J_n^{\mathfrak{C}}(c_0)    \\
                              J_{n, c_0, \tgt} : J_{n}^{\mathfrak{C}}(\tgt c_0) \to J_n^{\mathfrak{C}}(c_0)
                      \end{align*}
                      such that
                      \begin{align*}
                              \Type\,(J_{n,\partial c_0})(\self_n (\ty_k c_0)) = \ty_k (\self_n c_0) \\
                              \Cell\,(J_{n, c_0, \src})(\self_n (\src c_0)) = \src (\self_n c_0)     \\
                              \Cell\,(J_{n, c_0, \tgt})(\self_n (\tgt c_0)) = \tgt (\self_n c_0)
                      \end{align*}
        \end{enumerate}
        and we will show that
        \begin{enumerate}
                \item for $k>0$ and $A_0 = (A_0',a,b)$ the following square commutes
                      \[
                              \begin{tikzcd}
                                      J^{\mathfrak{T}}_{n}(A_0')
                                      \ar[r, "J_{n,\partial a}"]
                                      \ar[d, "J_{n,\partial b}" left]
                                      &
                                      J^{\mathfrak{C}}_n(a)
                                      \ar[d, "J_{n,A_0,1}"]
                                      \\
                                      J^{\mathfrak{C}}_n(b)
                                      \ar[r, "J_{n,A_0,2}" below]
                                      &
                                      J^{\mathfrak{T}}_n(A_0)
                              \end{tikzcd}
                      \]
                \item for $k>0$ and $c_0\in \mathfrak{C}_k$ the $n$-homomorphisms associated to the source and target natural transformations are determined by
                      \begin{align*}
                              J_{n,c_0,\src} & = J_{n, \ty_k c_0, 1}J_{n,\partial_{c_0}} \\
                              J_{n,c_0,\tgt} & = J_{n, \ty_k c_0, 2}J_{n,\partial_{c_0}} \\
                      \end{align*}
                      and hence they satisfy for $k>2$ the globularity conditions
                      \begin{align*}
                              J_{n,c_0,\src}J_{n,\src c_0, \src} & = J_{n,c_0,\tgt}J_{n,\tgt c_0, \src}  \\
                              J_{n,c_0,\src}J_{n,\src c_0, \tgt} & = J_{n,c_0,\tgt}J_{n,\tgt c_0, \tgt}.
                      \end{align*}
                \item for $k<n<\omega$ and every $c_0\in \mathfrak{C}_k$ and $A_0 \in \mathfrak{T}_k$, the following compatibility conditions hold
                      \begin{align*}
                              J_{n}^\mathfrak{C}(c_0) & = \sk_n J_{n-1}^\mathfrak{C}(c_0), & \self_n(c_0) = \self_{n-1}(c_0)  \\
                              J_{n}^\mathfrak{T}(A_0) & = \sk_n J_{n-1}^\mathfrak{T}(A_0). & \self_n(A_0) = \self_{n-1}(A_0),
                      \end{align*}
                      while for $n = \omega$,
                      \begin{align*}
                              J^\mathfrak{C}(c_0) & = \Sk_k J_{k}^\mathfrak{C}(c_0), & \self\,(c_0) = \self_k(c_0)  \\
                              J^\mathfrak{T}(A_0) & = \Sk_k J_{k}^\mathfrak{T}(A_0). & \self\,(A_0) = \self_k(A_0).
                      \end{align*}
        \end{enumerate}

        \noindent It is worth noting that if representability of $\Cell\,(-;c_0)$ is established, as well as that of $\Cell\,(-;\src c_0)$, $\Cell\,(-;\tgt c_0)$ and $\Type(-;\ty_k c_0)$, then existence and uniqueness of $J_{n,\partial c_0}$, $J_{n,c_0,\src}$, $J_{n,c_0,\tgt}$ and $J_{n,\ty c_0, i}$ with the given properties is a consequence of the Yoneda Lemma. Their uniqueness implies automatically that the various equations between them hold, since corresponding equations hold for the natural transformations they define. In light of that, it suffices to define the representations of each functor using those homomorphisms wherever needed.

        \paragraph{Base case.}
        First suppose that $n = 0$. Then there is a unique $0$-cell $\star\ \in \mathfrak{C}_0$, and we define
        \[
                J^{\mathfrak{C}}_0(\star) = \top_0 = \{\star\} \in \Comp_0^{\var}
        \]
        We set
        \[
                \self_0 \star = \star.
        \]
        There is a unique $0$-spheres $(\star, \star)$ in $\mathfrak{T}_0$, and we define
        $J^{\mathfrak{T}}_0(\star, \star)$ to be the coproduct
        \[
                J^{\mathfrak{C}}_0(\star) + J^{\mathfrak{C}}_0(\star) = \{\star_0, \star_1\},
        \]
        and we set
        \[
                \self_0 {(\star, \star)} = (\star_0, \star_1).
        \]
        We define also $J_{0,\star,i}$ to be the function picking $\star_i$. The functor of cells of shape $\star$ and that of spheres of shape $(\star_0,\star_1)$ are represented by those sets, since they are isomorphic to the identity and the power $(-)^2$ respectively.

        \paragraph{Low dimensional cells and spheres.}
        Let now $n>0$ and suppose that the representations above have been defined for all $n'<n$. For $k<n$ and $c_0\in \mathfrak{C}_k$, define the computad $J_{n}^{\mathfrak{C}}(c_0)$ and the cell $\self_n(c_0)$ so that the compatibility condition is satisfied. For finite $n$, this represents the functor $\Cell\,(-;c_0)$ on $n$-computads as well, since for every $n$-computad $C$ the following square commutes
        \[\begin{tikzcd}[column sep = large]
                        {\Comp_n^{\var}(J_n^{\mathfrak{C}}(c_0),C)} & {\Cell\,(C;c_0)} \\
                        {\Comp_n^{\var}(J_{n-1}^{\mathfrak{C}}(c_0),u_nC)} & {\Cell\,(u_nC;c_0)}
                        \arrow["{\self_n(c_0)}", from=1-1, to=1-2]
                        \arrow[equals, from=1-2, to=2-2]
                        \arrow["{u_n^{\var}}"', "\sim", from=1-1, to=2-1]
                        \arrow["{\self_{n-1}(c_0)}"', "\sim", from=2-1, to=2-2]
                \end{tikzcd}\]
        For $n = \omega$, there is a similar commutative square involving $U_k^{\var}$, proving that the defined computad represents $\Cell\,(-;c_0)$ again. For $A_0 \in \mathcal{T}_k$, the computad and the spheres given by the compatibility condition represent $\Type(-;A_0)$ for the same reason. It is easy to see that the canonical $n$-homomorphisms between the representing computads are also obtained by applying $\sk^{\var}_n$ or $\Sk_k^{\var}$ to the corresponding $(n-1)$-homomorphisms or $k$-homomorphisms respectively.

        \paragraph{Top dimensional cells.}
        Suppose further that $n < \omega$. The computad $J_n^{\mathfrak{C}}(c_0)$ for $c_0\in \mathfrak{C}_n$ together with the canonical cell $\self_n(c_0)$ will be constructed by induction on $c_0$ and it will be shown that they represent $\Cell\,(-;c_0)$.

        If $c_0 = \var\, A_0$ for some $A_0\in V^{\top_{\omega}}_n = \mathfrak{T}_{n-1}$, then let
        \[ J_n^{\mathfrak{C}}(c_0) = (J_{n-1}^{\mathfrak{T}}(A_0), \{\star\}, \{(\star\mapsto \self_n A_0)\})\]
        and let
        \[\self_n(c_0) = \var\,\star.\]
        A generator-preserving morphism $\sigma : J_n^{\mathfrak{C}}(c_0)\to C$ amounts to a sphere $A$ of $C$ of shape $A_0$, and a generator $v\in V_n^C$ of spheres $A$.
        A cell $c\in \Cell_n(C)$ is of shape $c_0$ exactly when it is of the form $\var v'$ for some generator such that
        \[!_V(v') = \var\,(\Type(!)\ty_{n,C}(\var\, v')) = A_0,\]
        so they are in bijection to pairs $(v,A)$ as above. Therefore, $J_n^{\mathfrak{C}}(c_0)$ and $\self_n(c_0)$ represent the functor of cells of shape $c_0$.

        Suppose now that $c_0 = \coh(B,A, \tau)$ for some Batanin tree $B$, full $(n{-}1)$\-sphere $A$ and morphism $\tau: \Free_n B\to \top_n$ and that for all $c_0'\in \mathfrak{C}_k$ of depth less than that of $c_0$, the representation of $\Cell(-;c_0')$ has been defined. Let $\el^*(\mathfrak{C})$ the full subcategory of the category of elements of $\mathfrak{C}$ consisting of cells of dimension at most $n-1$, and cells of dimension $n$ and depth less than that of $c_0$. Then the representing objects assemble into a functor
        \[J_n^{\mathfrak{C}} : \el^*(\mathfrak{C}) \to \Comp_n^{\var} \]
        together with the $n$-homomorphisms $J_{n,c_0',\src}$ and $J_{n,c_0',\tgt}$.

        The Batanin tree $B$ must have dimension at most $n$, so transposing $\tau$ along $\Sk_{n+1} \vdash U_{n+1}$ and $\Free \vdash \Cell$, we get a unique morphism of globular sets
        \[\tau^\dagger : B \to \mathfrak{C}\]
        such that for all $k\le n$ and $k$-position $p$ of $B$,
        \[\tau^{\dagger}(p) = \tau_{k,V}(p).\]
        The image of the functor that $\tau^\dagger$ induces between the categories of elements is contained in $\el^*(\mathfrak{C})$, so we may define
        \[J_{n}^{\mathfrak{C}}(c_0) = \colim (J_n^{\mathfrak{C}} \circ \el\,\tau^\dagger)\]
        and define for every position $p$ of $B$, $n$-homomorphisms
        \[\psi_p  : J_n^{\mathfrak{C}}(\tau^\dagger\,p) \to J_n^{\mathfrak{C}}(c_0) \]
        forming a colimit cocone. The assignment
        \[\sigma : B \to \Cell(\Sk_nJ_n^{\mathfrak{C}}(c_0)) \]
        given by sending $p$ to
        \[\sigma(p) = \Cell(\Sk_n\psi_p)(\self_n \tau^\dagger p) \]
        is easily checked to be a morphism of globular sets.
        Transposing $\sigma$ as above, we get an $n$-homomorphism
        \[\sigma^\dagger : \Free_n B \to J_n^{\mathfrak{C}}(c_0)\]
        and we let
        \[ \self_{n}(c_0) = \coh(B,A,\sigma^\dagger). \]

        It remains to show that this pair represents $\Cell(-;c_0)$. Generator-preserving $n$-homomorphisms
        \[\chi : J_{n+1}^{\mathfrak{C}}(c_0)\to C\]
        are in bijection to families of generator-preserving $n$-homomorphisms indexed by positions $p$ of $B$
        \[ \chi_p : J_{n+1}(\tau^\dagger p) \to C \]
        satisfying the boundary conditions
        \begin{align*}
                \chi_p\, J_{n,\tau^\dagger p, \src} & = \chi_{\src p}, \\
                \chi_p\, J_{n,\tau^\dagger p, \tgt} & = \chi_{\tgt p}.
        \end{align*}
        Such families are in turn in bijection to families of cells $d_p \in \Cell\,(\Sk_n C)$ such that
        \[\Cell(!)(d_p) = \tau^\dagger\, p\]
        and they satisfy the boundary conditions
        \begin{align*}
                \src\,d_p & = d_{\src p}, \\
                \tgt\,d_p & = d_{\tgt p}.
        \end{align*}
        Equivalently, they are in bijection with morphisms of globular sets
        \[ d : B \to \Cell(\Sk_n C) \]
        such that
        \[\Cell(!)\circ d = \tau^\dagger.\]
        Transposing as above, we see that such morphisms correspond to $n$-homomorphisms
        \[d^\dagger : \Free_n B \to C \]
        such that
        \[\tau = {!} \circ d^\dagger.\]
        Finally, such morphisms correspond to cells of $C$ of shape $c_0$, since any such cell must be of the form $\coh(B,A,d^\dagger)$ for such a morphism.

        To establish that this bijection is the one given by evaluating at $\self_n(c_0)$, we need to show that for all pairs $\chi$ and $d^\dagger$ related by this bijection,
        \[\coh(B,A,d^\dagger) = \Cell(\chi)(\self_n c_0) = \coh(B,A, \chi\sigma^\dagger),\]
        that is $d^\dagger = \chi \sigma^\dagger$. Transposing this equality, it is equivalent to showing that
        \[ d = \Cell(\chi) \sigma : B \to \Cell(C). \]
        For every position $p$,
        \begin{align*}
                d_p & = \Cell(\chi_p)(\self_n \tau^\dagger\, p)            \\
                    & = \Cell(\chi)\Cell(\psi_p)(\self_n \tau^\dagger\, p) \\
                    & = \Cell(\chi)\sigma(p),
        \end{align*}
        so the final equality holds and hence the previous ones as well.

        \paragraph{Top-dimensional spheres.}
        Suppose that we have a sphere $A_0 = (A_0', a, b) \in \mathfrak{T}_n$. Then we define $J^{\mathfrak{T}}_n(A_0)$ and the generator-preserving morphisms $J_{n,A_0,i}$ by the following pushout diagram
        \[
                \begin{tikzcd}
                        J^{\mathfrak{T}}_{n}(A_0')
                        \ar[r, "J_{n,\partial a}"]
                        \ar[d, "J_{n,\partial b}" left]
                        &
                        J^{\mathfrak{C}}_n(a)
                        \ar[d, "J_{n,A_0,1}"]
                        \\
                        J^{\mathfrak{C}}_n(b)
                        \ar[r, "J_{n,A_0,2}" below]
                        &
                        J^{\mathfrak{T}}_n(A_0)
                \end{tikzcd}
        \]
        By construction this square commutes as required. We define the sphere $\self_n(A_0)$ to be the one given by the parallel cells $\Type_n(\Sk_nJ_{n,A_0,i})(\self_n(\pr_i A_0))$. For every $n$-computad $C$, the set of spheres of shape $A_0$ is the pullback of the sets of cells of shapes $a$ and $b$ respectively over the set of spheres of shape $A_0'$. From that, it follows easily that $J^{\mathfrak{T}}_n(A_0)$ with the sphere $\self_n(A_0)$ represent the functor of spheres of shape $A_0$.
\end{proof}

\begin{corollary}\label{cor-compn-preserves-limits}
        The functor $\Type_n \zeta_n : \Comp_n^{\var} \to \Set$ preserves connected limits.
\end{corollary}
\begin{proof}
        This functor is familially representable by the previous proof, so this follows immediately, since coproducts in $\Set$ commute with connected limits.
\end{proof}

\begin{corollary}\label{cor:compn-complete}
        For each $n$, the category $\Comp_n^{\var}$ is complete. Furthermore, the inclusion $\zeta_n : \Comp_n^{\var} \to \Comp_n$ preserves all connected limits.
\end{corollary}
\begin{proof}
        Since $\Comp_n^{\var}$ has a terminal object, it suffices to show it has connected limits.
        Corollary \ref{cor-compn-preserves-limits} now allows us to proceed as we did for colimits.
        We prove the statement by induction on $n$.
        When $n = 0$, we have that  $\Comp_n^{\var} = \Comp_n = \Set$ and so the statement holds trivially. Suppose that $0 < n < \omega$. Given a connected diagram $J : \mathcal{D} \to \Comp_n^{\var}$, we set:
        \begin{align*}
                (\lim\, J)_{n-1} = \lim_{A \in D}{(J A)_{n-1}}
                \\
                V^{\lim\, J}_n = \lim_{A \in D} V^{J A}_n
        \end{align*}
        and we define $\phi^{\lim\,J}_n : V^{\lim J}_n \to \Type_{n-1}(\lim J)$ to be the following canonical morphism:
        \[
                V_n^{\lim\,J} = \lim V^J_n
                \longrightarrow
                \lim \Type_{n-1} J
                \cong
                \Type_{n-1}(\lim J)
        \]
        The universal property of this computad in both $\Comp_n$ and $\Comp_n^{\var}$ is now easily verified. When $n = \omega$, we define connected limits pointwise.
\end{proof}

\begin{corollary}\label{cor:mono-epi-comp-var}
        For all $n$, the inclusion $\zeta_n$ preserves and reflects monomorphisms and epimorphisms.
\end{corollary}
\begin{proof}
        Monomorphisms and epimorphisms can be characterised in terms of pullbacks and pushouts respectively. Since $\zeta$ preserves connected limits and colimits, it must preserve monomorphisms and epimorphisms. Moreover, it must reflect them, since it is faithful.
\end{proof}

\begin{theorem}
        The category $\Comp_n^{\var}$ is equivalent to the category of presheaves on $\mathcal{V}_n$, where $\mathcal{V}_n$ is its full subcategory on the $n$-computads of the form $J_n^{\mathcal{C}}(\var\, A_0)$ for some $-1\le k <n$ and $A_0\in \mathfrak{T}_k$.
\end{theorem}

\begin{proof}
        The category $\Comp_n^{\var}$ together with the functor
        \[V^\bullet = \coprod_{0\le k\le n} V_k^\bullet : \Comp_n^{\var} \to \Set\]
        returning the generators of a computad is a concrete category in the sense of Makkai~\cite{makkai_word_2005}. By the construction of colimits in $\Comp_n^{\var}$, it follows that $V^\bullet$ is cocontinuous. Moreover, it is easy to see that it reflects isomorphisms, i.e. that a generator-preserving morphisms that is bijective on generators is invertible.

        By the proof of the previous theorem, we may further decompose $V^\bullet$ into a coproduct of representable functors by
        \begin{align*}
                V^C &= \coprod_{0\le k\le n}\coprod_{A_0\in \mathfrak{T}_{k-1}}\{v\in V_k^C \text{ whose boundary has shape } A_0 \} \\
                &\cong \coprod_{0\le k\le n}\coprod_{A\in \mathfrak{T}_{k-1}} \Comp_n^{\var}(J_n^{\mathcal{C}}(\var\, A_0),C).
        \end{align*}
        It follows that the category of elements of $V^\bullet$ is the disjoint union of the category of elements of those representable functors. Since the category of elements of a representable functor has an initial object, namely the identity of the representing object, the result follows~\cite[Proposition~2]{makkai_word_2005}.
\end{proof}

\section{Computadic resolutions}\label{sec:replacement}

Every set $I$ of morphisms in a locally presentable category $\mathcal{C}$ generates an algebraic weak factorisation system $(\mathbb{L},\mathbb{R})$ by Garner's algebraic small object argument \cite{garner_small_object}. Restricting  the functorial factorisations to the slice $\emptyset/\mathcal{C}\cong \mathcal{C}$ below the initial object gives a comonad $(Q,r,\Delta)$, that Garner calls the \emph{universal cofibrant replacement} of the weak factorisation system. This comonad may be identified by the following recognition principle.

\begin{theorem}\label{thm:recongnition}\cite[Proposition~2.6]{garner_homomorphisms_2010}
        For every object $X\in \mathcal{C}$, the morphism $r_X : QX\to X$ over $X$ may be equipped with such a choice of liftings $\phi(j,h,k)$ for every lifting problem
        \[\begin{tikzcd}
                        {\dom(j)} & QX \\
                        {\cod(j)} & X
                        \arrow["j"', from=1-1, to=2-1]
                        \arrow["{r_X}", from=1-2, to=2-2]
                        \arrow["k"', from=2-1, to=2-2]
                        \arrow["h", from=1-1, to=1-2]
                        \arrow["{\phi(j,h,k)}"{description}, dashed, from=2-1, to=1-2]
                \end{tikzcd}\]
        where $j\in I$, such that $(r_X,\phi)$ is initial among objects of $\mathcal{C}/X$ equipped with a choice of liftings.
\end{theorem}
\noindent
In particular, the set $I$ of inclusions of spheres as boundaries of disks ${\mathbb{S}^{n-1} \overset{\iota_n}{\rightarrowtail} \mathbb{D}^n}$ generates  a weak factorisation system on $\omega\Cat$. Our goal in this section is to explicitly describe the universal cofibrant replacement in this case using the language of computads.

We will first show that for each computad $C$, the weak $\omega$-category $K^w C$ is always cofibrant; that is the map $\emptyset \to K^w C$ is in the left class of the weak factorisation system generated by $I$. In fact, we will prove that this map is \emph{$I$-cellular}, i.e. in the closure of $I$ under coproducts, pushouts and transfinite composition. We will then build a right adjoint $W : \omega\Cat\to \Comp^{\var}$ to the free $\omega$-category functor $K^w\zeta : \Comp^{\var} \to \omega\Cat$ and show that the induced monad is the universal one, in analogy to \cite[Proposition~5.7]{garner_homomorphisms_2010}.

\subsection{Immersions are cellular}

We now define a simple class of morphisms of computads called \emph{immersions}. We will show that immersions are $I$-cellular when viewed as morphisms of $\omega$-categories, and conclude that free $\omega$-categories on computads are also $I$-cellular.

\begin{definition}
        An \emph{immersion} \mbox{$\sigma: C\to D$} is a generator-preserving homomorphism such that for all $n \in \N$ the assignment on generators $\sigma': V_n^C\to V_n^D$ defining $\sigma$ is injective.
\end{definition}

\noindent Using representability of $\Cell$ and Corollary \ref{cor:mono-epi-comp-var}, we can easily see that a generator-preserving homomorphism is an immersion if and only if it is monic in $\Comp^{\var}$ or equivalently in $\Comp$.

We first show that certain immersions are pushouts of coproducts of maps in $I$. We will then show that every immersion is a transfinite composite of these particular immersions; consequently, every immersion is $I$-cellular. Restricting these results to morphisms out of the initial computad, we will conclude that free $\omega$\-categories on a computad are $I$-cellular and that they satisfy a freeness condition analogous to the one posed by Street \cite[Section~4]{street_orientals} for strict $\omega$\-categories.

\begin{proposition}\label{prop:immersion-pushout}
        Let $\sigma : C \to D$ be an immersion and suppose that there exists some $n\in\N$ such that, for all $k \neq n$, we have that $V_k^C = V_k^D$ and $\sigma_{k,V} = \id$. Let $V = V_n^D\setminus \sigma_{n,V}(V_n^C)$.
        Then, the following square is a pushout square:
        \[\begin{tikzcd}
                        {\coprod_{e \in V} K^w\mathbb{S}^{n-1}} & {\coprod_{e \in V} K^w\mathbb{D}^n} \\
                        {K^wC} & {K^wD}
                        \arrow["{\iota_n}", from=1-1, to=1-2]
                        \arrow[from=1-2, to=2-2]
                        \arrow[from=1-1, to=2-1]
                        \arrow["{K^w\sigma}"', from=2-1, to=2-2]
                \end{tikzcd}\]
        Here the vertical maps classify the generator in $V$ and their spheres respectively.
\end{proposition}

\begin{proof}
        The $\omega$-categories $\mathbb{D}^n$ and $\mathbb{S}^n$ represent $n$-cells and pairs of parallel $n$-cells in respectively. Hence, the proposition amounts to showing that, for each $\omega$-category $X$, each homomorphism $f : C \to X$, and each family of cells ${\{x_e \in X_n\}_{e \in V}}$ satisfying the boundary conditions,
        \begin{align*}
                \src x_e & = f\,(\src \var\, e), &
                \tgt x_e & = f\,(\tgt \var\, e),
        \end{align*}
        there exists a unique homomorphism $g : D\to X$ such that
        \begin{align}
                \label{pushout-in-detail}
                \tag{$\star$}
                f         & = g \circ \sigma, &
                g(\var e) & = x_e.
        \end{align}
        Let $h: \Cell D \to U^w X$ the underlying morphism of globular sets for some such $g$. We claim that $h$ satisfies the following conditions:
        \begin{enumerate}
                \item For each $k\in \N$, $v\in V_k^C$, we have that $h(\var\, v) = f(\var\, v)$. \label{condition1}

                \item For each $e\in V$, we have that $h(\var\, e) = x_e$. \label{condition2}

                \item Let $\alpha : \fc^w X\to X$ be the $\fc^w$-algebra structure of $X$. Then, for each coherence cell, we have that
                      \[
                              h(\coh\,(B,A,\tau)) = \alpha\,(\coh\,(B,A,\Free_{k+1}(h\tau^\dagger))).
                      \]
                      where $\tau^\dagger : B\to \Cell(D)$ is the morphism corresponding by adjointness to $\tau : \Free_{k+1} B \to D$.
                      \label{condition3}
        \end{enumerate}
        The first two conditions are easily verified. For \ref{condition3}, we observe that, since $g$ is a homomorphism, we have that
        \begin{align*}
                h (\coh(B, A, \tau))
                \
                 & = h (\coh(B, A, \varepsilon_{k+1, D} \circ \Free_{k+1} \tau^\dagger))
                \\
                 & = (h \circ \Cell \varepsilon_{D})(\coh(B, A, \Free_{k+1} \tau^\dagger))
                \\
                 & = (\alpha \circ \fc^w h)(\coh(B, A, \varepsilon_{k+1, D} \circ \tau^\dagger))
                \\
                 & = \alpha(\coh(B, A, \Free_{k+1} (h \circ \tau^\dagger)).
        \end{align*}
        These three conditions uniquely determine $g$. Consequently, it suffices to show that the map $h$, defined at each dimension by these conditions, is always a homomorphism of weak $\omega$-categories. This, in turn, is equivalent to the following statements holding for each dimension $k$:
        \begin{itemize}
                \item If $k > 0$, then for each $k$-cell $c$ of $D$, the map $h$ respects the source and target of $c$:
                      \begin{align*}
                              \src h(c) = h(\src c),
                               &  &
                              \tgt h(c) = h(\tgt c).
                      \end{align*}

                \item For each $k$-cell $c'$ of $\Free \Cell D$, we have that
                      \begin{align*}
                              (h \circ \Cell(\varepsilon_D))(c') = (\alpha \circ \fc^w h)(c').
                      \end{align*}
                      Here  $\varepsilon_D$ is the counit of the adjunction $\Cell \dashv \Free$.
        \end{itemize}
        We proceed by induction on $k$. In dimension $0$, this is straightforward. Hence, suppose the claim holds for some $k\in \N$. Suppose that $c$ is a $(k+1)$\-dimensional cell in $D$. If $c$ is a generator, then \ref{condition1} and \ref{condition2} imply that $h$ respects the source and target of $h$. Given a coherence $(k+1)$-cell $c = \coh\,(B,A,\tau)$, the inductive hypothesis implies that
        \begin{align*}
                \src\,h(c)
                 & = \src \alpha (\coh\,(B,A,\Free_{k+1}(h \tau^\dagger)))
                \\
                 & = \alpha \src  (\coh\,(B,A,\Free_{k+1}(h \tau^\dagger)))
                \\
                 & = (\alpha\circ \fc^w(h)) (\src\,\coh\,(B,A,\Free_{k+1}\tau^\dagger))        \\
                 & = (h\circ \Cell(\varepsilon_D)) (\src\,\coh\,(B,A,\Free_{k+1}\tau^\dagger)) \\
                 & = h (\src\,\coh\,(B,A,\varepsilon_{k+1,D}\Free_{k+1}\tau^\dagger))          \\
                 & = h(\src\,c).
        \end{align*}
        A similarly argument works for the target. Now suppose that $c'$ is a $(k+1)$-dimensional cell of $\Free \Cell D$. First suppose that $c' = \var v$. Then the unit laws of $\fc^w$ and $\alpha$ imply that
        \begin{align*}
                (h \circ \Cell(\varepsilon_D)) (\var v)
                = h(v)
                = \alpha(\var h(v))
                =(\alpha\circ \fc^w h)(\var v).
        \end{align*}
        Now suppose that $c' = \coh\,(B,A,\tau_0)$ is a coherence cell. Suppose inductively that $h$ satisfies the inductive hypothesis on cells in the image of the transpose $\tau_0^\dagger : B \to \fc^w \Cell D$. Then
        \begin{align*}
                (h\circ \Cell(\varepsilon_D)) (c')
                \
                 & = h(\coh(B, A, \epsilon_{k+1, D} \tau_0))
                \\
                 & = \alpha(\coh(B, A, \Free_{k+1}(h (\varepsilon_{k+1, D} \circ \tau_0)^\dagger)))
                \\
                 & = \alpha\circ \fc^w (h\circ (\varepsilon_D\tau_0)^\dagger)(\coh\,(B,A,\id))                                                      \\
                 & = \alpha\circ \fc^w (h\circ \Cell(\varepsilon_D)\tau_0^\dagger)(\coh\,(B,A,\id))                                                 \\
                 & \overset{\text{ind}}{=} \alpha\circ \fc^w(\alpha\circ \fc^wh\circ \tau_0^\dagger)(\coh\,(B,A,\id))                               \\
                 & = (\alpha\circ \fc^w\alpha) \circ (\fc^w\fc^w h) (\coh\,(B,A,\Free_{k+1}\tau_0^\dagger))                                         \\
                 & \overset{\text{alg}}{=} (\alpha\circ \Cell(\varepsilon_{\Free\,X})) \circ (\fc^w\fc^w h) (\coh\,(B,A,\Free_{k+1}\tau_0^\dagger)) \\
                 & \overset{\text{nat}}{=} \alpha\circ \Cell(\Free\,h\circ \varepsilon_{\fc^wD})(\coh\,(B,A,\Free_{k+1}\tau_0^\dagger))             \\
                 & = (\alpha\circ \fc^w h)(c').
        \end{align*}
        Thus, the map $h$ satisfies the conditions required of a homomorphism of weak $\omega$-categories.
\end{proof}

\begin{theorem}\label{thm:immersion-trans}
        Every immersion is the transfinite composite of maps satisfying the hypotheses of the previous theorem. Hence, immersions are $I$-cellular.
\end{theorem}

\begin{proof}
        Let $\sigma : C\to D$ be an arbitrary morphism of computads. We will define, by induction, a tower of factorisations of $\sigma$,
        \[\begin{tikzcd}
                        C & {P^0} & {P^1} & \cdots \\
                        &&& D
                        \arrow["{\sigma^2}", from=1-3, to=1-4]
                        \arrow["{\sigma^1}", from=1-2, to=1-3]
                        \arrow["{\sigma^0}", from=1-1, to=1-2]
                        \arrow["\sigma"', bend right = 15, from=1-1, to=2-4]
                        \arrow["{\rho^0}"', from=1-2, to=2-4]
                        \arrow["{\rho^1}", from=1-3, to=2-4]
                \end{tikzcd}\]
        where $U_{n}\rho^n = \id$. Intuitively, the computad $P^n$ looks like $D$ up to dimension $n$ and like $C$ above that dimension. The morphism $\rho_n$ above dimension $n+1$ is given by $\sigma$. More precisely, the tower can be defined inductively starting from $P^{-1} = C$ and $\rho^{-1} = \sigma$. For the inductive step, define first
        \begin{align*}
                U_{n+1}P^{n+1}      & = D_{n+1},         &
                U_{n+1}\rho^{n+1}   & = \id,             &
                U_{n+1}\sigma^{n+1} & = U_{n+1}(\rho^n),
        \end{align*}
        and then, for $k \ge n+1$, define
        \begin{align*}
                U_{k+1}P^{n+1}       & = (U_k P^{n+1}, V_{k+1}^{P^n}, \phi^{P^{n+1}}_{k+1}),
                \\
                \phi^{P^{n+1}}_{k+1} & = \Type_k(U_k\sigma^{n+1}) \phi^{P^{n}}_{k+1},
                \\
                U_{k+1}\sigma^{n+1}  & = (U_k\sigma^{n+1}, v \mapsto \var\, v),              \\
                U_{k+1}\rho^{n+1}    & = (U_k\rho^{n+1}, \rho_{k+1,V}^n).\end{align*}
        It is easy to see inductively that the morphisms above are well-defined and make the diagram commute. Moreover, if $\sigma$ is an immersion, then the $\sigma^n$ satisfy the hypothesis of the previous proposition.

        To finish the theorem, it remains to show that the $\rho^n$ form a colimit cocone under the diagram of the $\sigma^n$ in $\omega\Cat$. Since the cells of $P^n$ of dimension at most $n$ agree with those of $D$, it is easy to see that this becomes a colimit cocone after applying $\Cell = U^w K^w$. The forgetful functor $U^w$ reflects filtered colimits, since the monad $\fc^w$ preserves them, so this is also a colimit of $\omega$-categories.
\end{proof}

\begin{corollary}\label{cor:univ-prop-of-computads}
        Let $C$ be a computad. The free $\omega$-category $K^wC$ is the colimit of its $n$-skeletons
        \[\begin{tikzcd}
                        {\emptyset = K^w\Sk_{-1}C_{-1}} & {K^w\Sk_0C_0} & {K^w\Sk_1C_1} & \cdots,
                        \arrow[tail, from=1-1, to=1-2]
                        \arrow[tail, from=1-2, to=1-3]
                        \arrow[tail, from=1-3, to=1-4]
                \end{tikzcd}\]
        which for every $n\in \N$ fit in a pushout square
        \[\begin{tikzcd}
                        {\coprod_{v\in V_n^C} K^w\mathbb{S}^{n-1}} & {\coprod_{v\in V_n^C} K^w\mathbb{D}^{n}} \\
                        {K^w\Sk_{n-1}C_{n-1}} & {K^w\Sk_nC_n}
                        \arrow["{\iota_n}", tail, from=1-1, to=1-2]
                        \arrow[from=1-1, to=2-1]
                        \arrow[tail, from=2-1, to=2-2]
                        \arrow[from=1-2, to=2-2]
                \end{tikzcd}\]
        In particular, $K^wC$ is $I$-cellular.
\end{corollary}
\begin{proof}
        The tower of skeletons of $C$ is exactly the tower described in the proof of Theorem \ref{thm:immersion-trans} associated to the immersion $\emptyset \to C$.
\end{proof}

\subsection{The universal cofibrant replacement}

The previous corollary describes how to define morphisms out of those $\omega$\-categories that are free on a computad. We will use this to define, for each  $-1 \le n\le \omega$, a right adjoint $W_n : \omega\Cat\to \Comp_n^{\var}$ to the inclusion $K^w\Sk_n\zeta_n$ taking an $n$-computad to the free $\omega$-category on it.

The functor $W_{-1}$ is the unique functor to the terminal category, and it is right adjoint to $K^w\Sk_{-1}\zeta_{-1}$, since the latter sends the unique $(-1)$-computad to the initial $\omega$-category. Suppose that, for some $n \in \N$, we have $W_n$, right adjoint to $K^w\Sk_n\zeta_n$. Let $r_n$ be the counit of this adjunction. For an $\omega$-category $X$, we define $W_{n+1}X = (W_nX, V^{WX}_{n+1}, \phi^{WX}_{n+1})$ using the following pullback square:
\[\begin{tikzcd}
                {V^{WX}_{n+1}} && X_{n+1} \\
                && {\omega\Cat(\mathbb{D}^{n+1},X)} \\
                {\Type_{n}(W_{n}X)} & {\omega\Cat(\mathbb{S}^{n}, K^w\Sk_n\zeta_nW_nX)} & {\omega\Cat(\mathbb{S}^{n},X)}
                \arrow["{\phi_{n+1}^{WX}}", dashed, from=1-1, to=3-1]
                \arrow["\sim", from=3-1, to=3-2]
                \arrow["{(r_{n})_*}", from=3-2, to=3-3]
                \arrow["{(\iota_{n+1})^*}", from=2-3, to=3-3]
                \arrow["\rotatebox{90}{$\sim$}", from=1-3, to=2-3]
                \arrow["{k_{n+1,X}}", dashed, from=1-1, to=1-3]
        \end{tikzcd}\]
The natural isomorphisms in this diagram are given by $\mathbb{D}^{n+1}$ being free on a representable, by $\mathbb{S}^n$ classifying spheres, and $K^w$ being fully faithful respectively. For a homomorphism $f : X\to Y$, we define $W_{n+1}f = (W_nf, (W_{n+1}f)_V)$ to be the generator-preserving $(n+1)$-homomorphism given by
\[ (W_{n+1}f)_{V}(A, x) = \var\,(\Type_n(W_nf)(A), f(x)).\]

In order to show that $W_n$ is right adjoint to $K^w\Sk_{n+1}\zeta_{n+1}$, let $C$ be an $(n+1)$-computad, and let $X$ be an $\omega$-category. A generator-preserving ${(n+1)}$\-homomorphism $\sigma : C\to W_{n+1}X$ consists of
\begin{itemize}
        \item a generator-preserving morphism $\sigma_n : C_n \to W_nX$,
        \item a function $\sigma_V^1 : V^C_{n+1} \to X_{n+1}$,
        \item a function $\sigma_V^2 : V^C_{n+1} \to \Type_n(W_nX)$,
\end{itemize}
such that
\begin{align*}
        \src\,\sigma^1_V & = r_{n,X}(\pr_1 \sigma^2_V),     \\
        \tgt\,\sigma^1_V & = r_{n,X}(\pr_2 \sigma^2_V),     \\
        \sigma_V^2       & = \Type_n(\sigma_n)\phi_{n+1}^C.
\end{align*}
The first two equations are equivalent to $\sigma_V  = (\sigma_V^1, \sigma_V^2)$ defining a function into $V^{WX}_{n+1}$, while the third is equivalent to $(\sigma_n,\sigma_V)$ being an $(n+1)$-homomorphism. In the presence of this third equation, the function $\sigma_V^2$ is redundant and the other equations amount to
\begin{align*}
        \src\,\sigma^1_V & = r_n (K^w\Sk_n\zeta_n\sigma_n) (\pr_1\phi_n^C) = r_n (K^w\Sk_n\zeta_n\sigma_n)\,\src \\
        \tgt\,\sigma^1_V & = r_n (K^w\Sk_n\zeta_n\sigma_n) (\pr_2\phi_n^C) = r_n (K^w\Sk_n\zeta_n\sigma_n)\,\tgt
\end{align*}
By the inductive hypothesis, such data are in bijection to pairs of a homomorphism $\sigma_n^\dagger : \Sk_n C \to X$ and a function $\sigma_V^1$ as above, satisfying
\begin{align*}
        \src\,\sigma^1_V & = \sigma_n^\dagger\,\src, &
        \tgt\,\sigma^1_V & = \sigma_n^\dagger\,\tgt,
\end{align*}
which is the data of a homomorphism $\sigma_{n+1}^\dagger : \Sk_{n+1} C \to X$ by Corollary \ref{cor:univ-prop-of-computads}. It is easy to see that this bijection is natural in $X$ and $C$. The component of the counit $r_{n+1}$ at $X$ is the homomorphism given by $r_{n,X}$ and the projection $k_{n+1,X}$.

\begin{definition}
\label{def:W}
        For each $\omega$-category $X$, we define the computad $WX$ to consist of the $n$\-computads $(W_nX)$, and we define $f$ similarly on homomorphisms.
\end{definition}

A generator-preserving morphism $\sigma : C\to WX$ consists of a generator-preserving $n$-homo-morphisms $\sigma_n : C_n\to W_n X$ for every $n\in N$, satisfying that $u_{n+1}\sigma_{n+1} = \sigma_n$. Using the explicit description of the counit above and the triangle laws, it is easy to see that such families of $n$-homomorphisms are in bijection to morphisms $\sigma_n^\dagger : \Sk_nC\to X$ forming a cocone under the diagram in Corollary \ref{cor:univ-prop-of-computads}. Hence, they are in bijection to morphisms $\sigma^\dagger : C\to X$, and we have exhibited an adjunction $K^w\zeta \dashv W$.

\begin{definition}
        The \emph{computadic resolution} $(Q, \Delta, r)$ is the comonad induced by the adjunction $K^w\zeta \dashv W$.
\end{definition}

\noindent We will show that the pair $(Q, r)$ satisfies the hypothesis of Theorem \ref{thm:recongnition}. Since initial objects are unique up to unique isomorphism, this implies that, for each $\omega$\-category $X$, the morphism $r_X : QX\to X$ is its universal cofibrant replacement. We need to construct, for each $n\in \N$, each sphere $A\in \Type_n(WX)$, and each generator $x\in X_n$, a lift of the following form:
\[\begin{tikzcd}
                {K^w\mathbb{S}^{n-1}} & QX \\
                {K^w\mathbb{D}^n} & X
                \arrow["x", from=2-1, to=2-2]
                \arrow["A", from=1-1, to=1-2]
                \arrow["{r_X}", from=1-2, to=2-2]
                \arrow["{K^w\iota_n}"', from=1-1, to=2-1]
                \arrow["{\phi(n,x,A)}"{description}, dashed, from=2-1, to=1-2]
        \end{tikzcd}\]
Commutativity of this square is equivalent to $(A,x)$ being a generator of $W_{n+1}X$. Furthermore, the homomorphism $r_X$ acts on generators like a projection; that is,
\[ r_X(\var\, (A,x)) = x. \]
Hence, we may define the required lift by
\[ \phi(n, A, x) = \var\,(A, x). \]

\noindent We claim that $(r_X, \phi)$ is initial among objects above $X$ equipped with such a choice of lifts. Given $f : Y \to X$ with a choice of lifts $\psi$, we will define a sequence of homomorphisms $g_n : \Sk_nW_nX\to Y$ inductively. First, we define  $g_{-1}$ to be the unique homomorphism out of the initial $\omega$-category. The homomorphism $g_{n+1}$ is then given by $g_n$ on cells of dimension at most $n$ and sends the generator $(x,A)$ to
\[g_{n+1}(\var\,(x,A)) = \psi(n+1, g\circ A, x).\]
Finally, the homomorphism $g : QX\to Y$ is the one that agrees with $g_n$ on cells of dimension at most $n$. The homomorphism $g$ is the only one that can preserve the lifts strictly. We have that
\[g\circ \phi(n , A, x) = \psi(n, g\circ A, x).\]
Furthermore, any other homomorphism with this property would have to agree on all generators of $QX$ with $g$. For the same reason, we have that $fg = r_X$. Hence, $g$ is indeed the unique morphism above $X$ preserving lifts. To summarise, Theorem \ref{thm:recongnition} allows us to conclude the following proposition.

\begin{proposition}
        The pair $(Q, r)$ underlies the universal cofibrant replacement comonad on $\omega\Cat$ that is generated (in the sense of Garner~\cite{garner_homomorphisms_2010}) by the set of inclusions of spheres into disks.
\end{proposition}

\section{Comparison to other notions of \texorpdfstring{$\omega$}{ω}-categories}

Our goal in this section is to show that our notion of $\omega$-category coincides with that of Leinster \cite{leinster_higher_2003}. We will first introduce a class of morphisms called \emph{covers} and show that every morphism factors uniquely into a cover followed by an immersion. Using this observation, we will give an alternative characterisation of full spheres and show that morphisms of computads $\Free B \to \Free X$  also admit a \emph{(generic, free) factorisation}. This implies that the monad  $\fc^w$ has arities in the sense of Berger et al~\cite{berger_weber_2012}, in the full subcategory $\Theta_0$ of $\Glob$ consisting of the globular sets of positions of Batanin trees. Furthermore, the globular theory $\Theta_w$ associated to this monad is homogeneous. We will finally show that under the equivalence of homogeneous theories over the theory $\Theta_s$ of strict $\omega$-categories and Batanin's globular operads~\cite{berger_cellular_2002, berger_weber_2012}, the theory $\Theta_w$ corresponds to the initial contractible globular operad of Leinster \cite{leinster_higher_2003}. This implies that $\fc^w$ is isomorphic to Leinster's monad, and hence that our notion of $\omega$-category coincides with that of Leinster.

\subsection{Supports and covers}

Since $\Comp^{\var}$ is a presheaf category, every generator-preserving morphism can be factored into an epimorphism followed by a monomorphism. We will now extend this factorisation to an orthogonal factorisation system on the whole of $\Comp$; every morphism of computads factors essentially uniquely into a \emph{cover} followed by an \emph{immersion}. Recall that immersion are the generator-preserving monomorphisms.
Covers will be defined below in terms of the \emph{support}, and we will see that, a posteriori, they coincide with epimorphisms of computads.

\begin{definition}
        Suppose that $c \in \Cell_n(C)$ is a cell in a computad. For each $k\in \N$, we define the $k$\-dimensional \emph{support} $\fv_k(c) \subseteq V_k^C$ of $c$ inductively:
        \begin{itemize}
                \item If $k > n$, then
                      \[
                              \fv_k(c) = \emptyset.
                      \]

                \item If $n=0$, and $k = 0$, then
                      \[\fv_0(c) = \{c\}.\]

                \item If $n > 0$, and $c = \var v$. Then, we define
                      \[\fv_n(c) = \{v\},\]
                      and, for each $k < n$, we define
                      \[\fv_k(c) = \fv_k(\src\,\var\, v)\cup\fv_k(\tgt\,\var\, v),\]
                      where the union is taken inside $V_k^C$.

                \item If $n > 0$, and $c = \coh(B,A,\tau)$ is a coherence cell, then we define for $k \le n$,
                      \[\fv_k(c) = \bigcup_{\substack{m \in \N,\\ p\in\Pos_m(B)}} \fv_k(\tau_{m,V}(p)).\]
        \end{itemize}
        The $k$-dimensional \emph{support} of a morphism of computads $\sigma : D\to C$ is defined by
        \[\fv_k(\sigma) = \bigcup_{\substack{n \in \N, \\ v\in V^D_n}} \fv_k(\sigma_{n,V}(v)).\]
\end{definition}

\begin{remark}
        This notion of $k$-dimensional support subsumes the one given in Section \ref{sec-computads}, in that,
        for each $n$-cell $c$, we have that $\fv_n(c) = \fv(c)$.
\end{remark}

\noindent We can also define the $k$-dimensional support of a cell $c$ in an $n$-computad $C$ by viewing $c$ as a cell of $\Sk_n C$. Similarly, we can define the support of a morphism $\sigma$ of $n$-computads to be the support of $\Sk_n\sigma$. The following lemmas, that can be easily shown by induction, demonstrate how support allows us construct lifts along immersions and epimorphisms of computads.

\begin{lemma}\label{source-of-support}
        Let $c \in \Cell_n(C)$ a cell in a computad and $k\in \N$. Then
        \[
                \fv_k(\src\, c) \cup \fv_k(\tgt\, c) \subseteq \fv_k(c).
        \]
        Moreover, for every morphism of computads $\sigma : C\to D$,
        \[      \fv_k(\Cell_n(\sigma)(c)) = \bigcup_{\substack{m \in \N, \\ v\in \fv_m(c)}} \fv_k(\sigma_{m,V}(c))     \]
\end{lemma}

\begin{lemma}\label{support-of-var-to-var}
        Let $\sigma : C\to D$ be a generator-preserving morphism. The $k$\-dimensional support of $\sigma$ are those of the form $\sigma_{k,V}(v)$, for some generator $v\in V_k^C$.
\end{lemma}

\begin{lemma}\label{immersion-lifting}
        Let $\sigma : D'\to D$ be an $n$-immersion, and let $\tau : C\to D$ be an $n$-homomorphism, for some $n\le \omega$. Then, there exists a unique
        $\tau' : C\to D'$ such that $\tau = \sigma \tau'$ if and only if $\fv_k(\tau)\subseteq\fv_k(\sigma)$ for all $k\in \N$.
\end{lemma}

\begin{lemma}\label{support-epis}
        Two homomorphisms $\sigma, \sigma' : C\to D$ agree on a cell $c\in \Cell_n(C)$ if and only if they agree on its support.
\end{lemma}

\begin{remark}
        Lemma \ref{support-of-var-to-var} implies in particular that the $k$-dimensional support of a generator-preserving homomorphism $\sigma$ only depends on its truncation $U_k\sigma$. This is not true in general. Consider for example the computad corresponding to the tree $B_2 = \br\,[D_0,D_0]$ of Example \ref{example:one-dimensional-trees} consisting of two composable $1$-cells $f : x\to y$ and
        $g : y \to z$. THe homomorphism $\sigma : \mathbb{D}^1\to \Free B_2$ picking the composite $\comp(f,g)$ has $y$ in its support. However, $y$ is not in the support of $\sigma_0$.
\end{remark}

\begin{definition}
        We say that a morphism of computads $\sigma : D \to C$ is a \emph{cover} when $\fv_k(\sigma) = V^C_k$ for all $k\in \N$. We say that a cell $c\in \Cell_n(C)$ \emph{covers} $C$ when the corresponding morphism $\mathbb{D}^n\to C$ is a cover.
\end{definition}

\begin{example}
        By Lemma \ref{source-of-support}, covers are closed under composition and every isomorphism is both a cover and an immersion.
\end{example}

\begin{example}
        A coherence cell $\coh\,(B,A,\tau)$ is a cover if and only if $\tau$ is. In particular, cells of the form $\coh\,(B,A,\id)$ are always covers.
\end{example}

\begin{remark}
        A direct consequence of Lemma \ref{support-epis} is that covers are epimorphisms. Since $\Comp^{\var}$ is balanced, it follows that a morphism is both a cover and an immersion exactly when it is an isomorphism.
\end{remark}

\begin{remark}
        Lemmas \ref{source-of-support} and \ref{immersion-lifting} imply that the support of a morphism $\sigma : C\to D$ form a computad $\fim_\sigma$ and that $\sigma$ factors into a cover followed  by an immersion via $\fim_\sigma$. This factorisation is easily seen to be unique. Hence, covers and immersions form an orthogonal factorisation systems in $\Comp$. One corollary of this result is that epimorphism are covers.
\end{remark}

The following lemma and its converse, which is shown in Proposition \ref{full-types-as-covers}, offers an equivalent characterisation of full spheres in terms of covers.

\begin{lemma}\label{type-to-cover}
        Let $A$ a full $n$-sphere of a Batanin tree $B$. There exist unique cells $a, b$ covering the boundary $\partial_n B$ such that
        \begin{align*}
                \pr_1A & = \fc^w(s_n^B)(a) &
                \pr_2A & = \fc^w(t_n^B)(b)
        \end{align*}
\end{lemma}
\begin{proof}
        By the inductive definition of full spheres and Lemma \ref{source-of-support}, we can easily see that the $k$-dimensional support of $\pr_1 A$ must contain all source boundary $k$-positions when $k\le n$, and that equality holds when $k = n$. Lemma \ref{source-of-support} and Proposition \ref{prop:boundary-positions-generation} imply then that the $k$-dimensional support of $\pr_1 A$ and $\Free\,(s_n^B)$ coincide. The lifting lemma \ref{immersion-lifting} implies existence and uniqueness of $a$, while \ref{support-of-var-to-var} implies that $a$ covers $\partial_nB$. The existence of $b$ follows similarly.
\end{proof}

\subsection{Generic-free factorisations}

Let $\Theta_0$ the category with objects Batanin trees and morphisms of globular sets between their sets of positions
\[\Theta_0(B, B') = \Glob(B,B')\]
Our goal in this section is to show that the monad $\fc^w$ has arities in $\Theta_0$ in the sense of \cite{berger_weber_2012}. To do that, we will construct a \emph{generic}, generator-preserving factorisation for every homomorphism of the form $\Free B\to \Free X$.

\begin{definition}
        Let $X, Y$ globular sets. We will say that a morphism of computads $\sigma : \Free\,X \to \Free\,Y$ is \emph{generic} when for every solid commutative square of the form
        \[\begin{tikzcd}
                        {\Free\,X} & {\Free\,Z} \\
                        {\Free\,Y} & {\Free\,W}
                        \arrow["f"', from=2-1, to=2-2]
                        \arrow["g", from=1-2, to=2-2]
                        \arrow["\tau", from=1-1, to=1-2]
                        \arrow["\sigma"', from=1-1, to=2-1]
                        \arrow["h"{description}, dashed, from=2-1, to=1-2]
                \end{tikzcd}\]
        with $f, g$ generator-preserving, there exists a unique $h$ making the entire diagram commute.
\end{definition}

\begin{example}
        Every isomorphism is generic. If a generator-preserving morphism $\sigma : \Free\,X\to \Free\,Y$ is generic then it is an isomorphism by the existence of a lift to the following square:
        \[\begin{tikzcd}
                        {\Free\,X} & {\Free\,X} \\
                        {\Free\,Y} & {\Free\,Y}
                        \arrow[from=2-1, to=2-2, equals]
                        \arrow["\sigma", from=1-2, to=2-2]
                        \arrow[from=1-1, to=1-2, equals]
                        \arrow["\sigma"', from=1-1, to=2-1]
                        \arrow[dashed, from=2-1, to=1-2]
                \end{tikzcd}\]
\end{example}

\begin{remark}
        Generic homomorphisms correspond precisely to the $\fc^w$-generic morphisms of Berger et al~\cite{berger_weber_2012} under the adjunction $\Free\dashv \Cell$, since generator-preserving homomorphisms correspond to free ones and $\Free$ is faithful.
\end{remark}

\begin{lemma}\label{generic-free-unique}
        Suppose that we have Batanin trees $B_1$, $B_2$, $B_3$ and a commutative square
        \[\begin{tikzcd}
                        {\Free\, B_1} & {\Free\, B_3} \\
                        {\Free\, B_2} & {\Free\,X}
                        \arrow["f"', from=2-1, to=2-2]
                        \arrow["g", from=1-2, to=2-2]
                        \arrow["\tau", two heads, from=1-1, to=1-2]
                        \arrow["\sigma"', from=1-1, to=2-1]
                \end{tikzcd}\]
        where $\sigma$ is generic, $\tau$ is a cover, and $f$ and $g$  are generator-preserving. Then, $B_2 = B_3$, $f = g$ and $\sigma = \tau$.
\end{lemma}
\begin{proof}
        By definition of $\sigma$ being generic, there exists some $h : \Free\, B_2 \to \Free\, B_3$ such that $f=gh$ and $\tau = h \sigma$. By the first equation, $h$ must be generator-preserving, so $h = \Free h'$ for some morphisms of globular sets $h' : B\to B'$. By the second one, $h$ must be a cover, so by Lemma \ref{support-of-var-to-var}, $h'$ is epic. It follows from~\cite[Lemma~1.3]{berger_cellular_2002} that $h'$ must be an identity morphism.
\end{proof}

Let $\Bat$ be the globular set whose $n$-cells are Batanin trees $B$ of dimension at most $n$, and whose $n$-source and $n$-target are both given by the boundary operator $\partial_n$. The assignment $B \mapsto \Pos\,B$ can be extended to a functor ${\Pos : \el(\Bat) \to \Glob}$ out of the category of elements of $\Bat$ by sending the source and target morphisms $(n, \partial_n B) \to (n+k, B)$ to the morphisms $s_n^B$ and $t_n^B$ respectively. To define the (generic, free)-factorisation, we use the following proposition of Weber which says that the class of Batanin trees is closed under certain colimits.

\begin{proposition}\label{pasting-batanin-trees}\cite[Proposition~4.7]{weber_pra_2004}
        Let $f : \Pos\,(B) \to \Bat$ a morphism where $B$ is a Batanin tree. There exists (unique) Batanin tree $B_f$ of dimension at most $\dim\, B$ such that
        \[ \colim(\Pos\circ \el(f)) = \Pos\,B_f. \]
        Moreover, this construction commutes with the source inclusions in that
        \[ \colim(\Pos\circ \el(f\circ s_n^B)) = \Pos\,\partial_nB_f \]
        and the canonical morphism between the colimits is $s_n^B$, and similarly it commutes with target inclusions.
\end{proposition}

\begin{proposition}\label{generic-free}
        Suppose that $B$ is a Batanin tree, $X$ is globular set and $\sigma : \Free\, B\to \Free\, X$ a homomorphism. Then $\sigma$ factors uniquely as
        \[\begin{tikzcd}
                        \Free\, B \ar[r, "\gen_{\sigma}"] &
                        \Free\,{B_\sigma} \ar[r, "\fr_\sigma"] &
                        \Free X
                \end{tikzcd}\]
        for some Batanin tree $B_\sigma$ of dimension at most $\dim\,B$, some generic $\gen_\sigma$ and some generator-preserving morphism $\fr_\sigma$.
\end{proposition}

\begin{proof}
        Uniqueness of such factorisation is a direct consequence of the lifting property defining generics, and injectivity of morphisms between  trees~\cite[Lemma~1.3]{berger_cellular_2002}. We will construct a generic, generator-preserving factorisation for every homomorphism $\sigma : \Free\,^B\to \Free\,X$ by induction on $\dim B$ and $\mdpth \sigma_{\dim\,B}$. We will simultaneously show that the factorisation commutes with source and target inclusions; i.e.\,we have that $B_{\sigma \circ \Free\,s_k^B} = \partial_k B_\sigma$, and the following diagram commutes
        \[\begin{tikzcd}[column sep = huge, row sep = tiny]
                        {\Free\,{\partial_k B}} & {\Free\,{\partial_kB_\sigma}} \\
                        {} && {\Free\,{B'}} \\
                        {\Free\,B} & {\Free\,{B_\sigma}}
                        \arrow["{\fr_{\sigma\circ \Free\,s_k^B}}", from=1-2, to=2-3]
                        \arrow["{\gen_{\sigma\circ \Free\,s_k^B}}", from=1-1, to=1-2]
                        \arrow["{\gen_\sigma}", from=3-1, to=3-2]
                        \arrow["{\fr_\sigma}"', from=3-2, to=2-3]
                        \arrow["{\Free\,s_k^B}"', from=1-1, to=3-1]
                        \arrow["{\Free\,s_k^{B_\sigma}}"', dashed, from=1-2, to=3-2]
                \end{tikzcd}\]
        and a similar results holds for target inclusions.

        Suppose first that $B = D_n$ is a globe and that $\sigma$ corresponds to a generator. In this case, $\sigma$ must already be generator-preserving, so we may set ${\gen_\sigma = \id}$ and $\fr_\sigma = \sigma$. Commutativity of this construction with source and target inclusions follows from $\sigma\circ \Free\,s_k^{D_n}$ and $\sigma\circ \Free\,t_k^{D_n}$ also being generator-preserving morphisms out of a globe.

        Suppose now that $B = D_n$ and that $\sigma$ corresponds to a coherence cell $\coh\,(B',A,U_n\tau)$ for unique morphism $\tau : \Free\, B'\to \Free\,X$. If we let $\tilde{\sigma} : \mathbb{D}^n\to \Free\,{B'}$ the cover corresponding to the cell $\coh(B',A,\id)$, then $\sigma = \tau\tilde{\sigma}$, and we let
        \begin{align*}
                B_\sigma & = B_\tau, & \gen_\sigma & = \gen_\tau \tilde{\sigma}, & \fr_\sigma = \fr_\tau.
        \end{align*}
        To show that $\gen_\sigma$ is generic, consider a commutative square where $f, g$ are generator-preserving.
        \[\begin{tikzcd}
                        {\mathbb{D}^n} & {\Free Y} \\
                        {\Free\,B'} \\
                        {\Free\,{B_\tau}} & {\Free Z}
                        \arrow["f", from=1-2, to=3-2]
                        \arrow["g"', from=3-1, to=3-2]
                        \arrow["c", from=1-1, to=1-2]
                        \arrow["{\tilde{\sigma}}"', from=1-1, to=2-1]
                        \arrow["{\gen_\tau}"', from=2-1, to=3-1]
                        \arrow["\rho"{description}, dashed, from=2-1, to=1-2]
                        \arrow["h"{description}, dashed, from=3-1, to=1-2]
                \end{tikzcd}\]
        Since $f$ is generator-preserving, the cell corresponding to $c$ must be of the form $\coh(B,A,U_n\rho)$ for unique morphism $\rho$ making the diagram above commute. The existence of unique $h$ making the entire diagram commute follows then by $\gen_\tau$ being generic.

        When $k\ge n$, the source and target inclusions are identities, so this factorisation commutes trivially with them. let $k<n$ and consider the $k$-cell $a =\src^{(n-k-1)}(\pr_1 A)$ of $\Free X$. This is the source of a full spheres and corresponds to the morphism $\tilde{\sigma} \circ \Free\,s_k^{D_n} : \mathbb{D}^k\to \Free\,{B'}$. Therefore, by Lemma \ref{type-to-cover}, there exists unique cover $\tilde{a}:\mathbb{D}^k\to \Free\,{\partial_k B'}$ making the left square below commute.
        \[\begin{tikzcd}[column sep = huge ]
                        {\mathbb{D}^k} & {\Free\,{\partial_k B'}} & {\Free\,{\partial_k B_\tau}} \\
                        {\mathbb{D}^n} & {\Free\,{B_0}} & {\Free\,{B_\tau}} & {\Free X}
                        \arrow["{\tilde{a}}", from=1-1, to=1-2]
                        \arrow["{\gen_{\tau\circ \Free\,s_k^{B_0}}}", from=1-2, to=1-3]
                        \arrow["{\fr_{\tau\circ \Free\,s_k^{B_0}}}", from=1-3, to=2-4]
                        \arrow["{\tilde{\sigma}}", from=2-1, to=2-2]
                        \arrow["{\gen_\tau}", from=2-2, to=2-3]
                        \arrow["{\fr_\tau}", from=2-3, to=2-4]
                        \arrow["{\Free\,s_k^{D_n}}"', from=1-1, to=2-1]
                        \arrow["{\Free\,s_k^{B_0}}"', from=1-2, to=2-2]
                        \arrow["{\Free\,s_k^{B_\tau}}"', from=1-3, to=2-3]
                \end{tikzcd}\]
        The rest of the diagram commutes from the inductive hypothesis on $\tau$.
        The top row is a cover, generator-preserving factorisation of $\sigma \circ \Free\,s_k^{D_n}$. By Lemma~\ref{generic-free-unique}, it must coincide with the generic, generator-preserving factorisation, so constructed factorisation of $\sigma$ commutes with source inclusions. A similar argument shows that it commutes with target inclusions.

        Finally, let $B$ be an arbitrary Batanin tree and suppose that for every position $(l,p)\in \el(\Pos\,B)$ we have constructed the generic, generator-preserving factorisation of the morphism $\hat{p} : \mathbb{D}^l\to \Free\,X$ corresponding to the cell $\sigma_{l,V}(p)$.
        \[\begin{tikzcd}[column sep = huge]
                        {\Free\,D_l} & {\Free\,B_{\hat{p}}} & {\Free X}
                        \arrow["{\gen_{\hat{p}}}",  from=1-1, to=1-2]
                        \arrow["{\fr_{\hat{p}}}=\Free\,\fr'_{\hat{p}}", from=1-2, to=1-3]
                \end{tikzcd}\]
        Since the generic, generator-preserving factorisations commute with source and target inclusions, this factorisation is functorial in $(l,p)$. Taking colimits, we obtain morphisms
        \[\begin{tikzcd}[row sep = 0, column sep = huge]
                        \colim_{\el(\Pos\,B)} \Free\Pos\,D_k
                        \ar[r, "\gen_\sigma"] &
                        \colim_{\el(\Pos\,B)} \Free\Pos\,B_{\hat{p}} \\
                        \phantom{\colim_{\el(\Pos\,B)} \Free\Pos\,D_k}
                        \ar[r, "\fr_\sigma = \Free\,\fr_\sigma'"] &
                        \colim_{\el(\Pos\,B)} \Free X.
                \end{tikzcd}\]
        The first term is $\Free\,B$ by cocontinuity of $\Free$ and the density lemma. The third term is $\Free X$ since it is the colimit a constant diagram. The composite of the two morphisms is easily seen to be $\sigma$. Commutativity of generic, generator-preserving factorisations with source and target inclusions implies that the middle term is a colimit of the form mentioned in Proposition \ref{pasting-batanin-trees}, and so it is of the form $\Free\,{B_\sigma}$ for some Batanin tree $B_\sigma$. Moreover, this construction coincides with the one before in the case of globes, and commutes with the source and target inclusions by the last part of \ref{pasting-batanin-trees}.

        It remains to show that $\gen_\sigma$ is generic, so consider a commutative square as the one below where $f,g$ are generator-preserving
        \[\begin{tikzcd}[row sep = large]
                        {\mathbb{D}^l} & {\Free\,B} & {\Free Y} \\
                        {\Free\,{B_{\hat{\sigma}}}} & {\Free\,{B_\sigma}} & {\Free\,Z}
                        \arrow["g", from=1-3, to=2-3]
                        \arrow["f"', from=2-2, to=2-3]
                        \arrow["{\gen_\sigma}", from=1-2, to=2-2]
                        \arrow["\tau", from=1-2, to=1-3]
                        \arrow["{\Free\,i_p}", dashed, from=1-1, to=1-2]
                        \arrow["{\Free\,j_p}"', dashed, from=2-1, to=2-2]
                        \arrow["{\gen_{\hat{p}}}"', from=1-1, to=2-1]
                        \arrow["h_p"{near start}, dashed, bend left = 7, from=2-1, to=1-3]
                        \arrow["h"', dashed, from=2-2, to=1-3]
                \end{tikzcd}\]
        For every position $(l,p)$, let $i_p$ and $j_p$ the morphisms of globular sets forming the colimit cocones over $\Pos\,B$ and $\Pos\,B_\sigma$ respectively. Since $\gen_{\hat{p}}$ is generic, there exists unique lift $h_p$ to the diagram above. Uniqueness of those lifts shows that they are compatible with source and target inclusions, so they give a morphism $h$ out of the colimit. Uniqueness of $h$ follows from that of the lifts $h_p$.
\end{proof}

\begin{corollary}
        The monad $\fc^w$ has arities in $\Theta_0$.
\end{corollary}

\begin{proof}
        By \cite[Proposition~2.5]{berger_weber_2012}, $\fc^w$ has arities in $\Theta_0$ if and only if for every Batanin tree $B$, globular set $X$, and morphism $f : B\to \fc^wX$, a certain category of factorisations of $f$ is connected. The (generic, generator-preserving)-factorisation of the transpose $f^\dagger : \Free\,B\to \Free\,X$ corresponds to a factorisation of $f$ into an $\fc^w$-generic followed by a free morphism. This factorisation is by definition an initial object of the aforementioned factorisation category. Therefore $\fc^w$ has arities in $\Theta_0$.
\end{proof}

\begin{corollary}\label{cover-generic}
        Let $B$, $B'$ be Batanin trees. A morphism $\sigma : \Free\,B\to \Free\,{B'}$ is a cover if and only if it is generic. In particular, generics between Batanin trees are closed under composition.
\end{corollary}
\begin{proof}
        Suppose that $\sigma$ is a cover, and consider the following square:
        \[\begin{tikzcd}
                        {\Free\,{B}} & {\Free\,{B'}} \\
                        {\Free\,{B_\sigma}} & {\Free\,{B}}
                        \arrow["{\fr_\sigma}"', from=2-1, to=2-2]
                        \arrow["\id", from=1-2, to=2-2]
                        \arrow["\sigma", from=1-1, to=1-2]
                        \arrow["{\gen_\sigma}"', from=1-1, to=2-1]
                \end{tikzcd}\]
        Lemma \ref{generic-free-unique} implies that $\sigma = \gen_\sigma$ is generic. Conversely, suppose that $\sigma$ is generic. By uniqueness of the generic-free factorisation, we have that \mbox{$\sigma = \gen_\sigma$}. However, it follows from the explicit description of $\gen_\sigma$ in the proof of Proposition \ref{generic-free} that $\gen_\sigma$ is a cover.
\end{proof}

\subsection{Globular theories}
\label{sec:globular-theories}

Leinster's $\omega$-categories \cite{leinster_higher_2003} are algebras for the initial normalised \emph{globular operad} with a contraction.  By the equivalence of globular operads and \emph{homogeneous globular theories} established in \cite[6.6.8]{ara_sur_2010}, \cite[Proposition~1.16]{berger_cellular_2002} and \cite[Theorem~3.13]{berger_weber_2012}, we may equivalently speak of homogeneous globular theories with contractions. We equip the theory $\Theta_w$, corresponding to $\fc^w$, with a contraction. We will then show that $\Theta_w$ is the initial theory equipped with a contraction, and conclude that Leinster's operad is isomorphic to $\fc^w$. We begin by recalling the relevant notions.

\begin{definition}
        Let $F :\mathbb{G}\to \mathcal{C}$ a category under the category of globes. A \emph{globular sum} in $\mathcal{C}$ is a colimit of a diagram of the form
        \[\begin{tikzcd}
                        {F([n_0])} && \dots && {F([n_k])} \\
                        & {F([m_1])} && {F([m_k])}
                        \arrow["{F(t)}", from=2-2, to=1-1]
                        \arrow["{F(s)}"', from=2-2, to=1-3]
                        \arrow["{F(s)}"', from=2-4, to=1-5]
                        \arrow["{F(t)}", from=2-4, to=1-3]
                \end{tikzcd}\]
        where $n_i$ and $m_j$ are elements of $\mathbb{G}$.
\end{definition}

\begin{example}
        Viewing $\Glob$ as a category under $\mathbb{G}$ using the Yoneda embedding, Remark \ref{rem:trees-are-sums} explains that globular sums coincide with Batanin trees.
\end{example}

\begin{definition}[{\cite[Definition~1.5]{berger_cellular_2002}, \cite[2.2.6]{ara_sur_2010}}]
        Let $\Theta_0$ be the category with objects Batanin trees and arrows the morphisms between their globular sets of positions. A \emph{globular theory} is a category $\Theta_A$, together with a bijective on objects, faithful functor $j_A:\Theta_0 \hookrightarrow \Theta_A$ preserving globular sums. A \emph{morphism of globular theories} $i : \Theta_A\to \Theta_B$ is a functor such that $j_B = j_A i$.
\end{definition}

\begin{example}
        Every faithful monad $T$ on $\Glob$ induces a globular theory $\Theta_T$, such that $\Theta_T(B,B') = T\text{-}\operatorname{Alg}(TB, TB')$. In particular, the free $\omega$-category monad $\fc^w$ defines a globular theory $\Theta_w = \Theta_{\fc^w}$.
\end{example}

\begin{definition}[{See \cite[Definition~1.15]{berger_cellular_2002}, \cite[2.7.1]{ara_sur_2010} and \cite[Definition~3.9]{berger_weber_2012}}]
        An \emph{immersion} in a globular theory $\Theta_A$ is a morphism of $\Theta_0$. We say that a morphism $f$ of $\Theta_A$ is \emph{homogeneous} if whenever $f = ig$ for some immersion $i$, we have that $i = \id$. We say that $\Theta_A$ is \emph{homogeneous} if every morphism factors uniquely as $f = ic$ with $i$ an immersion and $c$ a cover.
\end{definition}

In light of Proposition \ref{morphisms-are-functors}, the category $\Theta_w$ may be identified with the full subcategory of $\Comp$ whose objects are free on a Batanin tree. Since every morphism between the globular sets of positions of Batanin trees is monic (see \cite[Lemma~1.3]{berger_cellular_2002}), the two notions of immersions coincide. Furthermore, Proposition~\ref{generic-free} and Corollary \ref{cover-generic} imply that the two notions of covers also coincide. Hence, we have the following result:

\begin{proposition}
        The globular theory $\Theta_w$ is \emph{homogeneous}.
\end{proposition}

Suppose that $\Theta_s$ is the homogeneous globular theory whose arrows are morphisms between the \emph{strict} $\omega$-categories generated by Batanin trees. More concretely, for each Batanin tree $B$ with $\dim B\le n$, there is a unique cover $c_n^B : D_n \to B$. We follow \cite[2.7.6]{ara_sur_2010} and say that that a globular theory $\Theta_A$ is \emph{homogeneous over $\Theta$} when it comes equipped with a morphism $\Theta_A \to \Theta_s$ that preserves (and reflects) covers.

\begin{remark}\label{homogeneous-globular-theories-are-globular-operads}
        By a result of Ara~\cite[Theorem 6.6.8]{ara_sur_2010} the category of homogeneous globular theories over $\Theta_s$ is equivalent to the category of \emph{globular operads} in the sense of Batanin~\cite{batanin_computads_1998}. (See also \cite[Proposition 1.16]{berger_cellular_2002}, \cite[Theorem 3.13]{berger_weber_2012}.) Given a homogeneous globular theory $\Theta_A$ over $\Theta_s$, the corresponding globular operad $O_A$ is defined so that, for each Batanin tree $B$ and $n\in \N$, an operation of $O_{A,n}(B)$ corresponds exactly to a cover $\sigma : D_n \to B$ in $\Theta_A$.
\end{remark}

\begin{proposition}
        The globular theory $\Theta_w$ is homogeneous over $\Theta_s$.
\end{proposition}
\begin{proof}
        We define an identity-on-objects functor $S : \Theta_w \to \Theta_s$. On immersions we set $S \Free i = i$. Suppose that $\sigma : \mathbb{D}^n \to \Free\,B$ is a cover in $\Theta_w$. It follows from Proposition \ref{generic-free} that $\dim B \leq n$. Hence, we define $S \sigma = c^B_n$.  Since every cover $B \to B'$ in $\Theta_w$ is a globular sum of covers of the form $\mathbb{D}^k \to B'$, and covers are closed under globular sums in both $\Theta_w$ and $\Theta_s$, these data suffice to define $S$ on all covers. Since every morphism in $\Theta_w$ factors uniquely as a cover followed by an immersion, this suffices to define $S$ on all arrows in $\Theta_w$. It is easily verified that this construction is functorial, and preserves and reflects covers.
\end{proof}

\begin{definition}
        A homogeneous globular theory over $\Theta_s$ is \emph{normalised} when there is a unique cover $D_0 \to D_0$, namely $\id_{D_0}$. A \emph{contraction} on a normalised homogeneous globular theory consists of, for each $n\in \N$, each Batanin tree $B$ with $\dim B \leq n+1$, and each pair of covers $c, d : D_{n} \to \partial_{n} B$ that are parallel, i.e.
        \begin{align*}
                c \circ s_{n-1}^{D_{n}} & = d \circ s_{n-1}^{D_{n}} & \text{and} &  &
                c \circ t_{n-1}^{D_{n}} = d \circ t_{n-1}^{D_{n}},
        \end{align*}
        a choice of cover $l^{c, d} : D_{n+1} \to B$ such that
        \begin{align*}
                l^{c, d} \circ  s_{n}^{D_{n+1}} = s^B_{n} \circ c &  & \text{and} &  & l^{c, d} \circ  t_{n}^{D_{n+1}} = t^B_{n} \circ d.
        \end{align*}
\end{definition}

\begin{remark}
        Contractibility is usually viewed as a property of globular operads. Under the equivalence of Remark \ref{homogeneous-globular-theories-are-globular-operads}, the contractions described here correspond to the contractions of globular operads described by Leinster \cite{leinster_higher_2003}.
\end{remark}

In order to equip $\Theta_w$ with a contraction and show that it is initial among globular theories with a contraction, we will need the following characterisation of full spheres.

\begin{proposition}\label{full-types-as-covers}
        Let $A$ an $n$-sphere of a Batanin tree $B$. Then $A$ is full if and only if there exist $n$-cells $a,b$ that cover $\partial n B$ such that
        \begin{align*}
                \pr_1 A & = \fc^w(s_n^B)(a), &
                \pr_2 A & = \fc^w(s_n^B)(b).
        \end{align*}
        Moreover, if $\dim B\le n+1$, this is equivalent to the existence of an $(n+1)$-cell of sphere $A$ that covers $B$.
\end{proposition}
\begin{proof}
        Let $n\in \N$ and suppose that for all $m<n$ and all $m$-spheres of some Batanin tree, the proposition holds. Let $B$ a Batanin tree and $A$ a full $n$-sphere of it. The existence of covers $a, b$ of ${\partial_n B}$ has then been shown in Lemma \ref{type-to-cover}. When $\dim B\le n+1$, a cover of $B$ of sphere $A$ is given by $\coh(B,A,\id)$.

        Suppose now $A$ is a sphere for which such $a, b$ exist. If $n=0$, then $A$ is full by definition of full $0$-spheres, so let $n>0$. The $n$-cell $a$ covers $\partial_n B$, so by the inductive hypothesis, its boundary is full and there exist cells $a',b'$ that cover ${\partial_{n-1}B} = $ such that
        \begin{align*}
                \src a & = \fc^w(s_{n-1}^{\partial_n B})(a'), &
                \tgt a & = \fc^w(t_{n-1}^{\partial_n B})(b').
        \end{align*}
        By the inductive hypothesis, we conclude that the boundary of $\pr_1A$ must be full, since
        \begin{align*}
                \src\,(\pr_1A) & = \fc^w(s_n^Bs_{n-1}^{\partial_n B})(a') = \fc^w(s_{n-1}^B)(a'), \\
                \tgt\,(\pr_1A) & = \fc^w(s_n^Bt_{n-1}^{\partial_n B})(b') = \fc^w(t_{n-1}^B)(b').
        \end{align*}
        The $n$-dimensional supports of $\pr_1 A$ and $\pr_2 B$ must be those of $\Free\,s_n^B$ and $\Free\,t_n^B$ respectively, hence the source and target boundary $n$-positions. Therefore, the sphere $A$ is full.

        Let now again $n\in\N$ arbitrary, $B$ a tree of dimension at most $n+1$ and $c\in \Cell_{n+1}(\Free\,B)$ a cover. It remains to show that there exist covers $a,b$ of ${\partial_n B}$ such that
        \begin{align*}
                \src\,c & = \fc^w(s_n^B)(a) & \tgt\,c & = \fc^w(t_n^B)(b).
        \end{align*}
        It follows from Lemmas \ref{immersion-lifting} and \ref{support-of-var-to-var} that existence of those cells is equivalent to the $k$-dimensional support of $\src\,c$ and $\tgt\,c$ being precisely those of $\Free\,s_n^B$ and $\Free\,t_n^B$ respectively for all $k\in \N$.

        We can easily see that it suffices for the $k$-dimensional support of $\src\,c$ to contain all source boundary $k$-positions for all $k\le n$, and similarly for $\tgt\,c$. By Proposition \ref{prop:boundary-positions-generation} and Lemma \ref{source-of-support}, if this is the case then the $k$\-dimensional support of $\src\,c$ contain those of $\Free\,s_n^B$ for all $k\in\N$. Since for $k<n$, the support of $\Free\,s_n^B$ contain all $k$-positions, it suffices to show that the $n$-dimensional support of $\src\,c$ do not contain any positions that are not source boundary. To see that consider the generic, generator-preserving factorisation
        \[\begin{tikzcd}
                        \mathbb{D}^n \arrow[r, "\gen_a", two heads] &
                        \Free\,{B_a} \ar["\fr_a", r, tail] &
                        \Free\,B
                \end{tikzcd}\]
        of the homomorphism corresponding to the cell $a$ and recall that $\dim B_a \le n$ and that $\fc_a$ must be an immersion. Since every position is parallel to a source boundary one, if the $n$-dimensional support of $\src\,c$ contained some non-source boundary $n$-position, they would contain two parallel $n$-positions. Then by the factorisation above, the Batanin tree $B_a$ would have two non-parallel top dimensional positions, which we can see inductively is impossible.

        We are left to show that the $k$-dimensional support of $\src\,c$ contain all source boundary $k$-positions for all $k\le n$ when $c\in \Cell_{n+1}(\Free\,B)$ is a cover, and a similar statement for $\tgt\,c$. We will instead prove the following stronger claim: If $c\in \Cell_{n+1}(\Free\,B)$ is an arbitrary cell and $p\in \partial^s_k(B)$ for $k\le n$ is in the support of $c$, then it is in the support of $\src\,c$. We will do so by induction on the cell $c$.

        If $c = \var\,p'$ is a generator, then its source and target are parallel generators as well so for $k<n$,
        \[\fv_k(c) = \fv_k(\src\,c) = \fv_k(\tgt\,c).\]
        Moreover, $\tgt\,p'$ can not be source boundary, so the only $n$-position in the support of $c$ that may be source boundary is $\src\,p' \in \fv_n(\src\,c)$. Therefore, the statement holds for generators.

        Suppose now that $c = \coh(B',A',\tau)$ is a coherence cell and fix $k\le n$. Then by definition of the support of $c$, we are left to show that for all $(l,q)\in \el(\Pos\,B')$,
        \[ \fv_k(\tau_{l,V}q) \cap \partial_k^s B \subseteq \fv_k(\src\,c). \]
        As the sphere $A'$ is full, the support of $\pr_1 A'$ must be those of $s_n^{B'}$, so by Lemma~\ref{source-of-support},
        \[  \fv_k(\src\,c) = \bigcup_{\substack{l\in\N \\ q\in \fv_l(\Free\,s_n^{B'})}} \fv_k(\tau_{l,V}q) \]
        Proposition \ref{prop:boundary-positions-generation} shows that $\fv_l(\Free\,s_n^{B'})$ contains all $l$-positions for $l<n$ and exactly the source boundary $n$-positions of $B'$. Hence, we are left to show the inclusion above when $l = n+1$, or when $l = n$ and $q$ is not source boundary.

        Let $q_0\in \Pos_n(B')$ not source boundary. Then there exists some position $q\in \Pos_{n+1}(B')$ with target $q_0$. Let $q_1\in \Pos_n(B')$ the source of $q$. By induction, we may assume that the claim holds for the cell $\tau_{V,l}(q)$, so
        \[  \partial_k^sB \cap \fv_k(\tau_{n,V}q_0) \subseteq \partial_k^sB \cap \fv_k(\tau_{n+1,V}q) \subseteq \cap \fv_k(\tau_{n,V}q_1). \]
        If $q_1$ is source boundary, then we are done. Otherwise, we may repeat this process to get a position $q_2$ and so forth. This process is guaranteed to terminate in finitely many steps, since the positions of a Batanin tree are well-ordered and $p_1<p_0$ \cite[Section~4]{weber_pra_2004}. Finally, let $q'\in\Pos_{n+1}(B')$, then
        \[  \partial_k^s B \cap \fv_k(\tau_{V,n+1}q') \subseteq \partial_k^s B \cap \fv_k(\tau_{V,n}(\src q)) \subseteq \fv_k(\src\,c).   \]
        This shows that the claim also holds for coherence cells.
\end{proof}

\begin{proposition}
        The globular theory $\Theta^w$ can be equipped with a contraction.
\end{proposition}
\begin{proof}
        The globular theory $\Theta^w$ is clearly normalised. We now equip $\Theta^w$ with a contraction. Suppose that $n >0$, and that $B$ is a Batanin tree with \mbox{$\dim B \leq n$}.
        Let $c, d : \mathbb{D}^{n-1}\to \Free\,{\partial B}$ parallel pair of covers.
        Then $\Free(s_n^B)(c)$ and $\Free(t_n^B)(d)$ constitute a full $(n-1)$-sphere $A$ of $B$. We define $l^{c, d} : D_{n} \to B$ to be the morphism corresponding to the cell $\coh(B, A, \id)$. This choice satisfies the required properties by construction.
\end{proof}

The following result now follows immediately from this definition together with the proof of Proposition \ref{generic-free}.

\begin{proposition}\label{pure-factorisation}
        Suppose that $X$ is a globular set. Suppose that $n > 0$, and that $\sigma : \mathbb{D}^n \to \Free X$  is a morphism of computads. Then either $\sigma$ is generator-preserving, or $\sigma$ can be uniquely written as a composite
        \[
                \begin{tikzcd}
                        \mathbb{D}^n \ar[r, "l^{c,d}"] & \Free\,{B'} \ar[r, "\sigma'"] & \Free X
                \end{tikzcd}
        \]
        for some unique covers $c, d$, some unique Batanin tree $B'$, and some unique morphism $\sigma'$ such that $\mdpth \sigma' < \mdpth \sigma$.
\end{proposition}

\begin{theorem}\label{theta-w-is-initial}
        The globular theory $\Theta_w$ with the above choice of contraction is the initial normalised homogeneous globular theory over $\Theta_s$ with contraction.
\end{theorem}
\begin{proof}
        Suppose that $i : \Theta_0 \to \Theta$ is a normalised homogeneous globular theory over $\Theta_s$ with contraction. We now define a functor $F : \Theta^w \to \Theta$. Since $\Free$ is injective on objects, we may define $F \Free B = i B$ for each object $B$ in $\Theta_w$. Similarly, since $\Free$ is faithful, for each immersion $\sigma : B \to B'$, we may define $F \Free \sigma = i \sigma$. Since every morphism in $\Theta_w$ factors into a cover followed by an immersion, it remains to define $F$ on covers.

        Suppose that $\sigma : B \to B'$ is a cover. If $\dim B = 0$, then we must have that $\sigma = \id_{D^0}$, since $\Theta_w$ is normalised. In this case $F \sigma = \id_{F D^0}$. Hence, suppose that $\dim B = n > 0$. We define $F \sigma$ by induction on $n$, and we simultaneously show that for each morphism $\tau : B' \to B''$, we have that $F \tau \circ F \sigma = F (\tau \circ \sigma)$. Suppose that we have defined $F$ on covers whose domain has dimension less than $n$. Then we define $F \sigma$ by induction on $\mdpth \sigma$. We simultaneously show that the assignment $F$ respects sources and targets:
        \[
                F \sigma  \circ I s_{n-1}^{B} = F (\sigma\circ \Free s_{n-1}^B).
        \]
        and similarly for targets.

        First suppose that $B = D^n$. When the morphism $\sigma$ is generator-preserving, we have already defined $F\sigma$. Sources and targets are trivially preserved in this case. Hence, suppose that $\sigma = \sigma' \circ l^{c, d}$ for some unique morphism $\sigma'$ with $\mdpth \sigma' < \mdpth \sigma$. Since $\sigma$ is a cover and $l^{c, d}$ is a cover, we have that $\sigma'$ is also a cover. Hence, we define
        \[
                F \sigma = F \sigma' \circ l^{Fc, Fd}.
        \]
        This is well defined by the uniqueness part of Proposition \ref{pure-factorisation}. Consider the following commutative diagram:
        \[
                \begin{tikzcd}[column sep=huge]
                        I D^{n-1}
                        \ar[r, "F c", two heads]
                        \ar[d, "I s^{D_n}_{n-1}" left, tail]
                        &
                        I \partial B''
                        \ar[dr, "F (\sigma' \circ \Free s^{\partial B''}_{n-1})", tail]
                        \ar[d, "I s^{B''}_{n-1}" left, tail]
                        \\
                        I D^n
                        \ar[r, "l^{F c,F d}" below, two heads]
                        &
                        I B''
                        \ar[r, "F \sigma'" below, tail]
                        &
                        I B'
                \end{tikzcd}
        \]
        The left hand square commutes by definition of $l^{Fc, Fd}$. The right hand triangle commutes by inductive hypothesis. The composite of the bottom row is $F \sigma$ by definition. Since, $n-1 < n$, the inductive hypothesis implies that the composite of the top row is $F(c \circ \sigma' \circ \Free s^{\partial B''}_{n-1})$. However,
        \[
                F(c \circ \sigma' \circ \Free s^{\partial B''}_{n-1})
                =
                F(c \circ \Free s^{B''}_{n-1} \circ \sigma')
                =
                F(I s^{D_n}_{n-1} \circ l^{c, d} \circ \sigma').
        \]
        Thus, $F$ respects sources. A similar argument shows that $F$ respects targets.

        Now suppose that $B$ is an arbitrary Batanin tree with $\dim B = n$. Suppose that $\sigma : \Free B \to \Free B'$ is a cover. For each $(k, x) \in \el(B)$, we have
        \[
                \mdpth (\gen_{ \sigma \circ \Free x}) = \mdpth (\sigma \circ \Free x) < \mdpth \sigma.
        \]
        The equality follows from the fact that $\fr_{\cov_{ \sigma \circ \Free x}}$ is generator-preserving. The inequality follows from the definition of $\mdpth$, and the Yoneda Lemma. Hence, by the representable case, we can assume that we have defined
        \[
                F (\sigma \circ \Free x)
                =
                F \inc_{\sigma \circ \Free x} \circ F \cov_{\sigma \circ \Free x}
        \]
        However, $B = \colim_{(k, x) \in \el(B)} D^k$ is a canonical globular sum of representables. It follows that $i B = \colim_{(k, x) \in \el(B)} i D^k$. Hence, since $F$ respects sources and targets, we may define $F \sigma : i B \to i B'$ to be the unique arrow in $\Theta$ such that, for each $(k, x) \in \el(B)$, we have that $F \sigma \circ i x = F (\sigma \circ \Free x)$.

        Now suppose that $\tau : B' \to B''$. A straightforward induction on $\mdpth \sigma$ now implies that $F \tau \circ F \sigma = F(\tau \circ \sigma)$. This completes the inductive definition of $F$.


        In order to show that $F$ is functorial, suppose that $f : A \to B$ and $g : B \to C$ are arrows in $\Theta_w$. If both $f$ and $g$ are covers, or both $f$ and $g$ are immersions, or if $f$ is a cover and $g$ is an immersion, then functoriality follows from the definition of $F$. By homogeneity it now suffices to check the case where $f$ is an immersion and $g$ is a cover. Note that, by definition, we have that $F (g \circ f) = F \inc_{g \circ f} \circ F \cov_{g \circ f}$. Now suppose that $x : ID^n \to A$ is an $n$-cell of $A$. Let $j$ and $d$ be an immersion and a cover respectively such that $\cov_{g \circ f} \circ x = j d$. Consider the following commutative diagram:

        \begin{equation*}
                \begin{tikzcd}
                        I D^n
                        \ar[r, "d", two heads]
                        \ar[d, "x"', tail]
                        &
                        \bullet
                        \ar[d, "j", tail]
                        \\
                        A
                        \ar[r, two heads, "\cov_{g \circ f}"]
                        \ar[d, "f"', tail]
                        &
                        \bullet
                        \ar[d, tail, "\inc_{g \circ f}"]
                        &
                        \\
                        B
                        \ar[r, "g"', two heads]
                        &
                        C
                \end{tikzcd}
        \end{equation*}
        Since $g$ is a cover, we have that $Fg \circ F(f \circ x) = F(\inc_{g \circ f} \circ j) \circ Fd$. Hence, since $F$ is functorial on immersions, we have that
        \[
                Fg \circ Ff \circ Fx = F \inc_{g \circ f} \circ Fj \circ Fd.
        \]
        On the other hand, since $\cov_{g \circ f}$ is a cover, we have that $F \cov_{g \circ f} \circ Fx = Fj \circ F d$. Precomposing with $F \inc_{g \circ f}$, we obtain
        \begin{align*}
                F(g \circ f) \circ Fx
                = F \inc_{g \circ f} \circ F \cov_{g \circ f} \circ Fx
                = F \inc_{g \circ f} \circ Fj \circ Fd.
        \end{align*}
        Hence, for all cells $x : ID^n \to A$, we have that
        \[
                F(g \circ f) \circ Fx = Fg \circ Ff \circ Fx.
        \]
        Since maps of the form $Fx : ID^n \to FA$ assemble into a colimit cone over $A$, it follows that $F(g \circ f) = Fg \circ Ff$.

        By construction, the functor $F$ is over $\Theta_s$ and preserves immersions, covers and globular sums of immersions. Conversely, any functor over $\Theta_s$ preserving these data must satisfy all the properties defining $F$. Thus, $\Theta^w$ is initial.
\end{proof}

In conclusion, we establish the desired comparison result.

\begin{corollary}
        The monad $\fc^w$ is the one induced by Leinster's operad. In particular, our notion of $\omega$-category coincides with that of Leinster \cite{leinster_higher_2003}.
\end{corollary}
\begin{proof}
        Let $T$ be the monad on globular sets induced by Leinster's operad. By Theorem \ref{theta-w-is-initial}, we have that $\Theta_{\fc^w} = \Theta_w \cong \Theta_T$ as globular theories.  Therefore, by the equivalence described in \cite[6.6.8]{ara_sur_2010} between homogeneous globular theories and globular operads, viewed as cartesian monads over the strict $\omega$\-category monad $\fc^s$, the monad $\fc^w$ must be isomorphic to $T$. Since Leinster's $\omega$-categories are $T$-algebras, the two notions of $\omega$-category coincide.
\end{proof}

\bibliographystyle{plainurl}
\bibliography{Bibliography}

\end{document}